\documentclass[11pt,reqno]{amsart}
\usepackage[dvipdfm,a4paper,top=30mm,bottom=25mm,left=25mm,right=25mm]{geometry}
\usepackage{amsmath,amssymb,amsthm,mathtools}
\usepackage{comment,mathrsfs,bm}

\usepackage[colorlinks=true,linkcolor=blue,citecolor=blue]{hyperref}

\usepackage{graphicx}

\newtheorem{theorem}{Theorem}[section]
\newtheorem{lemma}[theorem]{Lemma}
\newtheorem{proposition}[theorem]{Proposition}

\theoremstyle{definition}
\newtheorem{definition}[theorem]{Definition}

\theoremstyle{remark}

\numberwithin{equation}{section}

\allowdisplaybreaks[4]

\def\rnum#1{\expandafter{\romannumeral #1}} 
\def\Rnum#1{\uppercase\expandafter{\romannumeral #1}}

\def\~#1{\widetilde #1}

\def\({\left(}
\def\){\right)}
\def\<{\langle}
\def\>{\rangle}

\begin{document}

\title[nonlinear Klein-Gordon equation with an inverse-square potential]{Stability and instability of radial standing waves to NLKG equation with an inverse-square potential}


\author{Masaru Hamano}
\address{Department of Mathematics, Graduate School of Science and Engineering Saitama University, Shimo-Okubo 255, Sakura-ku, Saitama-shi, Saitama 338-8570, Japan}
\email{ess70116@mail.saitama-u.ac.jp}

\author{Masahiro Ikeda}
\address{Department of Mathematics, Faculty of Science and Technology, Keio University, 3-14-1 Hiyoshi, Kohoku-ku, Yokohama, 223-8522, Japan/Center for Advanced Intelligence Project, Riken, Japan}
\email{masahiro.ikeda@riken.jp / masahiro.ikeda@keio.jp}







\begin{abstract}
In this paper, we consider radial standing waves to a nonlinear Klein-Gordon equation with a repulsive inverse-square potential.
It is known that existence of a ``radial'' ground state to the stationary problem of the nonlinear Klein-Gordon equation.
Here, the ``radial'' ground state is a solution with the least energy among radial solutions to the stationary problem.
We deal with stability and instability of the standing wave for the ``radial'' ground state.
\end{abstract}

\maketitle

\tableofcontents


\section{Introduction}

\subsection{Nonlinear Klein-Gordon equation}

We consider the following nonlinear Klein-Gordon equation with the inverse-square potential:
\begin{equation}
	- \partial_t^2 u + \Delta_\gamma u - u
		= - |u|^{p-1}u,\quad (t,x) \in [0,\infty) \times \mathbb{R}^d, \tag{NLKG$_\gamma$} \label{NLKG}
\end{equation}
where $1 < p \leq 1 + \frac{4}{d-2}$, $d \geq 3$, $\Delta_\gamma = \Delta - \frac{\gamma}{|x|^2}$, $\gamma > - (\frac{d-2}{2})^2$, and $\gamma \neq 0$.
An unknown function $u(t,x) : [0,\infty) \times \mathbb{R}^d \longrightarrow \mathbb{C}$ is a solution to \eqref{NLKG}.
In particular, we deal with the Cauchy problem of \eqref{NLKG} with initial data
\begin{equation}
	(u(0,x),\partial_tu(0,x)) = (u_0,u_1), \quad x \in \mathbb{R}^d. \tag{IC} \label{IC}
\end{equation}
Let $- \Delta_\gamma^0$ be the natural action of $- \Delta + \frac{\gamma}{|x|^2}$ on $C_c^\infty(\mathbb{R}^d \setminus \{0\})$, where $C_c^\infty(\mathbb{R}^d \setminus \{0\})$ is a set of functions with a compact support on $\mathbb{R}^d \setminus \{0\}$.
For $\gamma > - (\frac{d-2}{2})^2$, $- \Delta_\gamma^0$ is a positive semi-definite symmetric operator.
Indeed, we have
\begin{align*}
	\<- \Delta_\gamma^0 f,f\>_{L^2}
		= \int_{\mathbb{R}^d}|\nabla f(x)|^2 + \frac{\gamma}{|x|^2}|f(x)|^2dx
		= \int_{\mathbb{R}^d}\left|\nabla f(x) + \frac{\rho}{|x|^2}xf(x)\right|^2dx
\end{align*}
for any $f \in C_c^\infty(\mathbb{R}^d \setminus \{0\})$, where
\begin{align}
	\rho
		:= \frac{d-2}{2} - \left\{\left(\frac{d-2}{2}\right)^2 + \gamma\right\}^\frac{1}{2}. \label{141}
\end{align}
We consider the self-adjoint extension of $- \Delta_\gamma^0$.
It is known that the extension is not unique in the range $- (\frac{d-2}{2})^2 < \gamma < - (\frac{d-2}{2})^2 + 1$ (see \cite{KalSchWalWus75}).
In this case, we choose $- \Delta_\gamma$ as Friedrichs extension among possible extensions.
Since $- \Delta_\gamma$ and $1 - \Delta_\gamma$ are non-negative operators, $(- \Delta_\gamma)^\frac{1}{2}$ and $(1 - \Delta_\gamma)^\frac{1}{2}$ are well-defined on the domain
\begin{align*}
	H_\gamma^1(\mathbb{R}^d)
		:= \{f \in H^1(\mathbb{R}^d) : \|f\|_{H_\gamma^1} < \infty\}
\end{align*}
with the norm $\|f\|_{H_\gamma^1}^2 := \|(1 - \Delta_\gamma)^\frac{1}{2}f\|_{L^2}^2 = \|f\|_{L^2}^2 + \|(- \Delta_\gamma)^\frac{1}{2}f\|_{L^2}^2$.
We note
\begin{align*}
	\|(- \Delta_\gamma)^\frac{1}{2}f\|_{L^2}^2
		= \|\nabla f\|_{L^2}^2 + \int_{\mathbb{R}^d}\frac{\gamma}{|x|^2}|f(x)|^2dx.
\end{align*}
Here, we define solutions to \eqref{NLKG} with \eqref{IC} clearly.

\begin{definition}[Solution]
Let $I \subset [0,\infty)$ be a nonempty time interval.
We say that a function $u : I \times \mathbb{R}^d \longrightarrow \mathbb{C}$ is a solution to \eqref{NLKG} with \eqref{IC} on $I$ if $(u,\partial_tu) \in (C_t \cap L_{t,\text{loc}}^\infty)(I;H_x^1(\mathbb{R}^d) \times L_x^2(\mathbb{R}^d))$, there exists $J \in \mathbb{N} \cup \{0\}$ such that $u \in \cap_{j=1}^J L_{t,\text{loc}}^{q_j}(I;W_x^{s,r_j}(\mathbb{R}^d))$ (, where we ignore this condition when $J = 0$), and the Duhamel formula
\begin{align*}
	u(t)
		= \partial_tK_\gamma(t)u_0(x) + K_\gamma(t)u_1(x) - \int_0^tK_\gamma(t-s)(|u|^{p-1}u)(s)ds
\end{align*}
holds, where $(q_j,r_j)$ is an admissible pair and belongs to $\Lambda_{1-s,0,\frac{(d-2)^2}{4}}$ with $0 \leq s \leq 1$ (see Definition \ref{Admissible pair} for the definition) and the operators $K_\gamma(t)$ and $\partial_tK_\gamma(t)$ for each $t \geq 0$ are defined as
\begin{align*}
	K_\gamma(t)
		:= \frac{\sin\{(1-\Delta_\gamma)^\frac{1}{2}t\}}{(1-\Delta_\gamma)^\frac{1}{2}},\ \ \ \ 
	\partial_tK_\gamma(t)
		:= \cos\{(1-\Delta_\gamma)^\frac{1}{2}t\}.
\end{align*}
\end{definition}

It is proved for $\gamma > 0$ in \cite{Din18, KilMurVisZhe17} that \eqref{NLKG} has the non-trivial standing wave $u(t,x) = e^{i\omega t}Q_{\omega,\gamma}(x)$ for $1 - \omega^2 > 0$, where $Q_{\omega,\gamma}$ satisfies the following elliptic equation:
\begin{align}
	- (1 - \omega^2)Q_{\omega,\gamma} + \Delta_\gamma Q_{\omega,\gamma}
		= - |Q_{\omega,\gamma}|^{p-1}Q_{\omega,\gamma}. \label{SP} \tag{SP$_{\omega,\gamma}$}
\end{align}
In particular, \eqref{SP} has a ``radial'' ground state, where a set $\mathcal{G}_{\omega,\gamma,\text{rad}}$ of a whole of ``radial'' ground state is defined as
\begin{align*}
	\mathcal{G}_{\omega,\gamma,\text{rad}}
		& := \{\phi \in \mathcal{A}_{\omega,\gamma,\text{rad}} : S_{\omega,\gamma}(\phi) \leq S_{\omega,\gamma}(\psi)\text{ for any }\psi \in \mathcal{A}_{\omega,\gamma,\text{rad}}\}, \\
	\mathcal{A}_{\omega,\gamma,\text{rad}}
		& := \{\phi \in H_\text{rad}^1(\mathbb{R}^d) \setminus \{0\} : S_{\omega,\gamma}'(\phi) = 0\}, \\
	S_{\omega,\gamma}(\phi)
		& := \frac{1-\omega^2}{2}\|\phi\|_{L^2}^2 + \frac{1}{2}\|(-\Delta_\gamma)^\frac{1}{2}\phi\|_{L^2}^2 - \frac{1}{p+1}\|\phi\|_{L^{p+1}}^{p+1}.
\end{align*}
Unlike the case $\gamma > 0$, it is known for $- (\frac{d-2}{2})^2 < \gamma < 0$ in \cite{Din18, KilMurVisZhe17} and $\gamma = 0$ in \cite{BerLio83, Str77} that \eqref{SP} has the ground state in the usual sense, where a set $\mathcal{G}_{\omega,\gamma}$ of a whole of ground state is defined as
\begin{align*}
	\mathcal{G}_{\omega,\gamma}
		& := \{\phi \in \mathcal{A}_{\omega,\gamma} : S_{\omega,\gamma}(\phi) \leq S_{\omega,\gamma}(\psi)\text{ for any }\psi \in \mathcal{A}_{\omega,\gamma}\}, \\
	\mathcal{A}_{\omega,\gamma}
		& := \{\phi \in H^1(\mathbb{R}^d) \setminus \{0\} : S_{\omega,\gamma}'(\phi) = 0\}.
\end{align*}
When $\gamma = 0$, the ground state is characterized by a functional $K_{\omega,\gamma}^{\alpha,\beta}$ with suitable exponent $(\alpha,\beta)$ as follows: If $Q_{\omega,0}$ is the ground state to \eqref{SP} with $\gamma = 0$, then $Q_{\omega,0} \in \mathcal{M}_{\omega,\gamma}^{\alpha,\beta}$ holds for some $(\alpha,\beta)$, where
\begin{align*}
	K_{\omega,\gamma}^{\alpha,\beta}(f)
		& := \left.\frac{d}{d\lambda}\right|_{\lambda = 0}S_{\omega,\gamma}(e^{\alpha\lambda}f(e^{\beta\lambda}\,\cdot\,)), \\
	\mathcal{M}_{\omega,\gamma}^{\alpha,\beta}
		& := \{\phi \in H^1(\mathbb{R}^d) \setminus \{0\} : S_{\omega,\gamma}(\phi) = n_{\omega,\gamma}^{\alpha,\beta},\,K_{\omega,\gamma}^{\alpha,\beta}(\phi) = 0\}, \\
	n_{\omega,\gamma}^{\alpha,\beta}
		& := \inf\{S_{\omega,\gamma}(\phi) : \phi \in H^1(\mathbb{R}^d) \setminus \{0\},\,K_{\omega,\gamma}^{\alpha,\,\beta}(\phi) = 0\}.
\end{align*}
By using the characterization, Shatah \cite{Sha83} showed stability of the set $\mathcal{G}_{\omega,0}$ and Ohta--Todorova \cite{OhtTod07} proved its very strong instability (see also \cite{BerCaz81, LiuOhtTod07, OhtTod05, Sha85, ShaStr85}), where very strong instability of $\mathcal{G}_{\omega,\gamma}$ and stability of $\mathcal{G}_{\omega,\gamma}$ are defined respectively as suitable rewriting of Definition \ref{Definition of stability} below.

\begin{theorem}[Stability versus very strong instability with $\gamma = 0$]\label{Stability of instability of standing waves with gamma=0}
Let $d \geq 2$, $1 < p < \infty$ if $d = 2$, $1 < p < 1 + \frac{4}{d-2}$ if $d \geq 3$, $\gamma = 0$, $1 - \omega^2 > 0$, and $\omega_c := \sqrt{\frac{p-1}{4 - (d-1)(p-1)}}$.
\begin{itemize}
\item (Stability of $\mathcal{G}_{\omega,0}$, \cite{Sha83})
Let $d \geq 3$, $1 < p < 1 + \frac{4}{d}$, and $\omega_c < |\omega| < 1$.
Then, $\mathcal{G}_{\omega,0}$ is stable.
\item (Very strong instability of $\mathcal{G}_{\omega,0}$, \cite{OhtTod07})
We assume $|\omega| \leq \omega_c$ if $1 < p < 1 + \frac{4}{d}$.
Then, $\mathcal{G}_{\omega,0}$ is very strongly unstable.
\item (Very strong instability of $\mathcal{A}_{\omega,0}$, \cite{OhtTod07})
Let $1 < p < 1 + \frac{4}{d}$ and $|\omega| = \omega_c$.
Then, $\mathcal{A}_{\omega,0}$ is very strong unstable.
\end{itemize}
\end{theorem}

In this paper, we consider a similar problem to Theorem \ref{Stability of instability of standing waves with gamma=0} for \eqref{NLKG} with $\gamma > 0$.
We note that $n_{\omega,\gamma}^{\alpha,\beta}$ does not have a minimizer for $\gamma > 0$ since we can not use the Schwarz symmetrization.
Therefore, we do not know \eqref{SP} has the ground state or not at the moment and hence, we restrict to radial functions.
That is, we use the minimization problem
\begin{align}
	r_{\omega,\gamma}^{\alpha,\,\beta}
		:= \inf\{S_{\omega,\gamma}(\phi) : \phi \in H_\text{rad}^1(\mathbb{R}^d) \setminus \{0\},\,K_{\omega,\gamma}^{\alpha,\,\beta}(\phi) = 0\} \label{146}
\end{align}
and prove that $\mathcal{G}_{\omega,\gamma,\,\text{rad}} \subset \mathcal{M}_{\omega,\gamma,\,\text{rad}}^{\alpha,\,\beta}$ for suitable $(\alpha,\beta)$, where
\begin{align*}
	\mathcal{M}_{\omega,\gamma,\,\text{rad}}^{\alpha,\,\beta}
		:= \{\phi \in H_\text{rad}^1(\mathbb{R}^d) \setminus \{0\} : S_{\omega,\gamma}(\phi) = r_{\omega,\gamma}^{\alpha,\,\beta},\,K_{\omega,\gamma}^{\alpha,\,\beta}(\phi) = 0\}.
\end{align*}
We define stability, instability, strong instability, and very strong instability of $\mathcal{G}_{\omega,\gamma,\text{rad}}$ and $\mathcal{A}_{\omega,\gamma,\text{rad}}$.

\begin{definition}\label{Definition of stability}
Let $\mathcal{H}_{\omega,\gamma,\text{rad}}$ denote $\mathcal{G}_{\omega,\gamma,\text{rad}}$ or $\mathcal{A}_{\omega,\gamma,\text{rad}}$.
\begin{itemize}
\item (Stability)
We say that $\mathcal{H}_{\omega,\gamma,\text{rad}}$ is stable in $H_\text{rad}^1(\mathbb{R}^d) \times L_\text{rad}^2(\mathbb{R}^d)$ when for any $\varepsilon > 0$, there exists $\delta > 0$ such that if $(u_0,u_1) \in H_\text{rad}^1(\mathbb{R}^d) \times L_\text{rad}^2(\mathbb{R}^d)$ satisfies
\begin{align}
	\inf_{Q_{\omega,\gamma} \in \mathcal{H}_{\omega,\gamma,\text{rad}}}\|(u_0,u_1) - (Q_{\omega,\gamma},i\omega Q_{\omega,\gamma})\|_{H^1 \times L^2}
		< \delta, \label{127}
\end{align}
then the solution $u$ to \eqref{NLKG} with initial data $(u_0,u_1)$ exists globally in time and satisfies
\begin{align*}
	\sup_{t \in [0,\infty)} \inf_{Q_{\omega,\gamma} \in \mathcal{H}_{\omega,\gamma,\text{rad}}} \|(u(t),\partial_tu(t)) - (Q_{\omega,\gamma},i\omega Q_{\omega,\gamma})\|_{H_x^1 \times L_x^2}
		< \varepsilon.
\end{align*}
\item (Instability)
If $\mathcal{H}_{\omega,\gamma,\text{rad}}$ is not stable in $H_\text{rad}^1(\mathbb{R}^d) \times L_\text{rad}^2(\mathbb{R}^d)$, then we say that $\mathcal{H}_{\omega,\gamma,\text{rad}}$ is unstable in $H_\text{rad}^1(\mathbb{R}^d) \times L_\text{rad}^2(\mathbb{R}^d)$.
\item (Strong instability)
We say that $\mathcal{H}_{\omega,\gamma,\text{rad}}$ is strongly unstable in $H_\text{rad}^1(\mathbb{R}^d) \times L_\text{rad}^2(\mathbb{R}^d)$ if for any $\delta > 0$, there exists $(u_0,u_1) \in H_\text{rad}^1(\mathbb{R}^d) \times L_\text{rad}^2(\mathbb{R}^d)$ such that $(u_0,u_1)$ satisfies \eqref{127} and the solution $u$ to \eqref{NLKG} with initial data $(u_0,u_1)$ blows up in finite time or ``exists globally in time and satisfies $\limsup_{t \rightarrow \infty}\|(u(t),\partial_tu(t))\|_{H_x^1 \times L_x^2} = \infty$''.
\item (Very strong instability)
We say that $\mathcal{H}_{\omega,\gamma,\text{rad}}$ is very strongly unstable in $H_\text{rad}^1(\mathbb{R}^d) \times L_\text{rad}^2(\mathbb{R}^d)$ if for any $\delta > 0$, there exists $(u_0,u_1) \in H_\text{rad}^1(\mathbb{R}^d) \times L_\text{rad}^2(\mathbb{R}^d)$ such that $(u_0,u_1)$ satisfies \eqref{127} and the solution $u$ to \eqref{NLKG} with initial data $(u_0,u_1)$ blows up in finite time.
\end{itemize}
\end{definition}

From the definitions, we see that if $\mathcal{H}_{\omega,\gamma,\text{rad}}$ is unstable, then $\mathcal{H}_{\omega,\gamma,\text{rad}}$ is strongly unstable and if $\mathcal{H}_{\omega,\gamma,\text{rad}}$ is strongly unstable, then $\mathcal{H}_{\omega,\gamma,\text{rad}}$ is very strongly unstable.
It also follows that if $\mathcal{H}_{\omega,\gamma,\text{rad}}$ is (very) strongly unstable in $H_\text{rad}^1(\mathbb{R}^d) \times L_\text{rad}^2(\mathbb{R}^d)$, then $\mathcal{H}_{\omega,\gamma,\text{rad}}$ is (very) strongly unstable in $H^1(\mathbb{R}^d) \times L^2(\mathbb{R}^d)$.

\subsection{Main result}

First, we state the following local well-posedness result of \eqref{NLKG}.

\begin{theorem}[Local well-posedness of \eqref{NLKG}]\label{Local well-posedness}
Let $d \geq 3$, $p \in (1,1 + \frac{2}{d-2}] \cup [1 + \frac{4}{d+1},1 + \frac{4}{d-2}]$, and $\gamma > 0$.
Then, for any $(u_0,u_1) \in H^1(\mathbb{R}^d) \times L^2(\mathbb{R}^d)$, there exists $T_\text{max} \in (0,\infty]$ such that \eqref{NLKG} with \eqref{IC} has a unique solution $(u,\partial_tu) \in C_t([0,T_\text{max}) ; H_x^1(\mathbb{R}^d) \times L_x^2(\mathbb{R}^d))$.
For each compact interval $I \subset [0,T_\text{max})$, the mapping $H^1(\mathbb{R}^d) \times L^2(\mathbb{R}^d) \ni (u_0,u_1) \mapsto (u,\partial_tu) \in C_t(I;H_x^1(\mathbb{R}^d) \times L_x^2(\mathbb{R}^d))$ is continuous.
Moreover, the solution $(u,\partial_tu)$ has the following blow-up alternative for $p \neq 1 + \frac{4}{d-2}$:
If $T_\text{max} < \infty$, then
\begin{align*}
	\lim_{t \nearrow T_\text{max}}\|(u(t),\partial_tu(t))\|_{H_x^1 \times L_x^2}
		= \infty.
\end{align*}
Furthermore, the solution $u$ preserves its charge, energy, and momentum with respect to time $t$, where they are defined as follows:
\begin{align*}
	\text{(Charge) }&\ \ C[u(t),\partial_tu(t)]
		:= \text{Im}\int_{\mathbb{R}^d} u(t,x)\overline{\partial_tu(t,x)}dx, \\
	\text{(Energy) }&\ \ E_\gamma[u(t),\partial_tu(t)]
		:= \frac{1}{2}\|(-\Delta_\gamma)^\frac{1}{2}u(t)\|_{L_x^2}^2 + \frac{1}{2}\|u(t)\|_{L_x^2}^2 \\
		& \hspace{5.0cm} - \frac{1}{p+1}\|u(t)\|_{L_x^{p+1}}^{p+1} + \frac{1}{2}\|\partial_tu(t)\|_{L_x^2}^2, \\
	\text{(Momentum) }&\ \ M[u(t),\partial_tu(t)]
		:= \text{Re}\int_{\mathbb{R}^d}\nabla u(t,x)\overline{\partial_tu(t,x)}dx.
\end{align*}
\end{theorem}

The next theorem is a main result in this paper.

\begin{theorem}[Stability versus very strong instability]\label{Instability versus stability}
Let $d \geq 3$, $1 < p < 1 + \frac{4}{d-2}$, $\gamma > 0$, $1 - \omega^2 > 0$, and $\omega_c := \sqrt{\frac{p-1}{4 - (d-1)(p-1)}}$.
We assume that \eqref{NLKG} with \eqref{IC} is locally well-posed for $1 + \frac{2}{d-2} < p < 1 + \frac{4}{d+1}$.
\begin{itemize}
\item (Stability of $\mathcal{G}_{\omega,\gamma,\text{rad}}$)
Let $1 < p < 1 + \frac{4}{d}$ and $\omega_c < |\omega| < 1$.
Then, $\mathcal{G}_{\omega,\gamma,\text{rad}}$ is stable.
\item (Very strong instability of $\mathcal{G}_{\omega,\gamma,\text{rad}}$)
We assume $|\omega| \leq \omega_c$ if $1 < p < 1 + \frac{4}{d}$.
Then, $\mathcal{G}_{\omega,\gamma,\text{rad}}$ is very strongly unstable.
\item (Very strong instability of $\mathcal{A}_{\omega,\gamma,\text{rad}}$)
Let $1 < p < 1 + \frac{4}{d}$ and $|\omega| = \omega_c$.
Then, $\mathcal{A}_{\omega,\gamma,\text{rad}}$ is very strong unstable.
\end{itemize}
\end{theorem}

We note that we assume local well-posedness of \eqref{NLKG} with \eqref{IC} for $1 + \frac{2}{d-2} < p < 1 + \frac{4}{d+1}$.
However, this restriction arises case only $d \geq 6$ and hence, we can get Theorem \ref{Instability versus stability} without assumption of local well-posedness for $3 \leq d \leq 5$.

The following theorem is useful to get very strong instability of $\mathcal{G}_{\omega,\gamma,\text{rad}}$ (or $\mathcal{A}_{\omega,\gamma,\text{rad}}$) for $1 < p < 1 + \frac{4}{d-1}$.

\begin{theorem}[Boundedness of global solutions]\label{Global implies boundedness}
Let $d \geq 3$, $1 < p < 1 + \frac{4}{d-1}$, and $\gamma > - (\frac{d-2}{2})^2$.
If $(u,\partial_tu) \in C_t([0,\infty) ; H_x^1(\mathbb{R}^d) \times L_x^2(\mathbb{R}^d))$ is a global solution to \eqref{NLKG}, then
\begin{align*}
	\sup_{t \geq 0}\|(u(t),\partial_tu(t))\|_{H_x^1 \times L_x^2}
		< \infty
\end{align*}
holds.
\end{theorem}

We can find Theorem \ref{Global implies boundedness} with $\gamma = 0$ in \cite{OhtTod07}.

\subsection{Organization of the paper}

The organization of the rest of this paper is as follows:
In Section \ref{Preliminary}, we prepare some notations and tools.
In Section \ref{Sec:Local well-posedness}, we prove local well-posedness of \eqref{NLKG}.
In Section \ref{Boundedness of global solutions}, we show that global existence of solutions to \eqref{NLKG} implies uniform boundedness of the solutions (Theorem \ref{Global implies boundedness}).
In Section \ref{Existence of a minimizer}, we get a minimizer to $r_{\omega,\gamma}^{\alpha,\beta}$.
In Section \ref{Characterization of the ground state}, we characterize a minimizer to $r_{\omega,\gamma}^{\alpha,\beta}$ by the virial functional $K_{\omega,\gamma}^{\alpha,\beta}$.
In Section \ref{Proof of stability}, we prove the stability result in Theorem \ref{Instability versus stability}.
In Section \ref{Proof of very strong instability}, we prove the very strong instability results in Theorem \ref{Instability versus stability}.

\section{Preliminary}\label{Preliminary}

In this section, we define some notations and collect some tools. 

\subsection{Notations and definitions}

For nonnegative $X$ and $Y$, we write $X\lesssim Y$ to denote $X\leq CY$ for some $C > 0$.
If $X \lesssim Y \lesssim X$ holds, we write $X \sim Y$.

For $1 \leq p \leq \infty$, $L^p(\mathbb{R}^d)$ denotes the usual Lebesgue space.
$H^1(\mathbb{R}^d)$ denotes the usual Sobolev space.
We note that $H^1(\mathbb{R}^d)$ is a real Hilbert space with an inner product:
\begin{align*}
	\langle f, g \rangle_{H^1}
		= \langle f, g\rangle_{L^2} + \langle \nabla f, \nabla g\rangle_{L^2}
		:= \text{Re}\int_{\mathbb{R}^d} (f(x)\overline{g(x)} + \nabla f(x)\cdot \overline{\nabla g(x)})dx.
\end{align*}
For a Banach space $X$, we use $L^q(I;X)$ to denote the Banach space of functions $f : I \times \mathbb{R}^d \longrightarrow \mathbb{C}$ whose norm is $\|f\|_{L_t^q(I;X)} := \|\|f(t)\|_X\|_{L_t^q(I)} < \infty$.
If a time interval is not specified, then the $t$-norm is taken over $[0,\infty)$, that is, $\|f\|_{L_t^qX} := \|f\|_{L_t^q(0,\infty;X)}$.
We define the Fourier transform $\mathcal{F}$ and the inverse Fourier transform $\mathcal{F}^{-1}$ on $\mathbb{R}^d$ as
\begin{align*}
	\mathcal{F}f(\xi)
		:=\int_{\mathbb{R}^d}e^{-2\pi ix\cdot\xi}f(x)dx\ \ \text{ and }\ \ 
	\mathcal{F}^{-1}f(x)
		:=\int_{\mathbb{R}^d}e^{2\pi ix\cdot\xi}f(\xi)d\xi,
\end{align*}
where $x\cdot\xi$ denotes the usual inner product of $x$ and $\xi$ on $\mathbb{R}^d$.
$W^{s,p}(\mathbb{R}^d) = (1 - \Delta)^{-\frac{s}{2}}L^p(\mathbb{R}^d)$ is the inhomogeneous Sobolev space and $\dot{W}^{s,p}(\mathbb{R}^d) = (-\Delta)^{-\frac{s}{2}}L^p(\mathbb{R}^d)$ is the homogeneous Sobolev space for $s \in \mathbb{R}$ and $p \in [1,\infty]$, where $(1 - \Delta)^\frac{s}{2} = \mathcal{F}^{-1}(1 + 4\pi^2|\xi|^2)^\frac{s}{2}\mathcal{F}$ and $(- \Delta)^\frac{s}{2} = |\nabla|^s := \mathcal{F}^{-1}(2\pi|\xi|)^s\mathcal{F}$.
When $p = 2$, we express $W^{s,2}(\mathbb{R}^d) = H^s(\mathbb{R}^d)$ and $\dot{W}^{s,2}(\mathbb{R}^d) = \dot{H}^s(\mathbb{R}^d)$.
We also define a Sobolev space with a potential as $W_\gamma^{s,p}(\mathbb{R}^d) = (1 - \Delta_\gamma)^{-\frac{s}{2}}L^p(\mathbb{R}^d)$ and $\dot{W}_\gamma^{s,p}(\mathbb{R}^d) = (- \Delta_\gamma)^{-\frac{s}{2}}L^p(\mathbb{R}^d)$ for $s \in \mathbb{R}$ and $p \in [1,\infty]$.
When $p = 2$, we express $W_\gamma^{s,2}(\mathbb{R}^d) = H_\gamma^s(\mathbb{R}^d)$ and $\dot{W}_\gamma^{s,2}(\mathbb{R}^d) = \dot{H}_\gamma^s(\mathbb{R}^d)$.
For a space $X$, we define a space $X_\text{rad} := \{f \in X : f\text{ is a radial function.}\}$.

We define an operator $\mathcal{D}^{\alpha,\beta}$ as
\begin{align*}
	\mathcal{D}^{\alpha,\beta}\phi
		:= \left.\frac{\partial}{\partial \lambda}\right|_{\lambda = 0}e^{\alpha\lambda}\phi(e^{\beta\lambda}\,\cdot\,)\ \ \text{ as }\ \ 
	\mathcal{D}^{\alpha,\beta}\mathscr{F}(\phi)
		:= \left.\frac{\partial}{\partial \lambda}\right|_{\lambda = 0}\mathscr{F}(e^{\alpha\lambda}\phi(e^{\beta\lambda}\,\cdot\,))
\end{align*}
for any function $\phi$, any functional $\mathscr{F}$, and the parameter $(\alpha,\beta)$ satisfying
\begin{align}
	\alpha
		\geq 0,\ \ 
	2\alpha - d\beta
		\geq 0,\ \ 
	2\alpha - (d-2)\beta
		> 0,\ \ 
	(p-1)\alpha - 2\beta
		\geq 0,\ \text{ and }\ 
	(\alpha,\beta) \neq (0,0). \label{101}
\end{align}
We note that the relation \eqref{101} deduces
\begin{align*}
	(p+1)\alpha - d\beta
		> 0.
\end{align*}
We set a positive constant $\overline{\mu}$ as follows:
\begin{align*}
\overline{\mu} :=
\begin{cases}
&\hspace{-0.4cm}\displaystyle{
	(p+1)\alpha - d\beta,
	}\ \text{ if }\ \beta \geq 0,\\
&\hspace{-0.4cm}\displaystyle{
	2\alpha - d\beta,
	}\ \text{ if }\ \beta < 0.
\end{cases}
\end{align*}
In particular, we use $(\alpha,\beta) = (d,2), (2,p-1), (0,-1)$ in this paper.

\subsection{Some tools}

\begin{lemma}[Radial Sobolev inequality, \cite{OgaTsu91}]\label{Radial Sobolev inequality}
Let $d \geq 2$ and $p \geq 1$.
Then, the following inequality holds:
\begin{align*}
	\|f\|_{L^{p+1}(|x| \geq R)}^{p+1}
		\lesssim R^{-\frac{(d-1)(p-1)}{2}}\|f\|_{L^2(|x| \geq R)}^\frac{p+3}{2}\|\nabla f\|_{L^2(|x| \geq R)}^\frac{p-1}{2}
\end{align*}
for any $f \in H_\text{rad}^1(\mathbb{R}^d)$ and any $R > 0$, where the implicit constant is independent of $f$ and $R$.
\end{lemma}

\begin{lemma}[Gagliardo-Nirenberg inequality with the potential, \cite{Din18, KilMurVisZhe17}]\label{Gagliardo-Nirenberg inequality with the potential}
Let $d \geq 3$, $1 < p \leq 1 + \frac{4}{d-2}$, and $\gamma > 0$.
Then, we have
\begin{align}
	\|f\|_{L^{p+1}}^{p+1}
		< C_\text{GN}(\gamma)\|f\|_{L^2}^{p+1 - \frac{d(p-1)}{2}}\|(-\Delta_\gamma)^\frac{1}{2}f\|_{L^2}^\frac{d(p-1)}{2} \tag{GN$_\gamma$} \label{G-N inequality}
\end{align}
for any $f \in H^1(\mathbb{R}^d) \setminus \{0\}$, where $C_\text{GN}(\gamma)$ is the best constant and satisfies $C_\text{GN}(\gamma) = C_\text{GN}(0)$.
\end{lemma}

We recall that the inequality \eqref{G-N inequality} with $\gamma = 0$ is attained by the ground state $Q_{\omega,0}$ to \eqref{SP} with $\gamma = 0$.

The following Hardy inequality is well-known and hence, the energy $E_\gamma$ of solutions to \eqref{NLKG} is well-defined.

\begin{lemma}[Hardy inequality]\label{Hardy inequality}
Let $d \geq 3$.
Then, it follows that
\begin{align*}
	\left(\frac{d-2}{2}\right)^2\int_{\mathbb{R}^d}\frac{1}{|x|^2}|f(x)|^2dx
		\leq \|\nabla f\|_{L^2}^2
\end{align*}
for any $f \in H^1(\mathbb{R}^d)$.
\end{lemma}

\begin{lemma}[Equivalence of Sobolev norm, \cite{KilMiaVisZhaZhe18}]\label{Equivalence of Sobolev norm}
Let $d \geq 3$, $\gamma > - (\frac{d-2}{2})^2$, and $0 < s < 2$.
Let $\rho$ be defined as \eqref{141}.
If $1 < p < \infty$ satisfies $\frac{s+\rho}{d} < \frac{1}{p} < \min\{1,\frac{d-\rho}{d}\}$, then
\begin{align*}
	\||\nabla|^sf\|_{L^p}
		\lesssim_{d, p, s}\|(-\Delta_\gamma)^\frac{s}{2}f\|_{L^p}
\end{align*}
for any $f \in C_c^\infty(\mathbb{R}^d)$.
If $1 < p < \infty$ satisfies $\max\{\frac{s}{d},\frac{\rho}{d}\} < \frac{1}{p} < \min\{1,\frac{d-\rho}{d}\}$, then
\begin{align*}
	\|(-\Delta_\gamma)^\frac{s}{2}f\|_{L^p}
		\lesssim_{d, p, s}\||\nabla|^sf\|_{L^p}
\end{align*}
for any $f \in C_c^\infty(\mathbb{R}^d)$.
\end{lemma}

To state the Strichartz estimate (Theorem \ref{Strichartz estimate}), we define admissible pairs.

\begin{definition}\label{Admissible pair}
We call $(q,r)$ is admissible if $(q,r) \in [2,\infty] \times [2,\infty)$ satisfies
\begin{align*}
	\frac{2}{q} + \frac{d - 1 + \theta}{r}
		\leq \frac{d - 1 + \theta}{2},\ \ \ \ 
	(q,r,d,\theta)
		\neq (2,\infty,3,0)
\end{align*}
for $\theta \in [0,1]$.
For $s \in \mathbb{R}$, we define a set $\Lambda_{s,\theta}$ as
\begin{center}
	$\Lambda_{s,\theta}
		:= \{(q,r) \in [2,\infty] \times [2,\infty) : \frac{1}{q} + \frac{d + \theta}{r} = \frac{d + \theta}{2} - s,\ (q,r)\text{ is admissible.}\}$.
\end{center}
For $\nu_0 > 0$, we define a set $\Lambda_{s,\theta,\nu_0}$ as
\begin{center}
	$\Lambda_{s,\theta,\nu_0}
		:= \{(q,r) \in \Lambda_{s,\theta} : \frac{1}{r} > \frac{1}{2} - \frac{1 + \nu_0}{d}\}$.
\end{center}
\end{definition}

\begin{theorem}[Strichartz estimate, \cite{BenCacDesZha20}]\label{Strichartz estimate}
Let $d \geq 3$, $\gamma > 0$, $0 \in I$, and $(u_0,u_1) \in H^s(\mathbb{R}^d) \times H^{s-1}(\mathbb{R}^d)$ for $s \geq 0$.
Let $u$ be a solution to the following equation on $I$:
\begin{equation*}
\begin{cases}
&\hspace{-0.4cm}\displaystyle{
	- \partial_t^2 u + \Delta_\gamma u - u
		= F(t,x),\quad (t,x) \in \mathbb{R} \times\mathbb{R}^d,
	} \\
&\hspace{-0.4cm}\displaystyle{
	(u(0,x),\partial_tu(0,x))
		= (u_0,u_1),\quad x\in\mathbb{R}^d,
	}
\end{cases}
\end{equation*}
If $q_1 > 2$, $(q_1,r_1) \in \Lambda_{s,\theta_1,\frac{(d-2)^2}{4}}$, and $(q_2,r_2) \in \Lambda_{1-s,\theta_2,\frac{(d-2)^2}{4}}$, then
\begin{align*}
	\|u\|_{L_t^{q_1}(I;L_x^{r_1})} + \|u\|_{L_t^\infty(I;H_x^s)}
		\lesssim \|(u_0,u_1)\|_{H^s \times H^{s-1}} + \|F\|_{L_t^{q_2'}(I;L_x^{r_2'})}.
\end{align*}
\end{theorem}

\section{Local well-posedness}\label{Sec:Local well-posedness}

In this section, we prove local well-posedness of \eqref{NLKG} (Theorem \ref{Local well-posedness}).
The proof is based on the constraction map argument used the Strichartz estimate (Theorem \ref{Strichartz estimate}).
In the proof, we use the same exponents with \cite{IkeInu16, KilStoVis14}.

\begin{proof}[Proof of Theorem \ref{Local well-posedness}]
Let $\|(u_0,u_1)\|_{H_x^1 \times L_x^2} \leq M$.
We define a map $\Phi$
\begin{align*}
	\Phi(u)
		:= \partial_tK_\gamma(t)u_0(x) + K_\gamma(t)u_1(x) - \int_0^tK_\gamma(t-s)(|u|^{p-1}u)(s)ds.
\end{align*}
Case 1: $1 < p \leq 1 + \frac{2}{d-2}$\\
We define a set $E$ and a distance on $E$ as
\begin{gather*}
	E
		:= \{(u,\partial_tu) \in L_t^\infty(I;H_x^1) \times L_t^\infty(I;L_x^2) : \|(u,\partial_tu)\|_{L_t^\infty(I;H_x^1) \times L_t^\infty(I;L_x^2)} \leq 2cM\}, \\
	d(u_1,u_2)
		:= \|(u_1 - u_2,\partial_tu_1 - \partial_tu_2)\|_{L_t^\infty(I;H_x^1) \times L_t^\infty(I;L_x^2)}.
\end{gather*}
Since $H^1(\mathbb{R}^d) \subset L^{2p}(\mathbb{R}^d)$ holds, we have
\begin{align*}
	\|(\Phi(u),\partial_t\Phi(u))\|_{L_t^\infty(I;H_x^1) \times L_t^\infty(I;L_x^2)}
		& \leq c\,\|(u_0,u_1)\|_{H_x^1 \times L_x^2} + c\,\||u|^{p-1}u\|_{L_t^1(I;L_x^2)} \\
		& = c\,\|(u_0,u_1)\|_{H_x^1 \times L_x^2} + c\,\|u\|_{L_t^p(I;L_x^{2p})}^p \\
		& \leq c\,\|(u_0,u_1)\|_{H_x^1 \times L_x^2} + c\,T\,\|u\|_{L_t^\infty(I;H_x^1)}^p \\
		& \leq (1 + 2^pc^pTM^{p-1})cM
\end{align*}
and
\begin{align*}
	d(\Phi(u_1),\Phi(u_2))
		& \leq c\,\||u_1|^{p-1}u_1 - |u_2|^{p-1}u_2\|_{L_t^1(I;L_x^2)} \\
		& \leq c\left\{\|u_1\|_{L_t^p(I;L_x^{2p})}^{p-1} + \|u_2\|_{L_t^p(I;L_x^{2p})}^{p-1}\right\}\|u_1 - u_2\|_{L_t^p(I;L_x^{2p})} \\
		& \leq c\,T\left\{\|u_1\|_{L_t^\infty(I;H_x^1)}^{p-1} + \|u_2\|_{L_t^\infty(I;H_x^1)}^{p-1}\right\}\|u_1 - u_2\|_{L_t^\infty(I;H_x^1)} \\
		& \leq 2^pc^pTM^{p-1}d(u_1,u_2).
\end{align*}
If we take $T > 0$ with $2^pc^pTM^{p-1} \leq \frac{1}{2}$, then $\Phi$ is a contraction map on $(E,d)$.\\
Case 2: $1 + \frac{4}{d+1} \leq p \leq 1 + \frac{4d}{(d+1)(d-2)}$\\
We set $\varrho := \frac{(p-1)(d+1)}{2}$, $\kappa := \frac{2(d+1)}{d-1}$, and $\theta := \frac{2d}{(p-1)(d+1)} - \frac{d-2}{2}$.
Then, $2 \leq \varrho \leq \frac{2d}{d-2}$ and $\theta \in [0,1]$ hold.
We define a set $E$ and a distance $d$ on $E$ as
\begin{gather*}
	E
		:= \Bigl\{u \in L_t^\infty(I;H_x^1) : \|u\|_{L_t^\infty(I;H_x^1) \cap L_t^\kappa(I;W_x^{\frac{1}{2},\kappa})} + \|\partial_tu\|_{L_t^\infty(I;L_x^2)} \leq 2cM\Bigr\}, \\
	d(u_1,u_2)
		:= \|u_1 - u_2\|_{L_t^\kappa(I;L_x^\kappa)}.
\end{gather*}
Applying the Sobolev's embedding, we have $\|u\|_{L_t^\varrho(I;L_x^\varrho)} \leq 2cMT^\frac{1}{\varrho}$.
It follows from Theorem \ref{Strichartz estimate} that
\begin{align*}
	& \|\Phi(u)\|_{L_t^\infty(I;H_x^1)} + \|\partial_t\Phi(u)\|_{L_t^\infty(I;L_x^2)} + \|(1-\Delta)^\frac{1}{4}\Phi(u)\|_{L_t^\kappa(I;L_x^\kappa)} \\
		& \hspace{5.0cm} \leq c\,\|(u_0,u_1)\|_{H_x^1 \times L_x^2} + c\,\|(1-\Delta)^\frac{1}{4}(|u|^{p-1}u)\|_{L_t^{\kappa'}(I;L_x^{\kappa'})} \\
		& \hspace{5.0cm} \leq cM + c\,\|(1-\Delta)^\frac{1}{4}u\|_{L_t^\kappa(I;L_x^\kappa)}\|u\|_{L_t^\varrho(I;L_x^\varrho)}^{p-1} \\
		& \hspace{5.0cm} \leq (1 + 2^pc^pM^{p-1}T^\frac{p-1}{\varrho})cM
\end{align*}
and
\begin{align*}
	d(\Phi(u_1),\Phi(u_2))
		& \leq c\,\||u_1|^{p-1}u_1 - |u_2|^{p-1}u_2\|_{L_t^{\kappa'}(I;L_x^{\kappa'})} \\
		& \leq c\left\{\|u_1\|_{L_t^\varrho(I;L_x^\varrho)}^{p-1} + \|u_2\|_{L_t^\varrho(I;L_x^\varrho)}^{p-1}\right\}\|u_1 - u_2\|_{L_t^\kappa(I;L_x^\kappa)} \\
		& \leq 2^pc^pM^{p-1}T^\frac{p-1}{\varrho}d(u_1,u_2).
\end{align*}
Therefore, if we take $2^pc^pM^{p-1}T^\frac{p-1}{\varrho} \leq \frac{1}{2}$, then $\Phi$ is a contraction map on $(E,d)$. \\
Case 3: $1 + \frac{4d}{(d+1)(d-2)} < p < 1 + \frac{4}{d-2}$\\
We set $\frac{1}{\sigma} := \frac{d-2}{2} - \frac{2d}{(p-1)(d+1)} > 0$ and $\nu := \frac{1}{\varrho} - \frac{1}{\sigma} > 0$.
We define a set $E$ and a distance on $E$ as
\begin{gather*}
	E
		:= \{u \in L_t^\infty(I;H_x^1) : \|u\|_{L_t^\infty(I;H_x^1) \cap L_t^\sigma(I;L_x^\varrho) \cap L_t^\kappa(I;W_x^{\frac{1}{2},\kappa})} + \|\partial_tu\|_{L_t^\infty(I;L_x^2)} \leq 2cM\}, \\
	d(u_1,u_2)
		:= \|u_1 - u_2\|_{L_t^\kappa(I;L_x^\kappa)}.
\end{gather*}
By the H\"older inequality, we have
\begin{align*}
	\|u\|_{L_t^\varrho(I;L_x^\varrho)}
		\leq T^\nu\|u\|_{L_t^\sigma(I;L_x^\varrho)}
		\leq 2cMT^\nu.
\end{align*}
It follows from Theorem \ref{Strichartz estimate} that
\begin{align*}
	& \|\Phi(u)\|_{L_t^\infty(I;H_x^1)} + \|\partial_t\Phi(u)\|_{L_t^\infty(I;L_x^2)} + \|\Phi(u)\|_{L_t^\sigma(I;L_x^\varrho)} + \|(1-\Delta)^\frac{1}{4}\Phi(u)\|_{L_t^\kappa(I;L_x^\kappa)} \\
		& \hspace{3.0cm} \leq cM + c\,\|(1-\Delta)^\frac{1}{4}(|u|^{p-1}u)\|_{L_t^{\kappa'}(I;L_x^{\kappa'})} \\
		& \hspace{3.0cm} \leq cM + c\,\|u\|_{L_t^\varrho(I;L_x^\varrho)}^{p-1}\|(1-\Delta_\gamma)^\frac{1}{4}u\|_{L_t^\kappa(I;L_x^\kappa)} \\
		& \hspace{3.0cm} \leq \{1 + 2^pc^pM^{p-1}T^{\nu(p-1)}\}cM
\end{align*}
and
\begin{align*}
	d(\Phi(u_1),\Phi(u_2))
		& \leq c\,\||u_1|^{p-1}u_1 - |u_2|^{p-1}u_2\|_{L_t^{\kappa'}(I;L_x^{\kappa'})} \\
		& \leq c\left\{\|u_1\|_{L_t^\varrho(I;L_x^\varrho)}^{p-1} + \|u_2\|_{L_t^\varrho(I;L_x^\varrho)}^{p-1}\right\}\|u_1 - u_2\|_{L_t^\kappa(I;L_x^\kappa)} \\
		& \leq 2^pc^pM^{p-1}T^{\nu(p-1)}d(u_1,u_2).
\end{align*}
Therefore, if we take $2^pc^pM^{p-1}T^{\nu(p-1)} \leq \frac{1}{2}$, then $\Phi$ is a contraction map on $(E,d)$.\\
Case 4: $p = 1 + \frac{4}{d-2}$\\
We note that $\varrho = \sigma = \frac{2(d+1)}{d-2}$ and $\nu = 0$ in this case.
We define a set $E$ and a distance $d$ on $E$ as
\begin{gather*}
E := \left\{u \in L_t^\infty(I;H_x^1) \left|
\begin{array}{l}
\|u\|_{L_t^\infty(I;H_x^1)} + \|\partial_tu\|_{L_t^\infty(I;L_x^2)} + \|u\|_{L_t^\kappa(I;W_x^{\frac{1}{2},\kappa})} \leq L,\\
\|u\|_{L_t^\sigma(I;L_x^\varrho)} \leq M
\end{array}
\right.\right\}, \\
d(u_1,u_2)
	:= \|u_1 - u_2\|_{L_t^\kappa(I;L_x^\kappa)},
\end{gather*}
where positive constants $L$ and $M$ are chosen later.
From Theorem \ref{Strichartz estimate}, we can take a sufficiently small $T = T(L,M) > 0$ satisfying
\begin{align*}
	& \|\partial_tK_\gamma(t)u_0\|_{L_t^\kappa(I;W_x^{\frac{1}{2},\kappa})} + \|\partial_tK_\gamma(t)u_0\|_{L_t^\sigma(I;L_x^\varrho)} \\
		& \hspace{2.0cm} + \|(1-\Delta_\gamma)^\frac{1}{2}K_\gamma(t)u_1\|_{L_t^\kappa(I;L_x^\kappa)} + \|K_\gamma(t)u_1\|_{L_t^\sigma(I;L_x^\varrho)}
		\leq \frac{1}{2}\min\{L,M\}.
\end{align*}
Then, we have
\begin{align*}
	\|\Phi(u)\|_{L_t^\kappa(I;W_x^{\frac{1}{2},\kappa})}
		& \leq \|\partial_tK_\gamma(t)u_0\|_{L_t^\kappa(I;W_x^{\frac{1}{2},\kappa})} + \|(1-\Delta_\gamma)^\frac{1}{2}K_\gamma(t)u_1\|_{L_t^\kappa(I;L_x^\kappa)} \\
		& \hspace{4.5cm} + c\,\|u\|_{L_t^\sigma(I;L_x^\rho)}^\frac{4}{d-2}\|(1-\Delta_\gamma)^\frac{1}{4}u\|_{L_t^\kappa(I;L_x^\kappa)} \\
		& \leq \frac{1}{2}L + c\,M^\frac{4}{d-2}L
		\leq L,
\end{align*}
where we take $M$ satisfying $0 < c\,M^\frac{4}{d-2} \leq \frac{1}{4}$ in the last inequality.
We also have
\begin{align*}
	\|\Phi(u)\|_{L_t^\sigma(I;L_x^\varrho)}
		& \leq \|\partial_tK_\gamma(t)u_0\|_{L_t^\sigma(I;L_x^\varrho)} + \|K_\gamma(t)u_1\|_{L_t^\sigma(I;L_x^\varrho)} \\
		& \hspace{3.0cm} + c\,\|u\|_{L_t^\sigma(I;L_x^\varrho)}^\frac{4}{d-2}\|(1-\Delta_\gamma)^\frac{1}{4}u\|_{L_t^\kappa(I;L_x^\kappa)} \\
		& \leq \frac{1}{2}M + c\,M^\frac{4}{d-2}L
		\leq M,
\end{align*}
where we take $L$ satisfying $0 < c\,M^\frac{6-d}{d-2}L \leq \frac{1}{2}$ in the last inequality.
By the same manner, it follows that
\begin{align*}
	d(\Phi(u_1),\Phi(u_2))
		& \leq c\,\Bigl\{\|u_1\|_{L_t^\sigma(I;L_x^\varrho)}^\frac{4}{d-2} + \|u_2\|_{L_t^\sigma(I;L_x^\varrho)}^\frac{4}{d-2}\Bigr\}\|u_1 - u_2\|_{L_t^\kappa(I;L_x^\kappa)} \\
		& \leq 2c\,M^\frac{4}{d-2}d(u_1,u_2)
		\leq \frac{1}{2}d(u_1,u_2).
\end{align*}
Therefore, we obtain the solution to \eqref{NLKG} in $(E,d)$.
Finally, we prove that the solution $u$ to \eqref{NLKG} in $(E,d)$ satisfies $\|(u,\partial_tu)\|_{L_t^\infty(I;H_x^1) \times L_t^\infty(I;L_x^2)} < \infty$.
\begin{align*}
	\|(u,\partial_tu)\|_{L_t^\infty(I;H_x^1) \times L_t^\infty(I;L_x^2)}
		& \leq c\,\|(u_0,u_1)\|_{H_x^1 \times L_x^2} + c\,\|(1-\Delta_\gamma)^\frac{1}{4}(|u|^\frac{4}{d-2}u)\|_{L_t^{\kappa'}(I;L_x^{\kappa'})} \\
		& \leq c\,\|(u_0,u_1)\|_{H_x^1 \times L_x^2} + c\,M^\frac{4}{d-2}L
		< \infty.
\end{align*}
\end{proof}

\section{Boundedness of global solutions}\label{Boundedness of global solutions}

In this section, we show that global existence implies uniform boundedness (Theorem \ref{Global implies boundedness}).
The proof is based on the argument in \cite{OhtTod07} due to \cite{Caz85, MerZaa03} (see also \cite{Miy21}).
In this section, we use the notation $f(t) = \|u(t)\|_{L_x^2}^2$.

\begin{lemma}\label{Boundedness of L2}
Let $d \geq 3$, $\gamma > - (\frac{d-2}{2})^2$, and $1 < p \leq 1 + \frac{4}{d-2}$.
Let $(u,\partial_tu) \in C_t([0,\infty);H_x^1(\mathbb{R}^d) \times L_x^2(\mathbb{R}^d))$ be a global solution to \eqref{NLKG}.
Then, it follows that
\begin{gather}
	\frac{d}{dt}([(p-1)f(t) - 2(p+1)E_\gamma(u(0),\partial_tu(0))]_+)
		\leq 0\ \ \text{ for }\ \ 0 \leq t < \infty, \label{115} \\
	E_\gamma(u(0),\partial_tu(0))
		\geq 0, \label{116} \\
	f(t)
		\leq \sup \left\{f(0),\frac{2(p+1)}{p-1}E_\gamma(u(0),\partial_tu(0))\right\}\ \ \text{ for }\ \ 0 \leq t < \infty, \label{117}
\end{gather}
where $[\,\cdot\,]_+ := \max\{\,\cdot\,,0\}$.
\end{lemma}

We note that the third one implies boundedness of $L^2$-norm for time global solutions.

\begin{proof}
We prove \eqref{115}.
We set
\begin{align*}
	g(t)
		= f(t) - \frac{2(p+1)}{p-1}E_\gamma(u(0),\partial_tu(0)).
\end{align*}
We assume for contradiction that there exists $t_0 \in [0,\infty)$ such that
\begin{align*}
	\frac{d}{dt}([(p-1)f(t_0) - 2(p+1)E_\gamma(u(0),\partial_tu(0))]_+)
		> 0.
\end{align*}
Then, we have $g'(t_0) > 0$ and $g(t_0) > 0$.
By the direct calculation,
\begin{align}
	g''(t)
		= f''(t)
		= (p-1)g(t) + (p+3)\|\partial_tu(t)\|_{L_x^2}^2 + (p-1)\|(-\Delta_\gamma)^\frac{1}{2}u(t)\|_{L_x^2}^2, \label{118}
\end{align}
which implies that $g''(t) \geq (p-1)g(t)$ for any $t \in [0,\infty)$.
Therefore, the function $g$ is convex and increasing on $[t_0,\infty)$ and hence, $\lim_{t \rightarrow \infty}g(t) = \infty$ and $g(t) > 0$ for any $t \geq t_0$.
Combining this fact and \eqref{118}, we have $f''(t) \geq (p+3)\|\partial_tu(t)\|_{L_x^2}^2$ for any $t \geq t_0$.
This inequality deduces that
\begin{align*}
	\{f'(t)\}^2
		= 4\{\text{Re}\<u(t),\partial_tu(t)\>_{L_x^2}\}^2
		\leq 4\|u(t)\|_{L_x^2}^2\|\partial_tu(t)\|_{L_x^2}^2
		\leq \frac{4}{p+3}f(t)f''(t)
\end{align*}
for any $t \geq t_0$.
We can see that
\begin{align*}
	[\{f(t)\}^{-\frac{p-1}{4}}]''
		= - \frac{p-1}{4}\{f(t)\}^{-\frac{p+7}{4}}\left[-\frac{p+3}{4}\{f'(t)\}^2 + f(t)f''(t)\right]
		\leq 0
\end{align*}
for any $t \geq t_0$, that is, $\{f(t)\}^{-\frac{p-1}{4}}$ is concave on $[t_0,\infty)$.
On the other hand, we have $\{f(t)\}^{-\frac{p-1}{4}} \longrightarrow 0$ as $t \rightarrow \infty$ from $\lim_{t \rightarrow \infty}f(t) = \infty$.
This is contradiction.
Therefore, we obtain \eqref{115}.
From \eqref{115}, it follows that
\begin{align*}
	(p-1)f(t) - 2(p+1)E_\gamma(u(0),\partial_tu(0))
		& \leq [(p-1)f(0) - 2(p+1)E_\gamma(u(0),\partial_tu(0))]_+
\end{align*}
for any $t \in [0,\infty)$, which implies \eqref{117}.
We prove \eqref{116}.
We assume for contradiction that $E_\gamma(u(0),\partial_tu(0)) < 0$.
Then, it follows from \eqref{118} that
\begin{align*}
	f''(t)
		\geq -2(p+1)E_\gamma(u(0),\partial_tu(0))
		=: \alpha
		> 0.
\end{align*}
Integrating this inequality on $[0,t]$, we have $f'(t) \geq \alpha t + f'(0)$ for any $t \geq 0$.
There exists $t_1 > 0$ such that $f'(t) > 1$ for any $t \geq t_1$, which implies that $f(t) \longrightarrow \infty$ as $t \rightarrow \infty$.
This contradicts \eqref{117}.
\end{proof}

\begin{lemma}
Let $d \geq 3$, $\gamma > - (\frac{d-2}{2})^2$, and $1 < p \leq 1 + \frac{4}{d-2}$.
Let $(u,\partial_tu) \in C_t([0,\infty);H_x^1(\mathbb{R}^d) \times L_x^2(\mathbb{R}^d))$ be a solution to \eqref{NLKG}.
Then, we have
\begin{align}
	f'(t)
		& \leq \frac{2(p+1)}{\sqrt{(p-1)(p+3)}}E_\gamma(u(0),\partial_tu(0)), \label{120} \\
	f'(t)
		& \geq \inf \left\{f'(0),- \frac{2(p+1)}{\sqrt{(p-1)(p+3)}}E_\gamma(u(0),\partial_tu(0))\right\} \label{121}
\end{align}
for any $t \in [0,\infty)$.
\end{lemma}

\begin{proof}
We prove \eqref{120}.
We set a function
\begin{align*}
	g(t)
		:= f'(t) - \frac{2(p+1)}{\sqrt{(p-1)(p+3)}}E_\gamma(u(0),\partial_tu(0))
\end{align*}
H\"older's inequality and Young's inequality give us
\begin{align}
	\sqrt{(p-1)(p+3)}|f'(t)|
		& = 2\sqrt{(p-1)(p+3)}|\text{Re}\<u(t),\partial_tu(t)\>_{L_x^2}| \notag \\
		& \leq 2\sqrt{(p-1)(p+3)}\|u(t)\|_{L_x^2}\|\partial_tu(t)\|_{L_x^2} \notag \\
		& \leq (p-1)\|u(t)\|_{L_x^2}^2 + (p+3)\|\partial_tu(t)\|_{L_x^2}^2 \label{122}
\end{align}
Using \eqref{118} and this inequality,
\begin{align*}
	g'(t)
		& = f''(t)
		\geq (p-1)\|u(t)\|_{L_x^2}^2 + (p+3)\|\partial_tu(t)\|_{L_x^2}^2 - 2(p+1)E_\gamma(u(0),\partial_tu(0)) \\
		& \geq \sqrt{(p-1)(p+3)}|f'(t)| - 2(p+1)E_\gamma(u(0),\partial_tu(0))
		\geq \sqrt{(p-1)(p+3)}g(t)
\end{align*}
for any $t \in [0,\infty)$ and hence, we have
\begin{align}
	g(t)
		\geq g(\tau)e^{\sqrt{(p-1)(p+3)}(t - \tau)} \label{119}
\end{align}
for any $\tau, t \in [0,\infty)$ with $\tau \leq t$.
Here, we assume for contradiction that there exists $t_0 \in [0,\infty)$ such that $g(t_0) > 0$.
Then, it follows from \eqref{116} and \eqref{119} that
\begin{align*}
	f'(t)
		& = g(t) + \frac{2(p+1)}{\sqrt{(p-1)(p+3)}}E_\gamma(u(0),\partial_tu(0)) \\
		& \geq g(t_0)e^{\sqrt{(p-1)(p+3)}(t-t_0)} + \frac{2(p+1)}{\sqrt{(p-1)(p+3)}}E_\gamma(u(0),\partial_tu(0))
		\geq g(t_0)
		> 0
\end{align*}
for any $t \geq t_0$, which implies that $f(t) \longrightarrow \infty$ as $t \rightarrow \infty$.
This contradicts \eqref{117}.
We prove \eqref{121}.
We define a function $h$ as
\begin{align*}
	h(t)
		= - f'(t) - \frac{2(p+1)}{\sqrt{(p-1)(p+3)}}E_\gamma(u(0),\partial_tu(0)).
\end{align*}
Combining \eqref{118} and \eqref{122},
\begin{align*}
	- h'(t)
		= f''(t)
		& \geq \sqrt{(p-1)(p+3)}|f'(t)| - 2(p+1)E_\gamma(u(0),\partial_tu(0)) \\
		& \geq - \sqrt{(p-1)(p+3)}f'(t) - 2(p+1)E_\gamma(u(0),\partial_tu(0)) \\
		& = \sqrt{(p-1)(p+3)}h(t),
\end{align*}
which deduces that $h(t) \leq h(0)e^{-\sqrt{(p-1)(p+3)}t}$ for any $t \geq 0$.
Therefore, we obtain $h(t) \leq \sup\{h(0),0\}$.
\end{proof}

\begin{lemma}\label{Boundedness of integral}
Let $d \geq 3$, $\gamma > - (\frac{d-2}{2})^2$, and $1 < p \leq 1 + \frac{4}{d-2}$.
Let $(u,\partial_tu) \in C_t([0,\infty);H_x^1(\mathbb{R}^d) \times L_x^2(\mathbb{R}^d))$.
Then, we have
\begin{align*}
	\int_t^{t+\tau}F(u(s),\partial_tu(s))ds
		& \leq 2(p+1)E_\gamma(u(0),\partial_tu(0))\tau \\
		& \hspace{0.5cm} + \frac{4(p+1)}{\sqrt{(p-1)(p+3)}}E_\gamma(u(0),\partial_tu(0)) + 2|\<u(0),\partial_tu(0)\>_{L_x^2}|
\end{align*}
for any $t, \tau \geq 0$, where $F(f,g) := (p-1)\|f\|_{L^2}^2 + (p-1)\|(-\Delta_\gamma)^\frac{1}{2}f\|_{L^2}^2 + (p+3)\|g\|_{L^2}^2$.
\end{lemma}

\begin{proof}
Integrating \eqref{118} on $s \in [t,t+\tau]$,
\begin{align*}
	\int_t^{t+\tau}f''(s)ds
		= \int_t^{t+\tau}F(u(s),\partial_tu(s))ds - 2(p+1)E_\gamma(u(0),\partial_tu(0))\tau,
\end{align*}
that is,
\begin{align*}
	\int_t^{t+\tau}F(u(s),\partial_tu(s))ds
		= 2(p+1)E_\gamma(u(0),\partial_tu(0))\tau + f'(t + \tau) - f'(t).
\end{align*}
Applying \eqref{120} and \eqref{121}, we have
\begin{align*}
	\int_t^{t+\tau}F(u(s),\partial_tu(s))ds
		\leq 2(p+1)E_\gamma(u(0),\partial_tu(0))\tau + \frac{4(p+1)}{\sqrt{(p-1)(p+3)}}E_\gamma(u(0),\partial_tu(0)) + |f'(0)|,
\end{align*}
which completes the proof.
\end{proof}

The following proposition holds by Lemmas \ref{Boundedness of L2} and \ref{Boundedness of integral}.

\begin{proposition}\label{Uniform estimate of global solution}
Let $d \geq 3$, $\gamma > - (\frac{d-2}{2})^2$, and $1 < p \leq 1 + \frac{4}{d-2}$.
Let $(u,\partial_tu) \in C_t([0,\infty);H_x^1(\mathbb{R}^d) \times L_x^2(\mathbb{R}^d))$ be a solution to \eqref{NLKG}.
Then, the followings hold:
\begin{gather}
	\sup_{t \geq 0}\|u(t)\|_{L_x^2}
		< \infty, \label{123} \\
	\sup_{t \geq 0}\int_t^{t+1} \|(u(s),\partial_tu(s))\|_{H_x^1 \times L_x^2}^2 ds
		< \infty. \label{124}
\end{gather}
\end{proposition}

To complete this section, we prove Theorem \ref{Global implies boundedness}.

\begin{proof}[Proof of Theorem \ref{Global implies boundedness}]
By the energy conservation law and \eqref{124}, we have
\begin{align}
	C_1
		:= \sup_{t \geq 0}\int_t^{t+1}\|u(s)\|_{L_x^{p+1}}^{p+1}ds
		< \infty. \label{125}
\end{align}
Combining the mean value theorem and \eqref{125}, for any $t \geq 0$, there exists $\tau(t) \in [t,t+1]$ such that
\begin{align}
	\|u(\tau(t))\|_{L_x^{p+1}}^{p+1}
		= \int_t^{t+1}\|u(s)\|_{L_x^{p+1}}^{p+1}ds
		\leq C_1. \label{126}
\end{align}
It follows from $2 < \frac{p+3}{2} < p+1$, \eqref{123}, and \eqref{126} that
\begin{align}
	\sup_{t \geq 0}\|u(\tau(t))\|_{L_x^\frac{p+3}{2}}
		\leq \sup_{t \geq 0}\|u(\tau(t))\|_{L_x^2}^{1-\theta}\|u(\tau(t))\|_{L_x^{p+1}}^\theta
		< \infty, \label{147}
\end{align}
where $\theta \in (0,1)$ satisfies $\frac{2}{p+3} = \frac{1-\theta}{2} + \frac{\theta}{p+1}$.
By the Young's inequality, we have
\begin{align*}
	\|u(t)\|_{L_x^\frac{p+3}{2}}^\frac{p+3}{2}
		& = \|u(\tau(t))\|_{L_x^\frac{p+3}{2}}^\frac{p+3}{2} + \int_{\tau(t)}^t\frac{d}{ds}\|u(s)\|_{L_x^\frac{p+3}{2}}^\frac{p+3}{2}ds \\
		& \lesssim \|u(\tau(t))\|_{L_x^\frac{p+3}{2}}^\frac{p+3}{2} + \int_{\tau(t)}^t\int_{\mathbb{R}^d}|u(s,x)|^\frac{p+1}{2}|\partial_su(s,x)|dxds \\
		& \lesssim \|u(\tau(t))\|_{L_x^\frac{p+3}{2}}^\frac{p+3}{2} + \int_{\tau(t)}^t(\|u(s)\|_{L_x^{p+1}}^{p+1} + \|\partial_su(s)\|_{L_x^2}^2)ds.
\end{align*}
Taking supremum of this inequality over $t \in [0,\infty)$, it follows from \eqref{147}, \eqref{125}, and \eqref{124} that
\begin{align*}
	\sup_{t \geq 0}\|u(t)\|_{L_x^\frac{p+3}{2}}^\frac{p+3}{2}
		\lesssim \sup_{t \geq 0}\|u(\tau(t))\|_{L_x^\frac{p+3}{2}}^\frac{p+3}{2} + \sup_{t \geq 0}\int_{\tau(t)}^t(\|u(s)\|_{L_x^{p+1}}^{p+1} + \|\partial_su(s)\|_{L_x^2}^2)ds
		< \infty.
\end{align*}
Applying Gagliardo--Nirenberg's inequality and Young's inequality, we have
\begin{align*}
	\|(u(t),\partial_tu(t))\|_{H_x^1 \times L_x^2}^2
		& \leq 8E_\gamma(u(0),\partial_tu(0)) + \frac{8}{p+1}\|u(t)\|_{L_x^{p+1}}^{p+1} \\
		& \leq 8E_\gamma(u(0),\partial_tu(0)) + c\,\|u(t)\|_{L_x^\frac{p+3}{2}}^{(p+1)(1 - \theta)}\|\nabla u(t)\|_{L_x^2}^{(p+1)\theta} \\
		& \leq 8E_\gamma(u(0),\partial_tu(0)) + C + \frac{1}{2}\|\nabla u(t)\|_{L_x^2}^2
\end{align*}
for some positive constant $C > 0$, where $\frac{1}{p+1} = \theta(\frac{1}{2} - \frac{1}{d}) + \frac{2(1 - \theta)}{p+3}$.
We note that the assumption $1 < p < 1 + \frac{4}{d-1}$ implies $(p+1)\theta < 2$.
Thus, we have
\begin{align*}
	\|(u(t),\partial_tu(t))\|_{H_x^1 \times L_x^2}^2
		\leq 16E_\gamma(u(0),\partial_tu(0)) + 2C,
\end{align*}
which completes the proof.
\end{proof}

\section{Existence of a minimizer to $r_{\omega,\gamma}^{\alpha,\beta}$}\label{Existence of a minimizer}

In this section, we prove that $r_{\omega,\gamma}^{\alpha,\beta}$ has a minimizer for some $(\alpha,\beta)$.

\subsection{Regularity of solutions to the elliptic equation}

We investigate regularity of $H^1$-solutions to \eqref{SP} in the next proposition, which assures $x \cdot \nabla Q_{\omega,\gamma} \in H^1(\mathbb{R}^d)$ and is used, for example, in Proposition \ref{M=G}.
The proof is based on \cite{Fuk21} due to \cite{Caz03} (see also \cite{FukOza05}).

\begin{proposition}\label{Regularity of the ground state}
Let $d \geq 3$, $1 < p < 1 + \frac{4}{d-2}$, $\gamma > 0$, and $1 - \omega^2 > 0$.
Let $Q_{\omega,\gamma} \in H^1(\mathbb{R}^d)$ be a solution to \eqref{SP}.
Then, we have $Q_{\omega,\gamma} \in C^2(\mathbb{R}^d \setminus \{0\})$.
Moreover, there exists $C > 0$ such that
\begin{align*}
	|Q_{\omega,\gamma}(x)| + |\nabla Q_{\omega,\gamma}(x)|
		\leq Ce^{-\frac{|x|}{d+2}}
\end{align*}
for any $|x| \geq 1$.
\end{proposition}

\begin{proof}
We set $Q_{\omega,\gamma}(x) = (1 - \omega^2)^\frac{1}{p-1}\phi((1 - \omega^2)^\frac{1}{2}x)$.
Then, $\phi$ solves
\begin{align*}
	- \phi + \Delta_\gamma \phi
		= - |\phi|^{p-1}\phi.
\end{align*}
Therefore, it suffices to consider case $\omega = 0$.
We define the following functions for $R > 0$.
A cut-off function $\mathscr{X}_R \in C^\infty(\mathbb{R})$ is radially symmetric and satisfies
\begin{equation*}
\mathscr{X}_R(r) := \mathscr{X}\left(\frac{r}{R}\right),\ \text{ where }\ \mathscr{X}(r) :=
\begin{cases}
\hspace{-0.4cm}&\displaystyle{\hspace{0.53cm}0\hspace{0.53cm}\ \ (0 \leq r \leq1),}\\
\hspace{-0.4cm}&\displaystyle{smooth\ \ (1 \leq r \leq 2),}\\
\hspace{-0.4cm}&\displaystyle{\hspace{0.53cm}1\hspace{0.53cm}\ \ (2 \leq r),}
\end{cases}
\end{equation*}
and $\mathscr{X}'(r) \geq 0$ for $r = |x|$.
Multiplying $\mathscr{X}_R$ by \eqref{SP} with $\omega = 0$, we have
\begin{align}
	- \mathscr{X}_R\phi + \Delta(\mathscr{X}_R\phi)
		= F_{R,\gamma,\phi}, \label{128}
\end{align}
where
\begin{align*}
	F_{R,\gamma,\phi}
		:= \left(\frac{\gamma}{|x|^2}\mathscr{X}_R + \Delta \mathscr{X}_R - |\phi|^{p-1}\mathscr{X}_R\right)\phi + 2\nabla \mathscr{X}_R\cdot\nabla \phi.
\end{align*}
We define a sequence $\{r_j\}_{1 \leq j \leq \infty}$ satisfying
\begin{gather*}
	\frac{1}{r_j}
		:= \frac{p^j}{p-1}\left(\frac{p-1}{p+1} - \frac{2}{d} + \frac{2}{dp^j}\right) + \frac{1}{d}.
\end{gather*}
We can see $\frac{p-1}{p+1} - \frac{2}{d} =: - \delta < 0$ from $p < 1 + \frac{4}{d-2}$.
The sequence $\{\frac{1}{r_j}\}$ is decreasing and satisfy $\frac{1}{r_j} \longrightarrow - \infty$ as $j \rightarrow \infty$ since
\begin{align*}
	\frac{1}{r_{j+1}} - \frac{1}{r_j}
		= \frac{p^{j+1} - p^j}{p-1}\left(\frac{p-1}{p+1} - \frac{2}{d}\right)
		= p^j\left(\frac{p-1}{p+1} - \frac{2}{d}\right)
		= - p^j\delta
		< 0.
\end{align*}
Combining the fact and $r_1 > 2$, there exists $J \geq 1$ such that $\frac{1}{r_j} > 0$ for $1 \leq j \leq J$ and $\frac{1}{r_{J+1}} \leq 0$ for $j \geq J + 1$.
We prove $\mathscr{X}_R\phi \in W^{1,r_j}(\mathbb{R}^d)$ with $1 \leq j \leq J$ for each $R > 0$ by utilizing induction.
From $\mathscr{X}_R\phi \in H^1(\mathbb{R}^d)$, we have $\mathscr{X}_R\phi \in L^{p+1}(\mathbb{R}^d)$ for each $R > 0$.
We see $F_{R,\gamma,\phi} \in L^{(p+1)/p}(\mathbb{R}^d)$ for each $R > 0$.
Indeed, it follows from
\begin{align*}
	\left\|\frac{\gamma}{|\,\cdot\,|^2}\mathscr{X}_R\phi\right\|_{L^\frac{p+1}{p}}
		& \leq \left\|\frac{\gamma}{|\,\cdot\,|^2}\right\|_{L^\frac{2(p+1)}{p-1}(R \leq |x|)}\|\mathscr{X}_R\phi\|_{L^2}, \\
	\|\Delta\mathscr{X}_R\phi\|_{L^\frac{p+1}{p}}
		& \leq \|\Delta\mathscr{X}_R\|_{L^\frac{p+1}{p-1}}\|\phi\|_{L^{p+1}(R \leq |x|)}, \\
	\|\mathscr{X}_R|\phi|^{p-1}\phi\|_{L^\frac{p+1}{p}}
		& \leq \|\mathscr{X}_R\|_{L^\infty}\|\phi\|_{L^{p+1}(R \leq |x|)}^p, \\
	\|\nabla\mathscr{X}_R\cdot\nabla \phi\|_{L^\frac{p+1}{p}}
		& \leq \|\nabla\mathscr{X}_R\|_{L^\frac{2(p+1)}{p-1}}\|\nabla \phi\|_{L^2(R \leq |x|)}.
\end{align*}
Using \eqref{128}, we have $\mathscr{X}_R\phi \in W^{2,\frac{p+1}{p}}(\mathbb{R}^d)$ for each $R > 0$.
Sobolev's embedding deduces
\begin{align*}
	\mathscr{X}_R\phi \in W^{1,r}(\mathbb{R}^d)\ \ \text{ for }\ \ \frac{p}{p+1} - \frac{1}{d} \leq \frac{1}{r} \leq \frac{p}{p+1}.
\end{align*}
In particular, we can get $\mathscr{X}_R\phi \in W^{1,r_1}(\mathbb{R}^d)$.
We assume $\mathscr{X}_R\phi \in W^{1,r_j}(\mathbb{R}^d)$ for $1 \leq j \leq J - 1$ and any $R > 0$ to use induction.
Since $\mathscr{X}_{R/2}\phi \in W^{1,r_j}(\mathbb{R}^d)$, we have $\phi \in W^{1,r_j}(R \leq |x|)$.
We see $F_{R,\gamma,\phi} \in L^{r_j/p}(\mathbb{R}^d)$ for each $R > 0$.
Indeed, it follows from
\begin{align*}
	\left\|\frac{\gamma}{|\,\cdot\,|^2}\mathscr{X}_R\phi\right\|_{L^\frac{r_j}{p}}
		& \leq \left\|\frac{\gamma}{|\,\cdot\,|^2}\right\|_{L^\frac{r_j}{p-1}(R \leq |x|)}\|\mathscr{X}_R\phi\|_{L^{r_j}}, \\
	\|\Delta\mathscr{X}_R\phi\|_{L^\frac{r_j}{p}}
		& \leq \|\Delta\mathscr{X}_R\|_{L^\frac{r_j}{p-1}}\|\phi\|_{L^{r_j}(R \leq |x|)}, \\
	\|\mathscr{X}_R|\phi|^{p-1}\phi\|_{L^\frac{r_j}{p}}
		& \leq \|\mathscr{X}_R\|_{L^\infty}\|\phi\|_{L^{r_j}(R \leq |x|)}^p, \\
	\|\nabla\mathscr{X}_R\cdot\nabla \phi\|_{L^\frac{r_j}{p}}
		& \leq \|\nabla\mathscr{X}_R\|_{L^\frac{r_j}{p-1}}\|\nabla \phi\|_{L^{r_j}}.
\end{align*}
Using \eqref{128}, we have $\mathscr{X}_R\phi \in W^{2,\frac{r_j}{p}}(\mathbb{R}^d)$ for each $R > 0$.
Sobolev's embedding deduces
\begin{align*}
	\mathscr{X}_R\phi \in W^{1,r}(\mathbb{R}^d)\ \ \text{ for }\ \ \frac{p}{r_j} - \frac{1}{d} \leq \frac{1}{r} \leq \frac{p}{r_j}.
\end{align*}
From $\frac{1}{r_{j+1}} = \frac{p}{r_j} - \frac{1}{d}$, the desired result holds.
Applying the same argument once again, we get $\mathscr{X}_R\phi \in W^{1,r}(\mathbb{R}^d)$ for any $r \geq \frac{r_J}{p}$.
If we take $r = \infty$, we see $F_{R,\gamma,\phi} \in L^2(\mathbb{R}^d) \cap L^\infty(\mathbb{R}^d)$ and hence, we have $\mathscr{X}_R\phi \in W^{2,r}(\mathbb{R}^d)$ for any $2 \leq r < \infty$ and $R > 0$.
We prove $F_{R,\gamma,\phi} \in W^{1,r}(\mathbb{R}^d)$ for any $2 \leq r < \infty$ and $R > 0$.
Since $|\mathscr{X}_{R/2}\phi|^{p-1}\mathscr{X}_{R/2}\phi \in W^{1,r}(\mathbb{R}^d)$ for any $2 \leq r < \infty$ and $R > 0$ (see \cite[Remark 1.4.1]{Caz03}), $F_{R,\gamma,\phi} \in W^{1,r}(\mathbb{R}^d)$ for any $2 \leq r < \infty$ and $R > 0$.
Therefore, $\mathscr{X}_R\phi \in W^{3,r}(\mathbb{R}^d)$ for any $2 \leq r < \infty$ and $R > 0$.
Sobolev's embedding implies $\mathscr{X}_R\phi \in C^{2,k}(\mathbb{R}^d)$ for any $0 < k < 1$, and $\mathscr{X}_R\phi$ and $\nabla(\mathscr{X}_R\phi)$ are Lipschitz continuous.
Therefore, $\phi \in C^2(\mathbb{R}^d \setminus \{0\})$ and $|\partial^\mathfrak{a}\phi(x)| \longrightarrow 0$ as $|x| \rightarrow \infty$ for any $\mathfrak{a} \in (\mathbb{N} \cup \{0\})^d$ with $|\mathfrak{a}| \leq 2$.
Let $\theta_\varepsilon(x) := e^{|x|/(1+\varepsilon|x|)}$.
Then, $\theta_\varepsilon$ is bounded and $|\nabla \theta_\varepsilon(x)| \leq \theta_\varepsilon(x) \leq e^{|x|}$ for any $x \in \mathbb{R}^d$.
Considering the scalar product of the equation \eqref{128} with $R = 1$ and $\theta_\varepsilon \mathscr{X} \overline{\phi}$, we have
\begin{align}
	& \int_{\mathbb{R}^d}\theta_\varepsilon|\mathscr{X}\phi|^2dx + \text{Re}\int_{\mathbb{R}^d}\nabla(\mathscr{X}\phi)\cdot\nabla(\theta_\varepsilon\mathscr{X}\overline{\phi})dx \notag \\
		& \hspace{3.0cm} = \<\mathscr{X}\phi - \Delta(\mathscr{X}\phi),\theta_\varepsilon \mathscr{X} \phi\>_{L^2}
		= - \<F_{1,\gamma,\phi},\theta_\varepsilon \mathscr{X}\phi\>_{L^2} \notag \\
		& \hspace{3.0cm} = - \int_{\mathbb{R}^d}\left(\frac{\gamma}{|x|^2} - |\phi|^{p-1}\right)\theta_\varepsilon|\mathscr{X}\phi|^2dx + C_0\int_{|x| \leq 2}e^{|x|}|\phi|^2dx. \label{143}
\end{align}
It follows from $|\nabla \theta_\varepsilon| \leq \theta_\varepsilon$ and Young's inequality that
\begin{align}
	& \int_{\mathbb{R}^d}\theta_\varepsilon|\mathscr{X}\phi|^2dx + \text{Re}\int_{\mathbb{R}^d}\nabla(\mathscr{X}\phi)\cdot\nabla(\theta_\varepsilon\mathscr{X}\overline{\phi})dx \notag \\
		& \hspace{2.0cm} = \int_{\mathbb{R}^d}\theta_\varepsilon(|\mathscr{X}\phi|^2+|\nabla(\mathscr{X}\phi)|^2)dx + \text{Re}\int_{\mathbb{R}^d}\mathscr{X}\overline{\phi}\nabla\theta_\varepsilon\cdot\nabla(\mathscr{X}\phi)dx \notag \\
		& \hspace{2.0cm} \geq \frac{1}{2}\int_{\mathbb{R}^d}\theta_\varepsilon(|\mathscr{X}\phi|^2+|\nabla(\mathscr{X}\phi)|^2)dx. \label{144}
\end{align}
We take positive constants $C_1 := \|(- \gamma|x|^{-2} + |\phi|^{p-1})\mathscr{X}\|_{L^\infty}$ and $R > 2$ with $- \gamma|x|^{-2} + |\phi|^{p-1} < 1/4$ for any $|x| > R$.
We note that there exists such $R > 0$ from $|\phi| \longrightarrow 0$ as $|x| \rightarrow \infty$.
Combining \eqref{143} and \eqref{144},
\begin{align*}
	\frac{1}{2}\int_{\mathbb{R}^d}\theta_\varepsilon(|\mathscr{X}\phi|^2+|\nabla(\mathscr{X}\phi)|^2)dx
		& \leq - \int_{\mathbb{R}^d}\left(\frac{\gamma}{|x|^2} - |\phi|^{p-1}\right)\theta_\varepsilon|\mathscr{X}\phi|^2dx + C_0\int_{|x| \leq 2}e^{|x|}|\phi|^2dx \\
		& \leq \frac{1}{4}\int_{\mathbb{R}^d}\theta_\varepsilon|\mathscr{X}\phi|^2dx + (C_0 + C_1)\int_{|x| \leq R}e^{|x|}|\phi|^2dx
\end{align*}
and hence, we have
\begin{align*}
	\int_{\mathbb{R}^d}\theta_\varepsilon(|\mathscr{X}\phi|^2+|\nabla(\mathscr{X}\phi)|^2)dx
		\lesssim \|\phi\|_{L^2}^2.
\end{align*}
Letting $\varepsilon \searrow 0$, we obtain
\begin{align}
	\int_{\mathbb{R}^d}e^{|x|}(|\mathscr{X}\phi|^2+|\nabla(\mathscr{X}\phi)|^2)dx
		\lesssim \|\phi\|_{L^2}^2
		< \infty. \label{145}
\end{align}
Since $\mathscr{X}\phi$ is Lipschitz, there exists $L > 0$ such that $|(\mathscr{X}\phi)(x) - (\mathscr{X}\phi)(y)| \leq L|x - y|$ for any $x, y \in \mathbb{R}^d$.
For each $x \in \mathbb{R}^d$, we define a set
\begin{align*}
	B(x)
		:= \{y \in \mathbb{R}^d : |x-y| \leq |(\mathscr{X}\phi)(x)|/2L\}.
\end{align*}
For any $y \in B(x)$, we have
\begin{align*}
	|(\mathscr{X}\phi)(x)|^2
		& \leq (|(\mathscr{X}\phi)(y)| + L|x - y|)^2 \\
		& \leq 2|(\mathscr{X}\phi)(y)|^2 + 2L^2|x - y|^2
		\leq 2|(\mathscr{X}\phi)(y)|^2 + \frac{1}{2}|(\mathscr{X}\phi)(x)|^2
\end{align*}
and hence, $|(\mathscr{X}\phi)(x)|^2 \leq 4|(\mathscr{X}\phi)(y)|^2$ holds.
Integrating this inequality over $B(x)$, we have
\begin{align*}
	|\mathbb{S}_d|\left(\frac{|(\mathscr{X}\phi)(x)|}{2L}\right)^d|(\mathscr{X}\phi)(x)|^2
		\leq 4 \int_{B(x)}|(\mathscr{X}\phi)(y)|^2dy,
\end{align*}
where $|\mathbb{S}_d|$ denotes the volume of the $d$-dimensional unit ball.
We set $a := \|\mathscr{X}\phi\|_{L^\infty}/2L$.
Since $e^{|x|} \leq e^{|y| + |(\mathscr{X}\phi)(x)|/2L} \leq e^ae^{|y|}$ for each $y \in B(x)$,
\begin{align*}
	e^{|x|}|(\mathscr{X}\phi)(x)|^{d+2}
		& \leq \frac{4(2L)^d}{|\mathbb{S}_d|}e^{|x|}\int_{B(x)}|(\mathscr{X}\phi)(y)|^2dy \\
		& \leq \frac{4(2L)^d}{|\mathbb{S}_d|}e^{a}\int_{B(x)}e^{|y|}|(\mathscr{X}\phi)(y)|^2dy
		\lesssim \frac{4(2L)^d}{|\mathbb{S}_d|}e^{a}\|\phi\|_{L^2}^2,
\end{align*}
where the last inequality is used \eqref{145}.
By the same argument for $\nabla(\mathscr{X}\phi)$,
\begin{align*}
	e^{|x|}|\nabla(\mathscr{X}\phi)(x)|^{d+2}
		\lesssim \frac{4(2M)^d}{|\mathbb{S}_d|}e^b\|\phi\|_{L^2}^2,
\end{align*}
where $M$ is a Lipschitz constant of $\nabla(\mathscr{X}\phi)$ and $b := \|\nabla(\mathscr{X}\phi)\|_{L^\infty}/2M$.
Therefore, we obtain
\begin{align*}
	|(\mathscr{X}\phi)(x)| + |\nabla(\mathscr{X}\phi)(x)|
		\leq Ce^{-\frac{|x|}{d+2}}
\end{align*}
for any $x \in \mathbb{R}^d$.
Since $\mathscr{X}\phi = \phi$ on $\{x \in \mathbb{R}^d : |x| > 1\}$, we conclude
\begin{align*}
	|\phi(x)| + |\nabla \phi(x)|
		\leq Ce^{-\frac{|x|}{d+2}}
\end{align*}
for any $|x| \geq 1$.
\end{proof}

\subsection{Existence of a minimizer to $r_{\omega,\gamma}^{\alpha,\beta}$}

In this subsection, we prove that the minimization problem $r_{\omega,\gamma}^{\alpha,\beta}$ (see \eqref{146} for the definition) has a minimizer.
The argument is based on that in \cite{IbrMasNak11, Naw94, OhtTod07, Sha83} (see also \cite{HamIkePro, HamIke20}).
First, we show that $(\dot{H}^1;\<\,\cdot\,,\,\cdot\,\>_{\dot{H}_\gamma^1})$ is a Hilbert space.
Thanks to the result, we get a lower semi-continuity for a norm including the potential (e.g. see \eqref{142}).

\begin{lemma}
Let $d \geq 3$ and $\gamma > - (\frac{d-2}{2})^2$.
Then, $(\dot{H}^1;\<\,\cdot\,,\,\cdot\,\>_{\dot{H}_\gamma^1})$ is a Hilbert space, where
\begin{align*}
	\<f,g\>_{\dot{H}_\gamma^1}
		:= \text{Re}\int_{\mathbb{R}^d}\nabla f(x)\cdot\overline{\nabla g(x)} + \frac{\gamma}{|x|^2}f(x)\overline{g(x)}dx.
\end{align*}
\end{lemma}

\begin{proof}
By the direct calculation, we have
\begin{align*}
	\<f,g+h\>_{\dot{H}_\gamma^1}
		= \<f,g\>_{\dot{H}_\gamma^1} + \<f,h\>_{\dot{H}_\gamma^1},
	\ \ \ 
	\overline{\<f,g\>_{\dot{H}_\gamma^1}}
		= \<g,f\>_{\dot{H}_\gamma^1},
	\ \ \text{ and }\ \ 
	\<f,f\>_{\dot{H}_\gamma^1}
		\geq 0
\end{align*}
for any $f, g, h \in \dot{H}^1(\mathbb{R}^d)$.
If $\<f,f\>_{\dot{H}_\gamma^1} = 0$, then we have $0 \leq \|\nabla f\|_{L^2}^2 \sim \<f,f\>_{\dot{H}_\gamma^1} = 0$ from Lemma \ref{Equivalence of Sobolev norm}.
Thus, $f = 0$ holds.
Take a Cauchy sequence $\{f_n\} \subset (\dot{H}^1;\|\,\cdot\,\|_{\dot{H}_\gamma^1})$.
Then, it follows that
\begin{align*}
	0 \leq
		\|f_m - f_n\|_{\dot{H}^1}
		\sim \|f_m - f_n\|_{\dot{H}_\gamma^1}
		\longrightarrow 0\ \text{ as }\ m, n \rightarrow \infty,
\end{align*}
which implies $\{f_n\} \subset (\dot{H}^1;\|\,\cdot\,\|_{\dot{H}^1})$ is a Cauchy sequence.
Therefore, there exists $f_\infty \in \dot{H}^1(\mathbb{R}^d)$ such that $\|f_n - f_\infty\|_{\dot{H}^1} \longrightarrow 0$ as $n \rightarrow \infty$.
Combining this limit and Lemma \ref{Hardy inequality}, we obtain
\begin{align*}
	0
		\leq \|f_n - f_\infty\|_{\dot{H}_\gamma^1}
		\sim \|f_n - f_\infty\|_{\dot{H}^1}
		\longrightarrow 0\ \text{ as }\ n \rightarrow \infty.
\end{align*}
Therefore, $\{f_n\}$ is a convergent sequence in $(\dot{H}^1;\|\,\cdot\,\|_{\dot{H}_\gamma^1})$.
\end{proof}

To get a minimizer of $r_{\omega,\gamma}^{\alpha,\beta}$, we define the following functionals
\begin{align}
	T_{\omega,\gamma}^{\alpha,\,\beta}(f)
		& := S_{\omega,\gamma}(f) - \frac{1}{\overline{\mu}}K_{\omega,\gamma}^{\alpha,\,\beta}(f) \notag \\
		& =
\begin{cases}
&\hspace{-0.4cm}\displaystyle{
		\frac{(p-1)\alpha(1 - \omega^2)}{2\{(p+1)\alpha - d\beta\}}\|f\|_{L^2}^2 + \frac{(p - 1)\alpha - 2\beta}{2\{(p+1)\alpha - d\beta\}}\|f\|_{L^{p+1}}^{p+1}\ \ \text{ if }\ \ \beta \geq 0,
		}\\
&\hspace{-0.4cm}\displaystyle{
		- \frac{\beta}{2\alpha - d\beta}\|(-\Delta_\gamma)^\frac{1}{2}f\|_{L^2}^2 + \frac{(p-1)\alpha}{(p+1)(2\alpha - d\beta)}\|f\|_{L^{p+1}}^{p+1}\ \ \text{ if }\ \ \beta < 0,
		}
\end{cases} \notag \\
	L_{\omega,\gamma}(u,v)
		& := (E_\gamma+\omega C)(u,v)
		= S_{\omega,\gamma}(u) + \frac{1}{2}\|v - i\omega u\|_{L^2}^2 \label{112}
\end{align}
and sets
\begin{align*}
	\mathcal{R}_{\omega,+}^{\alpha,\beta}
		:& = \{(u,v) \in H_\text{rad}^1(\mathbb{R}^d) \times L_\text{rad}^2(\mathbb{R}^d) : L_{\omega,\gamma}(u,v) < r_{\omega,\gamma}^{\alpha,\beta},\ K_{\omega,\gamma}^{\alpha,\beta}(u) \geq 0\}, \\
	\mathcal{R}_{\omega,-}^{\alpha,\beta}
		:& = \{(u,v) \in H_\text{rad}^1(\mathbb{R}^d) \times L_\text{rad}^2(\mathbb{R}^d) : L_{\omega,\gamma}(u,v) < r_{\omega,\gamma}^{\alpha,\beta},\ K_{\omega,\gamma}^{\alpha,\beta}(u) < 0\}.
\end{align*}

\begin{lemma}\label{Rewriting infimum}
Let $d \geq 3$, $1 < p < 1 + \frac{4}{d-2}$, $\gamma > 0$, and $1 - \omega^2 > 0$.
Let $(\alpha,\beta)$ satisfy \eqref{101}.
Then,
\begin{align*}
	r_{\omega,\gamma}^{\alpha,\,\beta}
		= \inf\{T_{\omega,\gamma}^{\alpha,\,\beta}(f) : f \in H_\text{rad}^1(\mathbb{R}^d) \setminus \{0\},\,K_{\omega,\gamma}^{\alpha,\,\beta}(f) \leq 0\}
\end{align*}
holds.
\end{lemma}

\begin{proof}
Let $\~{r_{\omega,\gamma}}^{\alpha,\,\beta}$ denote the right hand side.
Since $f \in H_\text{rad}^1(\mathbb{R}^d) \setminus \{0\}$ with $K_{\omega,\gamma}^{\alpha,\,\beta}(f) = 0$ satisfies $S_{\omega,\gamma}(f) = T_{\omega,\gamma}^{\alpha,\beta}(f)$, it follows that $r_{\omega,\gamma}^{\alpha,\,\beta} \geq \~{r_{\omega,\gamma}}^{\alpha,\,\beta}$.
When $f \in H_\text{rad}^1(\mathbb{R}^d) \setminus \{0\}$ satisfies $K_{\omega,\gamma}^{\alpha,\,\beta}(f) \leq 0$, there exists $\lambda \in (0,1]$ such that $K_{\omega,\gamma}^{\alpha,\,\beta}(\lambda f) = 0$.
For such $\lambda \in (0,1]$, we have
\begin{align*}
	r_{\omega,\gamma}^{\alpha,\,\beta}
		\leq S_{\omega,\gamma}(\lambda f)
		= T_{\omega,\gamma}^{\alpha,\,\beta}(\lambda f)
		\leq T_{\omega,\gamma}^{\alpha,\,\beta}(f).
\end{align*}
Therefore, we obtain $r_{\omega,\gamma}^{\alpha,\,\beta} \leq \~{r_{\omega,\gamma}}^{\alpha,\,\beta}$ by taking infimum over $f \in H_\text{rad}^1(\mathbb{R}^d) \setminus \{0\}$ satisfying $K_{\omega,\gamma}^{\alpha,\beta}(f) \leq 0$.
\end{proof}

\begin{lemma}\label{Estimates for the ground state}
Let $d \geq 3$, $1 < p < 1 + \frac{4}{d-2}$, $\gamma > 0$, and $1 - \omega^2 > 0$.
Let $(\alpha,\beta)$ satisfy \eqref{101}.
Then, the followings hold:
\begin{itemize}
\item[(1)]
$T_{\omega,\gamma}^{\alpha,\beta}(f) > r_{\omega,\gamma}^{\alpha,\beta}$ for any $f \in H_\text{rad}^1(\mathbb{R}^d)$ satisfying $K_{\omega,\gamma}^{\alpha,\beta}(f) < 0$.
\item[(2)]
We assume that $\mathcal{G}_{\omega,\gamma,\text{rad}} \subset \mathcal{M}_{\omega,\gamma,\text{rad}}^{\alpha,\beta}$.
If $Q_{\omega,\gamma} \in \mathcal{G}_{\omega,\gamma,\text{rad}}$, then $(\lambda Q_{\omega,\gamma}, i\omega \lambda Q_{\omega,\gamma}) \in \mathcal{R}_{\omega,-}^{\alpha,\beta}$ for any $\lambda > 1$.
\end{itemize}
\end{lemma}

\begin{proof}
We prove (1).
From $K_{\omega,\gamma}^{\alpha,\beta}(f) < 0$, there exists $\lambda_0 \in (0,1)$ such that $K_{\omega,\gamma}^{\alpha,\beta}(\lambda_0 f) = 0$.
For such $\lambda_0 \in (0,1)$, we have
\begin{align*}
	r_{\omega,\gamma}^{\alpha,\beta}
		\leq S_{\omega,\gamma}(\lambda_0f)
		= T_{\omega,\gamma}^{\alpha,\beta}(\lambda_0f)
		< T_{\omega,\gamma}^{\alpha,\beta}(f).
\end{align*}
We prove (2).
The relation $Q_{\omega,\gamma} \in \mathcal{G}_{\omega,\gamma,\text{rad}}$ deduces
\begin{align*}
	K_{\omega,\gamma}^{1,0}(Q_{\omega,\gamma})
		= \mathcal{D}^{1,0}S_{\omega,\gamma}(Q_{\omega,\gamma})
		= \<S_{\omega,\gamma}'(Q_{\omega,\gamma}),Q_{\omega,\gamma}\>
		= 0.
\end{align*}
Thus, we have $\left.\frac{d}{d\lambda}\right|_{\lambda = 1}S_{\omega,\gamma}(\lambda Q_{\omega,\gamma}) = K_{\omega,\gamma}^{1,0}(Q_{\omega,\gamma}) = 0$ and
\begin{align}
	\frac{d}{d\lambda}S_{\omega,\gamma}(\lambda Q_{\omega,\gamma})
		< 0\ \text{ for any }\ \lambda > 1. \label{110}
\end{align}
Here, we note that
\begin{align}
	L_{\omega,\gamma}(\lambda Q_{\omega,\gamma},\lambda i\omega Q_{\omega,\gamma})
		= S_{\omega,\gamma}(\lambda Q_{\omega,\gamma}). \label{111}
\end{align}
Combining \eqref{110} and \eqref{111}, we have
\begin{align*}
	L_{\omega,\gamma}(\lambda Q_{\omega,\gamma},\lambda i\omega Q_{\omega,\gamma})
		< S_{\omega,\gamma}(Q_{\omega,\gamma})
		= r_{\omega,\gamma}^{\alpha,\beta}
\end{align*}
for any $\lambda > 1$, where we use $Q_{\omega,\gamma} \in \mathcal{M}_{\omega,\gamma,\text{rad}}^{\alpha,\beta}$ to get the last identity.
In addition, it follows from $K_{\omega,\gamma}^{\alpha,\beta}(Q_{\omega,\gamma}) = 0$ and $\frac{d}{d\lambda}K_{\omega,\gamma}^{\alpha,\beta}(\lambda Q_{\omega,\gamma}) < 0$ for any $\lambda > 1$ that $K_{\omega,\gamma}^{\alpha,\beta}(\lambda Q_{\omega,\gamma}) < 0$ for any $\lambda > 1$.
\end{proof}

\begin{lemma}\label{Invariant set}
The following identities hold:
\begin{align*}
	\mathcal{R}_{\omega,+}^{\alpha,\beta}
		& = \{(u,v) \in H_\text{rad}^1(\mathbb{R}^d) \times L_\text{rad}^2(\mathbb{R}^d) : L_{\omega,\gamma}(u,v) < r_{\omega,\gamma}^{\alpha,\beta},\ T_{\omega,\gamma}^{\alpha,\beta}(u) < r_{\omega,\gamma}^{\alpha,\beta}\}, \\
	\mathcal{R}_{\omega,-}^{\alpha,\beta}
		& = \{(u,v) \in H_\text{rad}^1(\mathbb{R}^d) \times L_\text{rad}^2(\mathbb{R}^d) : L_{\omega,\gamma}(u,v) < r_{\omega,\gamma}^{\alpha,\beta},\ T_{\omega,\gamma}^{\alpha,\beta}(u) > r_{\omega,\gamma}^{\alpha,\beta}\}.
\end{align*}
\end{lemma}

\begin{proof}
We note that
\begin{align}
	\mathcal{R}_{\omega,+}^{\alpha,\beta} \cup \mathcal{R}_{\omega,-}^{\alpha,\beta}
		= \{(u,v) \in H_\text{rad}^1(\mathbb{R}^d) \times L_\text{rad}^2(\mathbb{R}^d) : L_{\omega,\gamma}(u,v) < r_{\omega,\gamma}^{\alpha,\beta}\}\ \text{ and }\ 
	\mathcal{R}_{\omega,+}^{\alpha,\beta} \cap \mathcal{R}_{\omega,-}^{\alpha,\beta}
		= \emptyset \label{129}
\end{align}
We set
\begin{align*}
	\widetilde{\mathcal{R}}_{\omega,+}^{\alpha,\beta}
		& := \{(u,v) \in H_\text{rad}^1(\mathbb{R}^d) \times L_\text{rad}^2(\mathbb{R}^d) : L_{\omega,\gamma}(u,v) < r_{\omega,\gamma}^{\alpha,\beta},\ T_{\omega,\gamma}^{\alpha,\beta}(u) < r_{\omega,\gamma}^{\alpha,\beta}\}, \\
	\widetilde{\mathcal{R}}_{\omega,-}^{\alpha,\beta}
		& := \{(u,v) \in H_\text{rad}^1(\mathbb{R}^d) \times L_\text{rad}^2(\mathbb{R}^d) : L_{\omega,\gamma}(u,v) < r_{\omega,\gamma}^{\alpha,\beta},\ T_{\omega,\gamma}^{\alpha,\beta}(u) > r_{\omega,\gamma}^{\alpha,\beta}\}.
\end{align*}
First, we show that $L_{\omega,\gamma}(u,v) < r_{\omega,\gamma}^{\alpha,\beta}$ implies that $T_{\omega,\gamma}^{\alpha,\beta}(u) \neq r_{\omega,\gamma}^{\alpha,\beta}$.
This fact deduces that
\begin{align}
	\widetilde{\mathcal{R}}_{\omega,+}^{\alpha,\beta} \cup \widetilde{\mathcal{R}}_{\omega,-}^{\alpha,\beta}
		= \mathcal{R}_{\omega,+}^{\alpha,\beta} \cup \mathcal{R}_{\omega,-}^{\alpha,\beta}
		\ \text{ and }\ 
	\widetilde{\mathcal{R}}_{\omega,+}^{\alpha,\beta} \cap \widetilde{\mathcal{R}}_{\omega,-}^{\alpha,\beta}
		= \emptyset \label{130}
\end{align}
Let $(u,v) \in H_\text{rad}^1(\mathbb{R}^d) \times L_\text{rad}^2(\mathbb{R}^d)$ satisfy $L_{\omega,\gamma}(u,v) < r_{\omega,\gamma}^{\alpha,\beta}$.
When $(u,v)$ satisfies $K_{\omega,\gamma}^{\alpha,\beta}(u) \geq 0$, we have
\begin{align}
	T_{\omega,\gamma}^{\alpha,\,\beta}(u)
		= S_{\omega,\gamma}(u) - \frac{1}{\overline{\mu}}K_{\omega,\gamma}^{\alpha,\,\beta}(u)
		\leq S_{\omega,\gamma}(u)
		\leq L_{\omega,\gamma}(u,v)
		< r_{\omega,\gamma}^{\alpha,\beta}. \label{131}
\end{align}
When $(u,v)$ satisfies $K_{\omega,\gamma}^{\alpha,\beta}(u) < 0$, we have $T_{\omega,\gamma}^{\alpha,\beta}(u) > r_{\omega,\gamma}^{\alpha,\beta}$ from Lemma \ref{Estimates for the ground state}.
By \eqref{129} and \eqref{130}, it suffices to prove $\mathcal{R}_{\omega,+}^{\alpha,\beta} = \widetilde{\mathcal{R}}_{\omega,+}^{\alpha,\beta}$.\\
\eqref{131} implies that $\mathcal{R}_{\omega,+}^{\alpha,\beta} \subset \widetilde{\mathcal{R}}_{\omega,+}^{\alpha,\beta}$ holds.
We assume $\mathcal{R}_{\omega,+}^{\alpha,\beta} \subsetneq \widetilde{\mathcal{R}}_{\omega,+}^{\alpha,\beta}$ for contradiction.
Then, there exists $(u,v) \in H_\text{rad}^1(\mathbb{R}^d) \times L_\text{rad}^2(\mathbb{R}^d)$ such that $L_{\omega,\gamma}(u,v) < r_{\omega,\gamma}^{\alpha,\beta}$, $K_{\omega,\gamma}^{\alpha,\beta}(u) < 0$, and $T_{\omega,\gamma}^{\alpha,\beta}(u) < r_{\omega,\gamma}^{\alpha,\beta}$ hold.
This contradicts Lemma \ref{Estimates for the ground state}.
\end{proof}

\begin{lemma}\label{Estimate of virial}
Let $d \geq 3$, $1 < p < 1 + \frac{4}{d-2}$, $\gamma > 0$, and $1 - \omega^2 > 0$.
Let $(u_0,u_1) \in H_\text{rad}^1(\mathbb{R}^d) \times L_\text{rad}^2(\mathbb{R}^d)$ and $(\alpha,\beta)$ satisfy \eqref{101}.
We assume that $r_{\omega,\gamma}^{\alpha,\beta} > 0$.
Then, $\mathcal{R}_{\omega,+}^{\alpha,\beta}$ and $\mathcal{R}_{\omega,-}^{\alpha,\beta}$ are invariant under the flow of \eqref{NLKG}.
Moreover, if $(u_0,u_1) \in \mathcal{R}_{\omega,-}^{\alpha,\beta}$, then we have
\begin{align*}
	- \frac{1}{\overline{\mu}}K_{\omega,\gamma}^{\alpha,\beta}(u(t))
		> r_{\omega,\gamma}^{\alpha,\beta} - L_{\omega,\gamma}(u_0,u_1)
\end{align*}
for any $t \in [0,T_\text{max})$.
\end{lemma}

\begin{proof}
Let $(u_0,u_1) \in \mathcal{R}_{\omega,-}^{\alpha,\beta}$.
We assume that an open subinterval $I \varsubsetneq [0,T_\text{max})$ satisfies $K_{\omega,\gamma}^{\alpha,\beta}(u(t)) < 0$ for any $t \in I$ and ``$K_{\omega,\gamma}^{\alpha,\beta}(u(\inf I)) = 0$ or $K_{\omega,\gamma}^{\alpha,\beta}(u(\sup I)) = 0$''.
We consider only the case $K_{\omega,\gamma}^{\alpha,\beta}(u(\sup I)) = 0$ since the desired result holds by the same manner in the case $K_{\omega,\gamma}^{\alpha,\beta}(u(\inf I)) = 0$.
Lemma \ref{Rewriting infimum} implies
\begin{align*}
	S_{\omega,\gamma}(u(t)) - \frac{1}{\overline{\mu}}K_{\omega,\gamma}^{\alpha,\beta}(u(t))
		= T_{\omega,\gamma}^{\alpha,\beta}(u(t))
		\geq r_{\omega,\gamma}^{\alpha,\beta}
		> 0
\end{align*}
for each $t \in I$.
Taking a limit $t \longrightarrow \sup I$, we see $u(\sup I) \neq 0$.
By the definition of $r_{\omega,\gamma}^{\alpha,\beta}$, we have $S_{\omega,\gamma}(u(\sup I)) \geq r_{\omega,\gamma}^{\alpha,\beta}$.
On the other hand, it follows from \eqref{112} and $(u_0,u_1) \in \mathcal{R}_{\omega,-}^{\alpha,\beta}$ that
\begin{align*}
	S_{\omega,\gamma}(u(\sup I))
		\leq L_{\omega,\gamma}(u(\sup I),\partial_tu(\sup I))
		= L_{\omega,\gamma}(u_0,u_1)
		< r_{\omega,\gamma}^{\alpha,\beta},
\end{align*}
which is contradiction.
Therefore, we obtain
\begin{align*}
	- \frac{1}{\overline{\mu}}K_{\omega,\gamma}^{\alpha,\beta}(u(t))
		> r_{\omega,\gamma}^{\alpha,\beta} - S_{\omega,\gamma}(u(t))
		\geq r_{\omega,\gamma}^{\alpha,\beta} - L_{\omega,\gamma}(u(t),\partial_tu(t))
		= r_{\omega,\gamma}^{\alpha,\beta} - L_{\omega,\gamma}(u_0,u_1)
\end{align*}
for any $t \in [0,T_\text{max})$ by Lemma \ref{Estimates for the ground state} and \eqref{112}.\\
Let $(u_0,u_1) \in \mathcal{R}_{\omega,+}^{\alpha,\beta}$.
If we can not get the desired result, there exists $t_0 \in [0,T_\text{max})$ such that $(u(t_0),\partial_tu(t_0)) \in \mathcal{R}_{\omega,-}^{\alpha,\beta}$.
This contradicts the former result.
\end{proof}

\subsubsection{Case $1 + \frac{4}{d} < p < 1 + \frac{4}{d-2}$}

In this subsection, we get a minimizer to $r_{\omega,\gamma}^{d,2}$ for $1 + \frac{4}{d} < p < 1 + \frac{4}{d-2}$.
We note that
\begin{align*}
	K_{\omega,\gamma}^{d,2}(f)
		& = 2\|(-\Delta_\gamma)^\frac{1}{2}f\|_{L^2}^2 - \frac{(p-1)d}{p+1}\|f\|_{L^{p+1}}^{p+1}, \\
	T_{\omega,\gamma}^{d,2}(f)
		& = \frac{(1 - \omega^2)}{2}\|f\|_{L^2}^2 + \frac{(p - 1)d - 4}{2d(p-1)}\|f\|_{L^{p+1}}^{p+1}.
\end{align*}

\begin{lemma}\label{Equivalence of S and H1}
Let $d \geq 3$, $1 + \frac{4}{d} < p < 1 + \frac{4}{d-2}$, $\gamma > 0$, and $1 - \omega^2 > 0$.
Then, we have
\begin{align*}
	2\{(p-1)d - 4\}S_{\omega,\gamma}(f)
		\leq \{(p-1)d - 4\}\|f\|_{H_{\omega,\gamma}^1}^2
		\leq 2d(p-1)S_{\omega,\gamma}(f)
\end{align*}
for any $f \in H^1(\mathbb{R}^d)$ with $K_{\omega,\gamma}^{d,2}(f) \geq 0$, where
\begin{align*}
	\|f\|_{H_{\omega,\gamma}^1}^2
		:= (1 - \omega^2)\|f\|_{L^2}^2 + \|(-\Delta_\gamma)^\frac{1}{2}f\|_{L^2}^2.
\end{align*}
\end{lemma}

\begin{proof}
The left inequality holds immediately.
From $K_{\omega,\gamma}^{d,2}(f) \geq 0$, we have
\begin{align*}
	& \{(p-1)d - 4\}\|f\|_{H_{\omega,\gamma}^1}^2
		\leq \{(p-1)d - 4\}\|f\|_{H_{\omega,\gamma}^1}^2 + 2K_{\omega,\gamma}^{d,2}(f) \\
		& \hspace{1.0cm} = \{(p-1)d - 4\}(1 - \omega^2)\|f\|_{L^2}^2 + d(p-1)\|(-\Delta_\gamma)^\frac{1}{2}f\|_{L^2}^2 - \frac{2d(p-1)}{p+1}\|f\|_{L^{p+1}}^{p+1} \\
		& \hspace{1.0cm} = 2d(p-1)S_{\omega,\gamma}(f) - 4(1 - \omega^2)\|f\|_{L^2}^2 \\
		& \hspace{1.0cm} \leq 2d(p-1)S_{\omega,\gamma}(f).
\end{align*}
\end{proof}

\begin{lemma}
Let $d \geq 3$, $1 + \frac{4}{d} < p < 1 + \frac{4}{d-2}$, $\gamma > 0$, and $1 - \omega^2 > 0$.
Then, $r_{\omega,\gamma}^{d,2} > 0$ holds.
\end{lemma}

\begin{proof}
We take any $f \in H_\text{rad}^1(\mathbb{R}^d)$ with $K_{\omega,\gamma}^{d,2}(f) = 0$.
From Lemma \ref{Gagliardo-Nirenberg inequality with the potential}, we have
\begin{align*}
	2\|(-\Delta_\gamma)^\frac{1}{2}f\|_{L^2}^2
		\leq \frac{(p-1)d}{p+1}\|f\|_{L^{p+1}}^{p+1}
		\leq \frac{(p-1)d}{p+1}C_\text{GN}(0)\|f\|_{L^2}^{p+1 - \frac{d(p-1)}{2}}\|(-\Delta_\gamma)^\frac{1}{2}f\|_{L^2}^\frac{d(p-1)}{2}.
\end{align*}
Therefore, we obtain
\begin{align*}
	1
		\lesssim \|f\|_{L^2}^{p+1 - \frac{d(p-1)}{2}}\|(-\Delta_\gamma)^\frac{1}{2}f\|_{L^2}^{\frac{d(p-1)}{2} - 2}
		\lesssim \|f\|_{H_{\omega,\gamma}^1}^{p-1}
		\sim S_{\omega,\gamma}(f)^\frac{p-1}{2},
\end{align*}
which implies the desired result.
\end{proof}

\begin{lemma}\label{Positivity of K}
Let $d \geq 3$, $1 + \frac{4}{d} < p < 1 + \frac{4}{d-2}$, $\gamma > 0$, and $1 - \omega^2 > 0$.
Assume that $\{f_n\} \subset H^1(\mathbb{R}^d) \setminus \{0\}$ is bounded sequence and satisfies $\|\nabla f_n\|_{L^2} \longrightarrow 0$ as $n \rightarrow \infty$.
Then, there exists $n_0 \in \mathbb{N}$ such that $K_{\omega,\gamma}^{d,2}(f_n) > 0$ for any $n \geq n_0$.
\end{lemma}

\begin{proof}
Take a positive constant $C_0$ satisfying $\sup_{n \in \mathbb{N}}\|f_n\|_{L^2} \leq C_0$.
Using Lemma \ref{Gagliardo-Nirenberg inequality with the potential}, it follows that
\begin{align*}
	& K_{\omega,\gamma}^{d,2}(f_n)
		\geq \left\{2 - \frac{(p-1)d}{p+1}C_\text{GN}(0)\|f_n\|_{L^2}^\frac{d+2 - (d-2)p}{2}\|(-\Delta_\gamma)^\frac{1}{2}f_n\|_{L^2}^\frac{d(p-1)-4}{2}\right\}\|(-\Delta_\gamma)^\frac{1}{2}f_n\|_{L^2}^2 \\
		& \hspace{0.3cm} \geq \left\{2 - \frac{(p-1)d}{p+1}C_\text{GN}(0)C_0^\frac{d+2 - (d-2)p}{2}\|(-\Delta_\gamma)^\frac{1}{2}f_n\|_{L^2}^\frac{d(p-1)-4}{2}\right\}\|(-\Delta_\gamma)^\frac{1}{2}f_n\|_{L^2}^2
		\geq \|(-\Delta_\gamma)^\frac{1}{2}f_n\|_{L^2}^2
\end{align*}
for sufficiently large $n$.
From $f_n \neq 0$, we obtain the desired result.
\end{proof}

\begin{proposition}\label{Minimizer in inter-critical}
Let $d \geq 3$, $1 + \frac{4}{d} < p < 1 + \frac{4}{d-2}$, $\gamma > 0$, and $1 - \omega^2 > 0$.
Then, $r_{\omega,\gamma}^{d,2}$ has a minimizer.
\end{proposition}

\begin{proof}
We take a minimizing sequence $\{\phi_n\} \subset H_\text{rad}^1(\mathbb{R}^d) \setminus \{0\}$, that is, $\phi_n$ satisfies
\begin{align*}
	S_{\omega,\gamma}(\phi_n)
		= T_{\omega,\gamma}^{d,2}(\phi_n)
		\searrow r_{\omega,\gamma}^{d,2}\ \text{ as }\ n \rightarrow \infty \ \ \text{ and }\ \ 
	K_{\omega,\gamma}^{d,2}(\phi_n)
		= 0\ \text{ for any }\ n \in \mathbb{N}.
\end{align*}
From Lemma \ref{Equivalence of S and H1}, $\{\phi_n\}$ is a bounded sequence in $H_\gamma^1(\mathbb{R}^d)$.
Since $\dot{H}_\gamma^1(\mathbb{R}^d)$ is a Hilbert space and the embedding $H_\text{rad}^1(\mathbb{R}^d) \subset L^{p+1}(\mathbb{R}^d)$ is compact, there exists $\phi_\infty \in H_\text{rad}^1(\mathbb{R}^d)$ such that
\begin{align*}
	\phi_n
		\xrightharpoonup[]{\hspace{0.4cm}}\phi_\infty\ \text{ in }\ H_\gamma^1(\mathbb{R}^d)
	\ \ \text{ and }\ \ 
	\phi_n
		\longrightarrow \phi_\infty\ \text{ in }\ L^{p+1}(\mathbb{R}^d).
\end{align*}
Then, we have
\begin{align}
	\|\phi_\infty\|_{L^2}
		& \leq \liminf_{n \rightarrow \infty}\|\phi_n\|_{L^2}, \notag \\
	\|(-\Delta_\gamma)^\frac{1}{2}\phi_\infty\|_{L^2}
		& \leq \liminf_{n \rightarrow \infty}\|(-\Delta_\gamma)^\frac{1}{2}\phi_n\|_{L^2}, \label{142} \\
	\|\phi_\infty\|_{L^{p+1}}
		& = \lim_{n \rightarrow \infty}\|\phi_n\|_{L^{p+1}}. \notag
\end{align}
These relations deduce that
\begin{gather}
	S_{\omega,\gamma}(\phi_\infty)
		\leq \liminf_{n \rightarrow \infty}S_{\omega,\gamma}(\phi_n)
		= r_{\omega,\gamma}^{d,2}, \notag \\
	T_{\omega,\gamma}^{d,2}(\phi_\infty)
		\leq \liminf_{n \rightarrow \infty}T_{\omega,\gamma}^{d,2}(\phi_n)
		= r_{\omega,\gamma}^{d,2}, \label{132} \\
	K_{\omega,\gamma}^{d,2}(\phi_\infty)
		\leq \liminf_{n \rightarrow \infty}K_{\omega,\gamma}^{d,2}(\phi_n)
		= 0. \notag
\end{gather}
We prove $\phi_\infty \neq 0$.
If $\phi_\infty = 0$, then
\begin{align*}
	0
		= \frac{(p-1)d}{p+1}\lim_{n \rightarrow \infty}\|\phi_n\|_{L^{p+1}}^{p+1}
		= 2\lim_{n \rightarrow \infty}\|(-\Delta_\gamma)^\frac{1}{2}\phi_n\|_{L^2}^2
		\geq 0
\end{align*}
from $K_{\omega,\gamma}^{d,2}(\phi_n) = 0$.
Lemma \ref{Positivity of K} deduces contradiction and hence, $\phi_\infty \neq 0$ holds.
It follows from Lemma \ref{Rewriting infimum} and \eqref{132} that $T_{\omega,\gamma}^{d,2}(\phi_\infty) = r_{\omega,\gamma}^{\alpha,\beta}$.
Taking $\lambda \in (0,1]$ satisfying $K_{\omega,\gamma}^{\alpha,\beta}(\lambda \phi_\infty) = 0$, we obtain
\begin{align*}
	r_{\omega,\gamma}^{d,2}
		\leq S_{\omega,\gamma}(\lambda \phi_\infty)
		= T_{\omega,\gamma}^{d,2}(\lambda \phi_\infty)
		\leq T_{\omega,\gamma}^{d,2}(\phi_\infty)
		= r_{\omega,\gamma}^{d,2}.
\end{align*}
Therefore, $\lambda$ must be $1$ and $\phi_\infty$ attains $r_{\omega,\gamma}^{d,2}$.
\end{proof}

\subsubsection{Case $p = 1 + \frac{4}{d}$}

In this subsection, we get a minimizer to $r_{\omega,\gamma}^{d,2}$ for $p = 1 + \frac{4}{d}$.
We note that
\begin{gather*}
	K_{\omega,\gamma}^{d,2}(f)
		= 2\|(-\Delta_\gamma)^\frac{1}{2}f\|_{L^2}^2 - \frac{2d}{d+2}\|f\|_{L^\frac{2(d+2)}{d}}^\frac{2(d+2)}{d} \ \ \text{ and }\ \ 
	T_{\omega,\gamma}^{d,2}(f)
		= \frac{1 - \omega^2}{2}\|f\|_{L^2}^2.
\end{gather*}

\begin{lemma}\label{Lower bound of r}
Let $d \geq 3$, $p = 1 + \frac{4}{d}$, $\gamma > 0$, and $1 - \omega^2 > 0$.
Then, we have
\begin{align*}
	r_{\omega,\gamma}^{d,2}
		\geq \frac{1 - \omega^2}{2}\left\{\frac{d+2}{dC_\text{GN}(0)}\right\}^\frac{d}{2}
		> 0,
\end{align*}
where $C_\text{GN}(0)$ is defined in Lemma \ref{Gagliardo-Nirenberg inequality with the potential}.
\end{lemma}

\begin{proof}
From $K_{\omega,\gamma}^{d,2}(f) \leq 0$ and Lemma \ref{Gagliardo-Nirenberg inequality with the potential}, it follows that
\begin{align*}
	2\|(-\Delta_\gamma)^\frac{1}{2}f\|_{L^2}^2
		& \leq \frac{2d}{d+2}C_\text{GN}(0)\|f\|_{L^2}^\frac{4}{d}\|(-\Delta_\gamma)^\frac{1}{2}f\|_{L^2}^2.
\end{align*}
Therefore, we obtain
\begin{align*}
	\frac{1 - \omega^2}{2}\left\{\frac{d+2}{dC_\text{GN}(0)}\right\}^\frac{d}{2}
		\leq \frac{1 - \omega^2}{2}\|f\|_{L^2}^2
		= T_{\omega,\gamma}^{\alpha,\beta}(f).
\end{align*}
Taking infimum over $f \in H_\text{rad}^1(\mathbb{R}^d) \setminus \{0\}$ satisfying $K_{\omega,\gamma}^{d,2}(f) \leq 0$, we get the desired result from Lemma \ref{Rewriting infimum}.
\end{proof}

\begin{proposition}\label{Existence of minimizer for mass-critical}
Let $d \geq 3$, $p = 1 + \frac{4}{d}$, $\gamma > 0$, and $1 - \omega^2 > 0$.
Then, $r_{\omega,\gamma}^{d,2}$ has a minimizer.
\end{proposition}

\begin{proof}
We take a minimizing sequence $\{\phi_n\} \subset H_\text{rad}^1(\mathbb{R}^d) \setminus \{0\}$, that is, $\phi_n$ satisfies
\begin{align*}
	\frac{1 - \omega^2}{2}\|\phi_n\|_{L^2}^2
		= S_{\omega,\gamma}(\phi_n)
		\searrow r_{\omega,\gamma}^{d,2}\ \text{ as }\ n \rightarrow \infty \ \ \text{ and }\ \ 
	K_{\omega,\gamma}^{d,2}(\phi_n)
		= 0\ \text{ for any }n \in \mathbb{N}.
\end{align*}
Replacing $\phi_n$ with $\psi_n := \lambda^\frac{d}{2}\phi_n(\lambda\,\cdot\,)$ ($\lambda = \|\phi_n\|_{L^\frac{2(d+2)}{d}}^{-\frac{d+2}{d}}$), we have
\begin{align*}
	\frac{1 - \omega^2}{2}\|\psi_n\|_{L^2}^2
		\searrow r_{\omega,\gamma}^{d,2},\ \ \ 
	\|\psi_n\|_{L^\frac{2(d+2)}{d}}
		= 1,\ \ \text{ and }\ \ 
	K_{\omega,\gamma}^{d,2}(\psi_n)
		= 0.
\end{align*}
Then, $\{\psi_n\} \subset H_{\gamma,\text{rad}}^1(\mathbb{R}^d)$ is a bounded sequence.
Since $\dot{H}_\gamma^1(\mathbb{R}^d)$ is a Hilbert space and the embedding $H_\text{rad}^1(\mathbb{R}^d) \subset L^\frac{2(d+2)}{d}(\mathbb{R}^d)$ is compact, there exists a subsequence of $\{\psi_n\}$ (, which is denoted by the same symbol) and $\psi_\infty \in H_{\gamma,\text{rad}}^1(\mathbb{R}^d)$ such that
\begin{align*}
	\psi_n
		\xrightharpoonup[]{\hspace{0.4cm}}\psi_\infty\ \text{ in }\ H_\gamma^1(\mathbb{R}^d)
	\ \ \text{ and }\ \ 
	\psi_n
		\longrightarrow \psi_\infty\ \text{ in }\ L^\frac{2(d+2)}{d}(\mathbb{R}^d).
\end{align*}
Then, we have
\begin{gather*}
	\|\psi_\infty\|_{L^2}
		\leq \liminf_{n \rightarrow \infty}\|\psi_n\|_{L^2}, \\
	\|(-\Delta_\gamma)^\frac{1}{2}\psi_\infty\|_{L^2}
		\leq \liminf_{n \rightarrow \infty}\|(-\Delta_\gamma)^\frac{1}{2}\psi_n\|_{L^2}, \\
	\|\psi_\infty\|_{L^\frac{2(d+2)}{d}}
		= \lim_{n \rightarrow \infty}\|\psi_n\|_{L^\frac{2(d+2)}{d}}
		= 1.
\end{gather*}
These relations deduce that $\psi_\infty \in H_\text{rad}^1(\mathbb{R}^d) \setminus \{0\}$,
\begin{gather}
	S_{\omega,\gamma}(\psi_\infty)
		\leq \liminf_{n \rightarrow \infty}S_{\omega,\gamma}(\psi_n)
		= r_{\omega,\gamma}^{d,2}, \notag \\
	T_{\omega,\gamma}^{d,2}(\psi_\infty)
		\leq \liminf_{n \rightarrow \infty}T_{\omega,\gamma}^{d,2}(\psi_n)
		= r_{\omega,\gamma}^{d,2}, \label{133} \\
	K_{\omega,\gamma}^{d,2}(\psi_\infty)
		\leq \liminf_{n \rightarrow \infty}K_{\omega,\gamma}^{d,2}(\psi_n)
		= 0. \notag
\end{gather}
It follows from Lemma \ref{Rewriting infimum} and \eqref{133} that $T_{\omega,\gamma}^{d,2}(\psi_\infty) = r_{\omega,\gamma}^{d,2}$.
Taking $\lambda \in (0,1]$ satisfying $K_{\omega,\gamma}^{d,2}(\lambda \psi_\infty) = 0$, we obtain
\begin{align*}
	r_{\omega,\gamma}^{d,2}
		\leq S_{\omega,\gamma}(\lambda \psi_\infty)
		= T_{\omega,\gamma}^{d,2}(\lambda \psi_\infty)
		\leq T_{\omega,\gamma}^{d,2}(\psi_\infty)
		= r_{\omega,\gamma}^{d,2}.
\end{align*}
Thus, $\lambda$ must be $1$ and $\psi_\infty$ attains $r_{\omega,\gamma}^{d,2}$.
\end{proof}

\subsubsection{Case $1 < p < 1 + \frac{4}{d}$}

In this subsection, we get a minimizer to $r_{\omega,\gamma}^{2,p-1}$ and $r_{\omega,\gamma}^{0,-1}$ respectively for $1 < p < 1 + \frac{4}{d}$.
We note that
\begin{gather*}
	K_{\omega,\gamma}^{2,p-1}(f)
		= \frac{q(p-1)(1 - \omega^2)}{2}\|f\|_{L^2}^2 + \frac{(q+2)(p-1)}{2}\|(-\Delta_\gamma)^\frac{1}{2}f\|_{L^2}^2 - \frac{(q+2)(p-1)}{p+1}\|f\|_{L^{p+1}}^{p+1}, \\
	T_{\omega,\gamma}^{2,p-1}(f)
		= \frac{1 - \omega^2}{q+2}\|f\|_{L^2}^2,
\end{gather*}
where $q := \frac{4}{p-1} - d > 0$ and
\begin{gather*}
	K_{\omega,\gamma}^{0,-1}(f)
		= \frac{d(1 - \omega^2)}{2}\|f\|_{L^2}^2 + \frac{d-2}{2}\|(-\Delta_\gamma)^\frac{1}{2}f\|_{L^2}^2 - \frac{d}{p+1}\|f\|_{L^{p+1}}^{p+1}, \\
	T_{\omega,\gamma}^{0,-1}(f)
		= \frac{1}{d}\|(-\Delta_\gamma)^\frac{1}{2}f\|_{L^2}^2.
\end{gather*}

First, we get a minimizer to $r_{\omega,\gamma}^{2,p-1}$.

\begin{proposition}\label{Characterization of the ground state in mass-subcritical 1}
Let $d \geq 3$, $1 < p < 1 + \frac{4}{d}$, $\gamma > 0$, and $1 - \omega^2 > 0$.
Then, $r_{\omega,\gamma}^{2,p-1}$ has a minimizer.
\end{proposition}

\begin{proof}
Since $S_{\omega,\gamma}(f) = \frac{1 - \omega^2}{q+2}\|f\|_{L^2}^2$ if $K_{\omega,\gamma}^{2,p-1}(f) = 0$, we have $r_{\omega,\gamma}^{2,p-1} \geq 0$.
Take a minimizing sequence $\{\phi_n\} \subset H_\text{rad}^1(\mathbb{R}^d) \setminus \{0\}$ of $r_{\omega,\gamma}^{2,p-1}$, that is, $\phi_n$ satisfies
\begin{align*}
	\frac{1 - \omega^2}{q+2}\|\phi_n\|_{L^2}^2
		= S_{\omega,\gamma}(\phi_n)
		\searrow r_{\omega,\gamma}^{2,p-1}\ \text{ as }\ n \rightarrow \infty\ \ \text{ and }\ \ 
	K_{\omega,\gamma}^{2,p-1}(\phi_n)
		= 0\ \text{ for any }\ n \in \mathbb{N}.
\end{align*}
We notice that $\|\phi_n\|_{L^2}$ is a bounded sequence.
In addition, it follows from $K_{\omega,\gamma}^{2,p-1}(\phi_n) = 0$ and Proposition \ref{Gagliardo-Nirenberg inequality with the potential} that
\begin{align*}
	\|(-\Delta_\gamma)^\frac{1}{2}\phi_n\|_{L^2}^2
		\leq \|\phi_n\|_{L^{p+1}}^{p+1}
		\leq C_\text{GN}(0)\|\phi_n\|_{L^2}^{p+1 - \frac{d(p-1)}{2}}\|(-\Delta_\gamma)^\frac{1}{2}\phi_n\|_{L^2}^\frac{d(p-1)}{2}.
\end{align*}
Since $1 < p < 1 + \frac{4}{d}$ implies $0 < \frac{d(p-1)}{2} < 2$, we have
\begin{align*}
	\|(-\Delta_\gamma)^\frac{1}{2}\phi_n\|_{L^2}^{2 - \frac{d(p-1)}{2}}
		\leq C_\text{GN}(0)\|\phi_n\|_{L^2}^{p+1 - \frac{d(p-1)}{2}},
\end{align*}
and hence $\|(-\Delta_\gamma)^\frac{1}{2}\phi_n\|_{L^2}$ is a bounded sequence.
Therefore, there exist a subsequence of $\{\phi_n\}$ (, which is denoted by the same symbol) and $\phi_\infty \in H_{\gamma,\text{rad}}^1(\mathbb{R}^d)$ such that
\begin{align*}
	\phi_n
		\xrightharpoonup[]{\hspace{0.4cm}} \phi_\infty\ \text{ in }\ H_\gamma^1(\mathbb{R}^d)\ \text{ as }\ n \rightarrow \infty\ \ \text{ and }\ \ 
	\phi_n
		\longrightarrow \phi_\infty\ \text{ in }\ L^{p+1}(\mathbb{R}^d)\ \text{ as }\ n \rightarrow \infty.
\end{align*}
By the lower semi-continuity of weak convergence, we have
\begin{align*}
	\frac{q(p-1)}{2}\|\phi_\infty\|_{L^2}^2
		\leq \frac{q(p-1)}{2}\liminf_{n \rightarrow \infty}\|\phi_n\|_{L^2}^2\ \ \text{ and }\ \ 
	K_{\omega,\gamma}^{2,p-1}(\phi_\infty)
		\leq \liminf_{n \rightarrow \infty}K_{\omega,\gamma}^{2,p-1}(\phi_n)
		= 0.
\end{align*}
We assume that $\phi_\infty = 0$ for contradiction.
Combining $K_{\omega,\gamma}^{2,p-1}(\phi_n) = 0$ and $\phi_n \longrightarrow 0$ in $L^{p+1}(\mathbb{R}^d)$, we have $\phi_n \longrightarrow 0$ in $H_\gamma^1(\mathbb{R}^d)$.
On the other hand, we have
\begin{align*}
	& \frac{q(p-1)(1 - \omega^2)}{2}\|\phi_n\|_{L^2}^2 + \frac{(q+2)(p-1)}{2}\|(-\Delta_\gamma)^\frac{1}{2}\phi_n\|_{L^2}^2 \\
		& \hspace{3.0cm} = \frac{(q+2)(p-1)}{p+1}\|\phi_n\|_{L^{p+1}}^{p+1} \\
		& \hspace{3.0cm} \leq \frac{(q+2)(p-1)}{p+1}C_\text{GN}(0)\|\phi_n\|_{L^2}^{p+1 - \frac{d(p-1)}{2}}\|(-\Delta_\gamma)^\frac{1}{2}\phi_n\|_{L^2}^\frac{d(p-1)}{2} \\
		& \hspace{3.0cm} \lesssim \left\{\frac{q(p-1)(m^2 - \omega^2)}{2}\|\phi_n\|_{L^2}^2 + \frac{(q+2)(p-1)}{2}\|(-\Delta_\gamma)^\frac{1}{2}\phi_n\|_{L^2}^2\right\}^\frac{p+1}{2}
\end{align*}
from $K_{\omega,\gamma}^{2,p-1}(\phi_n) = 0$, Lemma \ref{Gagliardo-Nirenberg inequality with the potential}, and Young's inequality.
Therefore, we obtain
\begin{align*}
	1
		\lesssim \left\{\frac{q(p-1)(1 - \omega^2)}{2}\|\phi_n\|_{L^2}^2 + \frac{(q+2)(p-1)}{2}\|(-\Delta_\gamma)^\frac{1}{2}\phi_n\|_{L^2}^2\right\}^\frac{p-1}{2},
\end{align*}
which is contradiction.
Thus, we see that $\phi_\infty \in H_{\gamma,\text{rad}}^1(\mathbb{R}^d) \setminus \{0\}$ and
\begin{align*}
	r_{\omega,\gamma}^{2,p-1}
		\leq \frac{1 - \omega^2}{q+2}\|\phi_\infty\|_{L^2}^2
		\leq \frac{1 - \omega^2}{q+2}\liminf_{n \rightarrow \infty}\|\phi_n\|_{L^2}^2
		= r_{\omega,\gamma}^{2,p-1}.
\end{align*}
We can see that $K_{\omega,\gamma}^{2,p-1}(\phi_\infty) = 0$ from Lemma \ref{Estimates for the ground state}.
\end{proof}

Next, we get a minimizer to $r_{\omega,\gamma}^{0,1}$.
Before we prepare the following lemma.

\begin{lemma}\label{Positivity of K in mass-subcritical}
Let $d \geq 3$, $1 < p < 1 + \frac{4}{d}$, $\gamma > 0$, and $1 - \omega^2 > 0$.
Assume that $\{f_n\} \subset H^1(\mathbb{R}^d) \setminus \{0\}$ satisfies $\|f_n\|_{H_\gamma^1} \longrightarrow 0$ as $n \rightarrow \infty$.
Then, there exists $n_0 \in \mathbb{N}$ such that $K_{\omega,\gamma}^{0,-1}(f_n) > 0$ for any $n \geq n_0$.
\end{lemma}

\begin{proof}
Using Lemma \ref{Gagliardo-Nirenberg inequality with the potential}, it follows that
\begin{align*}
	K_{\omega,\gamma}^{0,-1}(f_n)
		& \geq \frac{d(1-\omega^2)}{2}\|f_n\|_{L^2}^2 + \frac{d-2}{2}\|(-\Delta_\gamma)^\frac{1}{2}f_n\|_{L^2}^2 - \frac{dC_\text{GN}(0)}{p+1}\|f_n\|_{L^2}^{p+1 - \frac{d(p-1)}{2}}\|(-\Delta_\gamma)^\frac{1}{2}f_n\|_{L^2}^\frac{d(p-1)}{2} \\
		& \geq \left[\min\left\{\frac{d(1-\omega^2)}{2},\frac{d-2}{2}\right\} - \frac{dC_\text{GN}(0)}{p+1}\|f_n\|_{H_\gamma^1}^{p-1}\right]\|f_n\|_{H_\gamma^1}^2 \\
		& \geq \frac{1}{2}\min\left\{\frac{d(1-\omega^2)}{2},\frac{d-2}{2}\right\}\|f_n\|_{H_\gamma^1}^2
		> 0
\end{align*}
for sufficiently large $n$.
\end{proof}

\begin{proposition}\label{Minimizer to r^{0,-1}}
Let $d \geq 3$, $1 < p < 1 + \frac{4}{d}$, $\gamma > 0$, and $1 - \omega^2 > 0$.
Then, $r_{\omega,\gamma}^{0,-1}$ has a minimizer.
\end{proposition}

\begin{proof}
Take a minimizing sequence $\{\phi_n\} \subset H_\text{rad}^1(\mathbb{R}^d) \setminus \{0\}$, that is, $\phi_n$ satisfies
\begin{align*}
	\frac{1}{d}\|(-\Delta_\gamma)^\frac{1}{2}\phi_n\|_{L^2}^2
		= S_{\omega,\gamma}(\phi_n)
		\searrow r_{\omega,\gamma}^{0,-1}\ \text{ as }\ n \rightarrow \infty \ \ \text{ and }\ \ 
	K_{\omega,\gamma}^{0,-1}(\phi_n)
		= 0\ \text{ for any }\ n \in \mathbb{N}.
\end{align*}
We see that $\{\phi_n\}$ is bounded in $\dot{H}_\gamma^1(\mathbb{R}^d)$.
By $K_{\omega,\gamma}^{0,-1}(\phi_n) = 0$, Young's inequality, and Sobolev's embedding, we have
\begin{align*}
	\frac{d(1 - \omega^2)}{2}\|\phi_n\|_{L^2}^2
		& = \frac{d}{p+1}\|\phi_n\|_{L^{p+1}}^{p+1} - \frac{d-2}{2}\|(-\Delta_\gamma)^\frac{1}{2}\phi_n\|_{L^2}^2 \\
		& \leq \varepsilon\|\phi_n\|_{L^2}^2 + C(\varepsilon)\|\phi_n\|_{L^\frac{2d}{d-2}}^\frac{2d}{d-2} - \frac{d-2}{2}\|(-\Delta_\gamma)^\frac{1}{2}\phi_n\|_{L^2}^2 \\
		& \leq \varepsilon\|\phi_n\|_{L^2}^2 + C(\varepsilon)\|(-\Delta_\gamma)^\frac{1}{2}\phi_n\|_{L^2}^\frac{2d}{d-2} - \frac{d-2}{2}\|(-\Delta_\gamma)^\frac{1}{2}\phi_n\|_{L^2}^2,
\end{align*}
which combined with boundedness of $\{\phi_n\}$ in $\dot{H}_\gamma^1(\mathbb{R}^d)$ implies boundedness of $\{\phi_n\}$ in $L^2(\mathbb{R}^d)$.
Since $H_{\gamma,\text{rad}}^1(\mathbb{R}^d) \subset L^{p+1}(\mathbb{R}^d)$ is compact, there exist a subsequence of $\{\phi_n\}$ (, which denoted by the same symbol) and $\phi_\infty \in H_\gamma^1(\mathbb{R}^d)$ such that
\begin{align*}
	\phi_n
		\xrightharpoonup[]{\hspace{0.4cm}} \phi_\infty\ \text{ in }\ H_\gamma^1(\mathbb{R}^d)\ \ \text{ and }\ \ 
	\phi_n
		\longrightarrow \phi_\infty\ \text{ in }\ L^{p+1}(\mathbb{R}^d)
\end{align*}
as $n \rightarrow \infty$.
In addition, we see that
\begin{gather*}
	\|\phi_\infty\|_{L^2}
		\leq \liminf_{n \rightarrow \infty}\|\phi_n\|_{L^2},\\
	\|(-\Delta_\gamma)^\frac{1}{2}\phi_\infty\|_{L^2}
		\leq \liminf_{n \rightarrow \infty}\|(-\Delta_\gamma)^\frac{1}{2}\phi_n\|_{L^2},\\
	\|\phi_\infty\|_{L^{p+1}}
		= \lim_{n \rightarrow \infty}\|\phi_n\|_{L^{p+1}}.
\end{gather*}
These relations deduce $K_{\omega,\gamma}^{0,-1}(\phi_\infty) \leq \liminf_{n \rightarrow \infty}K_{\omega,\gamma}^{0,-1}(\phi_n) = 0$.
Since $\phi_\infty \neq 0$ from Lemma \ref{Positivity of K in mass-subcritical}, we have
\begin{align*}
	r_{\omega,\gamma}^{0,-1}
		& \leq S_{\omega,\gamma}(\lambda\phi_\infty)
		= \frac{1}{d}\|(-\Delta_\gamma)^\frac{1}{2}(\lambda\phi_\infty)\|_{L^2}^2 \\
		& \leq \frac{1}{d}\|(-\Delta_\gamma)^\frac{1}{2}\phi_\infty\|_{L^2}^2
		\leq \frac{1}{d}\liminf_{n \rightarrow \infty}\|(-\Delta_\gamma)^\frac{1}{2}\phi_n\|_{L^2}^2
		= r_{\omega,\gamma}^{0,-1},
\end{align*}
where $\lambda \in (0,1]$ satisfies $K_{\omega,\gamma}^{0,-1}(\lambda\phi_\infty) = 0$.
So, $\lambda$ must be $1$ and $\phi_\infty$ is a minimizer to $r_{\omega,\gamma}^{0,-1}$.
\end{proof}

\section{Characterization of the ground state}\label{Characterization of the ground state}

In this section, we prove $\mathcal{G}_{\omega,\gamma,\text{rad}} \subset \mathcal{M}_{\omega,\gamma,\text{rad}}^{\alpha,\beta}$.
We note that the inverse direction also holds for $p \in (1,1 + \frac{4}{d}) \cup (1 + \frac{4}{d},1 + \frac{4}{d-2})$.\\

\noindent
6.1. \textbf{Case $1 < p < 1 + \frac{4}{d}$ and $1 + \frac{4}{d} < p < 1 + \frac{4}{d-2}$.}

\begin{proposition}\label{M=G}
Let $d \geq 3$, $p \in (1,1 + \frac{4}{d}) \cup (1 + \frac{4}{d},1 + \frac{4}{d-2})$, $\gamma > 0$, and $1 - \omega^2 > 0$.
Assume that $(\alpha,\beta) = (2,p-1), (0,-1)$ when $1 < p < 1 + \frac{4}{d}$ and $(\alpha,\beta) = (d,2)$ when $1 + \frac{4}{d} < p < 1 + \frac{4}{d-2}$.
Then, $\mathcal{M}_{\omega,\gamma,\text{rad}}^{\alpha,\beta} = \mathcal{G}_{\omega,\gamma,\text{rad}}$ holds.
\end{proposition}

\begin{proof}
First, we prove $\mathcal{M}_{\omega,\gamma,\text{rad}}^{\alpha,\beta} \subset \mathcal{G}_{\omega,\gamma,\text{rad}}$.
We take $\phi \in \mathcal{M}_{\omega,\gamma,\text{rad}}^{\alpha,\beta}$.
From $K_{\omega,\gamma}^{\alpha,\beta}(\phi) = 0$, we have
\begin{align}
	\<(K_{\omega,\gamma}^{\alpha,\beta})'(\phi),\mathcal{D}^{\alpha,\beta}\phi\>
		< 0. \label{138}
\end{align}
Indeed, we get
\begin{align*}
	\<(K_{\omega,\gamma}^{\alpha,\beta})'(\phi),\mathcal{D}^{\alpha,\beta}\phi\>
		& = \mathcal{D}^{\alpha,\beta}K_{\omega,\gamma}^{\alpha,\beta}(\phi) \\
		& = \begin{cases}
			&\hspace{-0.4cm}\displaystyle{
				\mathcal{D}^{\alpha,\beta}K_{\omega,\gamma}^{\alpha,\beta}(\phi) - \{2\alpha - (d-2)\beta\}K_{\omega,\gamma}^{\alpha,\beta}(\phi),\ \text{ if }\ (\alpha,\beta) = (2,p-1), (d,2)
			}\\
			&\hspace{-0.4cm}\displaystyle{
				\mathcal{D}^{0,-1}K_{\omega,\gamma}^{0,-1}(\phi) - dK_{\omega,\gamma}^{0,-1}(\phi)
			}
		\end{cases} \\
		& = \begin{cases}
		&\hspace{-0.4cm}\displaystyle{
			- \frac{(p-1)d\{(p-1)d - 4\}}{p+1}\|\phi\|_{L^{p+1}}^{p+1},\ \ \text{ if }\ \ (\alpha,\beta) = (d,2),
		} \\
		&\hspace{-0.4cm}\displaystyle{
			- (p-1)\{4 - d(p-1)\}\|\phi\|_{L^2}^2,\hspace{0.61cm}\ \ \text{ if }\ \ (\alpha,\beta) = (2,p-1),
		} \\
		&\hspace{-0.4cm}\displaystyle{
			- 2(d-2)\|(-\Delta_\gamma)^\frac{1}{2}\phi\|_{L^2}^2,\hspace{1.47cm}\ \ \text{ if }\ \ (\alpha,\beta) = (0,-1)
		}
		\end{cases}
\end{align*}
when $\phi$ satisfies $K_{\omega,\gamma}^{\alpha,\beta}(\phi) = 0$.
Thus, there exists the Lagrange multiplier $\eta \in \mathbb{R}$ such that
\begin{align}
	S_{\omega,\gamma}'(\phi)
		= \eta (K_{\omega,\gamma}^{\alpha,\beta})'(\phi). \label{113}
\end{align}
Using this identity,
\begin{align}
	0
		= K_{\omega,\gamma}^{\alpha,\beta}(\phi)
		= \mathcal{D}^{\alpha,\beta}S_{\omega,\gamma}(\phi)
		= \<S_{\omega,\gamma}'(\phi), \mathcal{D}^{\alpha,\beta}\phi\>
		= \eta \<(K_{\omega,\gamma}^{\alpha,\beta})'(\phi),\mathcal{D}^{\alpha,\beta}\phi\>. \label{114}
\end{align}
Combining \eqref{138} and \eqref{114}, we have $\eta = 0$.
Therefore, we obtain $S_{\omega,\gamma}'(\phi) = 0$ by \eqref{113}.
We take any function $\psi \in \mathcal{A}_{\omega,\gamma,\text{rad}}$.
Then, $K_{\omega,\gamma}^{\alpha,\beta}(\psi) = \<S_{\omega,\gamma}'(\psi), \mathcal{D}^{\alpha,\beta}(\psi)\> = 0$ and hence, $S_{\omega,\gamma}(\phi) \leq S_{\omega,\gamma}(\psi)$, that is, $\phi \in \mathcal{G}_{\omega,\gamma,\text{rad}}$. \\
Next, we prove $\mathcal{G}_{\omega,\gamma,\text{rad}} \subset \mathcal{M}_{\omega,\gamma,\text{rad}}^{\alpha,\beta}$.
We take $\phi \in \mathcal{G}_{\omega,\gamma,\text{rad}}$.
Since $\mathcal{M}_{\omega,\gamma,\text{rad}}^{\alpha,\beta}$ is not empty, we can take a function $\psi \in \mathcal{M}_{\omega,\gamma,\text{rad}}^{\alpha,\beta} \subset \mathcal{G}_{\omega,\gamma,\text{rad}}$.
Let $\Phi \in H_\text{rad}^1(\mathbb{R}^d) \setminus \{0\}$ satisfy $K_{\omega,\gamma}^{\alpha,\beta}(\Phi) = 0$.
Since $S_{\omega,\gamma}(\phi) = S_{\omega,\gamma}(\psi) \leq S_{\omega,\gamma}(\Phi)$ and $K_{\omega,\gamma}^{\alpha,\beta}(\phi) = \<S_{\omega,\gamma}'(\phi),\mathcal{D}^{\alpha,\beta}\phi\> = 0$, we obtain $\phi \in \mathcal{M}_{\omega,\gamma,\text{rad}}^{\alpha,\beta}$.
\end{proof}

\noindent
6.2. \textbf{Case $p = 1 + \frac{4}{d}$.}

\begin{proposition}
Let $d \geq 3$, $p = 1 + \frac{4}{d}$, $\gamma > 0$, and $1 - \omega^2 > 0$.
Then, $\mathcal{G}_{\omega,\gamma,\text{rad}}$ is not empty and $\mathcal{G}_{\omega,\gamma,\text{rad}} \subset \mathcal{M}_{\omega,\gamma,\text{rad}}^{d,2}$ holds.
\end{proposition}

\begin{proof}
We note that $r_{\omega,\gamma}^{d,2}$ has a minimizer $\psi_\infty \in H_\text{rad}^1(\mathbb{R}^d) \setminus \{0\}$ satisfying
\begin{align*}
	\|\psi_\infty\|_{L^2}^2
		= \frac{2}{1 - \omega^2}r_{\omega,\gamma}^{d,2},\ \ \ 
	\|(-\Delta_\gamma)^\frac{1}{2}\psi_\infty\|_{L^2}^2
		= \frac{d}{d+2},\ \ \text{ and }\ \ 
	\|\psi_\infty\|_{L^\frac{2(d+2)}{d}}^\frac{2(d+2)}{d}
		= 1
\end{align*}
from Proposition \ref{Existence of minimizer for mass-critical}.
Since $C_\text{GN}(\gamma)$ is the best constant of the Gagliardo-Nirenberg inequality (Lemma \ref{Gagliardo-Nirenberg inequality with the potential}), we can take a minimizing sequence $\{\phi_n\} \subset H_\text{rad}^1(\mathbb{R}^d) \setminus \{0\}$ satisfying
\begin{align*}
	J_\gamma(\phi_n)
		:= \frac{\|\phi_n\|_{L^2}^\frac{4}{d}\|(-\Delta_\gamma)^\frac{1}{2}\phi_n\|_{L^2}^2}{\|\phi_n\|_{L^\frac{2(d+2)}{d}}^\frac{2(d+2)}{d}}
		\searrow \frac{1}{C_\text{GN}(\gamma)}\ \text{ as }\ n \rightarrow \infty.
\end{align*}
We set $\psi_n(x) := \phi_n\bigl(\frac{x}{c_n}\bigr)$ for $c_n := \sqrt{\frac{d+2}{d}\cdot\|(-\Delta_\gamma)^\frac{1}{2}\phi_n\|_{L^2}^2\|\phi_n\|_{L^\frac{2(d+2)}{d}}^{-\frac{2(d+2)}{d}}}$.
Then, we have
\begin{align*}
	J_\gamma(\psi_n)
		= J_\gamma(\phi_n)
		= \frac{d}{d+2}\cdot\left(\frac{2}{1 - \omega^2}\right)^\frac{2}{d}\{T_{\omega,\gamma}^{d,2}(\psi_n)\}^\frac{2}{d}
\end{align*}
and
\begin{align*}
	K_{\omega,\gamma}^{d,2}(\psi_n)
		= c_n^{d-2}\left\{2\|(-\Delta)^\frac{1}{2}\phi_n\|_{L^2}^2 - \frac{2d}{d+2}c_n^2\|\phi_n\|_{L^\frac{2(d+2)}{d}}^\frac{2(d+2)}{d}\right\}	
		= 0.
\end{align*}
Therefore, it follows that
\begin{align*}
	r_{\omega,\gamma}^{d,2}
		\leq T_{\omega,\gamma}^{d,2}(\psi_n)
		= \frac{1 - \omega^2}{2}\left\{\frac{d+2}{d}J_\gamma(\psi_n)\right\}^\frac{d}{2}
		\longrightarrow \frac{1 - \omega^2}{2}\left\{\frac{d+2}{dC_\text{GN}(\gamma)}\right\}^\frac{d}{2}\ \text{ as }\ n \rightarrow \infty.
\end{align*}
Combining this inequality and Lemma \ref{Lower bound of r}, we have
\begin{align*}
	r_{\omega,\gamma}^{d,2}
		= \frac{1 - \omega^2}{2}\left\{\frac{d+2}{dC_\text{GN}(\gamma)}\right\}^\frac{d}{2},
\end{align*}
which implies that $\psi_\infty$ is a critical point of $J_\gamma$.
For any $\varphi \in C_c^\infty(\mathbb{R}^d)$, it follows that
\begin{align*}
	\left.\frac{d}{ds}\right|_{s=0}J_\gamma(\psi_\infty + s\varphi)
		= 0,
\end{align*}
that is,
\begin{align*}
	\Bigl\<- \frac{1 - \omega^2}{(d+2)r_{\omega,\gamma}^{d,2}}\psi_\infty + \Delta_\gamma \psi_\infty + |\psi_\infty|^\frac{4}{d}\psi_\infty,\varphi\Bigr\>_{H^{-1},H^1}
		= 0.
\end{align*}
We set
\begin{align*}
	\psi_\infty(x)
		= \nu^\frac{d}{2} Q_{\omega,\gamma}(\nu x),\ \ \ \nu^2 = \frac{1}{(d+2)r_{\omega,\gamma}^{d,2}}.
\end{align*}
Then, we obtain $S_{\omega,\gamma}'(Q_{\omega,\gamma}) = 0$, $K_{\omega,\gamma}^{d,2}(Q_{\omega,\gamma}) = 0$, and
\begin{gather*}
	S_{\omega,\gamma}(Q_{\omega,\gamma})
		= T_{\omega,\gamma}^{d,2}(Q_{\omega,\gamma})
		= \frac{1 - \omega^2}{2}\|Q_{\omega,\gamma}\|_{L^2}^2
		= \frac{1 - \omega^2}{2}\|\psi_\infty\|_{L^2}^2
		= r_{\omega,\gamma}^{d,2}.
\end{gather*}
We take any function $\psi \in \mathcal{A}_{\omega,\gamma,\,\text{rad}}$.
Then, $K_{\omega,\gamma}^{d,2}(\psi) = \<S_{\omega,\gamma}'(\psi), \mathcal{D}^{d,2}\psi\> = 0$ and hence, $S_{\omega,\gamma}(Q_{\omega,\gamma}) = r_{\omega,\gamma}^{d,2} \leq S_{\omega,\gamma}(\psi)$, that is, $Q_{\omega,\gamma} \in \mathcal{G}_{\omega,\gamma,\text{rad}}$.
Therefore, $\mathcal{G}_{\omega,\gamma,\text{rad}}$ is not empty.
We can also see $\mathcal{G}_{\omega,\gamma,\text{rad}} \subset \mathcal{M}_{\omega,\gamma,\text{rad}}^{d,2}$.
\end{proof}

\section{Proof of stability}\label{Proof of stability}

In this section, we prove the stability result in Theorem \ref{Instability versus stability}.
The proof is based on the argument in \cite{Sha83}.
Throughout this section, we assume local well-posedness of \eqref{NLKG} with \eqref{IC} for $1 + \frac{2}{d-2} < p < 1 + \frac{4}{d+1}$.

\begin{lemma}\label{convergence to minimizer}
If $\phi_n \in H_\text{rad}^1(\mathbb{R}^d)$ satisfies
\begin{align*}
	\frac{1}{d}\|(-\Delta_\gamma)^\frac{1}{2}\phi_n\|_{L^2}^2
		\longrightarrow r_{\omega,\gamma}^{0,-1}\ \ \text{ and }\ \ 
	S_{\omega,\gamma}(\phi_n)
		\longrightarrow r_0
		\leq r_{\omega,\gamma}^{0,-1}
\end{align*}
as $n \rightarrow \infty$ for some $r_0 \in \mathbb{R}$, then $r_0 = r_{\omega,\gamma}^{0,-1}$ and $\phi_n \longrightarrow \phi_\omega$ in $H^1(\mathbb{R}^d)$ for some minimizer $\phi_\omega$ to $r_{\omega,\gamma}^{0,-1}$.
\end{lemma}

\begin{proof}
$\{\phi_n\} \subset H_\text{rad}^1(\mathbb{R}^d)$ is a bounded sequence by the same argument with proof of Proposition \ref{Minimizer to r^{0,-1}}.
So, there exists $\phi_\infty \in H_\text{rad}^1(\mathbb{R}^d)$ such that
\begin{align*}
	\phi_n
		\xrightharpoonup[]{\hspace{0.4cm}} \phi_\infty\ \text{ in }\ H_\gamma^1(\mathbb{R}^d),\ \ \ 
	\phi_n
		\longrightarrow \phi_\infty\ \text{ in }\ L^{p+1}(\mathbb{R}^d)
\end{align*}
as $n \rightarrow \infty$ and
\begin{gather*}
	\|\phi_\infty\|_{L^2}
		\leq \liminf_{n \rightarrow \infty}\|\phi_n\|_{L^2},\\
	\|(-\Delta_\gamma)^\frac{1}{2}\phi_\infty\|_{L^2}
		\leq \liminf_{n \rightarrow \infty}\|(-\Delta_\gamma)^\frac{1}{2}\phi_n\|_{L^2}
		= d \cdot r_{\omega,\gamma}^{0,-1},\\
	\|\phi_\infty\|_{L^{p+1}}
		= \lim_{n \rightarrow \infty}\|\phi_n\|_{L^{p+1}}.
\end{gather*}
From
\begin{align*}
	K_{\omega,\gamma}^{0,-1}(\phi_\infty)
		& \leq \liminf_{n \rightarrow \infty}K_{\omega,\gamma}^{0,-1}(\phi_n) \\
		& = d \cdot \liminf_{n \rightarrow \infty}S_{\omega,\gamma}(\phi_n) - \liminf_{n \rightarrow \infty}\|(-\Delta_\gamma)^\frac{1}{2}\phi_n\|_{L^2}^2
		= d(r_0 - r_{\omega,\gamma}^{0,-1})
		\leq 0,
\end{align*}
there exists $\lambda \in (0,1]$ such that $K_{\omega,\gamma}^{0,-1}(\lambda\phi_\infty) = 0$.
For such $\lambda \in (0,1]$, we have
\begin{align*}
	r_{\omega,\gamma}^{0,-1}
		\leq S_{\omega,\gamma}(\lambda\phi_\infty)
		= \frac{1}{d}\|(-\Delta_\gamma)^\frac{1}{2}(\lambda\phi_\infty)\|_{L^2}^2
		\leq \frac{1}{d}\|(-\Delta_\gamma)^\frac{1}{2}\phi_\infty\|_{L^2}^2
		\leq r_{\omega,\gamma}^{0,-1}
\end{align*}
and hence, $\lambda$ must be $1$.
That is, we obtain $S_{\omega,\gamma}(\phi_\infty) = r_{\omega,\gamma}^{0,-1}$ and $K_{\omega,\gamma}^{0,-1}(\phi_\infty) = 0$.
We can also see $\phi_n \longrightarrow \phi_\infty$ in $H^1(\mathbb{R}^d)$.
Indeed, we have already seen $\|(-\Delta_\gamma)^\frac{1}{2}\phi_n\|_{L^2} \longrightarrow \|(-\Delta_\gamma)^\frac{1}{2}\phi_\infty\|_{L^2}$ as $n \rightarrow \infty$ and if $\|\phi_\infty\|_{L^2} < \liminf_{n \rightarrow \infty}\|\phi_n\|_{L^2}$, then
\begin{align*}
	r_{\omega,\gamma}^{0,-1}
		= S_{\omega,\gamma}(\phi_\infty)
		< \liminf_{n \rightarrow \infty}S_{\omega,\gamma}(\phi_n)
		= r_0
		\leq r_{\omega,\gamma}^{0,-1}.
\end{align*}.
\end{proof}

From now on, we consider only $\omega > 0$.
We can get the stability result for the case of $\omega < 0$ by the same manner.

\begin{lemma}\label{Strict convexity}
Let $d \geq 3$, $1 < p < 1 + \frac{4}{d}$, $\gamma > 0$, and $\omega_c < \omega < 1$.
Then, $r_{\omega,\gamma}^{0,-1}$ is strictly decreasing and strictly convex for $\omega$.
\end{lemma}

\begin{proof}
We take a $\phi_\omega \in \mathcal{M}_{\omega,\gamma,\text{rad}}^{0,-1}$.
From $\phi_\omega = (1 - \omega^2)^\frac{1}{p-1}\phi_0((1 - \omega^2)^\frac{1}{2}\,\cdot\,)$, we have
\begin{align*}
	r_{\omega,\gamma}^{0,-1}
		& = S_{\omega,\gamma}(\phi_\omega)
		= (1 - \omega^2)^{\frac{p+1}{p-1} - \frac{d}{2}}S_{0,\gamma}(\phi_0), \\
	\frac{d}{d\omega}r_{\omega,\gamma}^{0,-1}
		& = - \frac{d+2 - (d-2)p}{p-1}\omega(1 - \omega^2)^{\frac{2}{p-1} - \frac{d}{2}}S_{0,\gamma}(\phi_0)
		< 0, \\
	\frac{d^2}{d\omega^2}r_{\omega,\gamma}^{0,-1}
		& = \frac{d+2 - (d-2)p}{(p-1)^2}\{4 - (d-1)(p-1)\}(\omega^2 - \omega_c^2)(1 - \omega^2)^{\frac{2}{p-1} - \frac{d}{2} - 1}S_{0,\gamma}(\phi_0)
		> 0.
\end{align*}
\end{proof}

\begin{lemma}\label{Near the norm}
Let $d \geq 3$, $1 < p < 1 + \frac{4}{d}$, $\gamma > 0$, and $\omega_c < \omega_0 < 1$.
Let $(u_0,u_1) \in H_\text{rad}^1(\mathbb{R}^d) \times L_\text{rad}^2(\mathbb{R}^d)$.
Then, for any $\varepsilon > 0$, there exists $\delta = \delta(\varepsilon) > 0$ such that if
\begin{align*}
	\inf_{Q_{\omega_0,\gamma} \in \mathcal{G}_{\omega_0,\gamma,\text{rad}}}\|(u_0,u_1) - (Q_{\omega_0,\gamma},i\omega Q_{\omega_0,\gamma})\|_{H^1 \times L^2}
		< \delta(\varepsilon),
\end{align*}
then the solution $u$ to \eqref{NLKG} with \eqref{IC} satisfies $(u(t),\partial_tu(t)) \in \mathcal{R}_{\omega_0-\varepsilon,+}^{0,-1} \cap \mathcal{R}_{\omega_0+\varepsilon,-}^{0,-1}$ for any $t \in [0,T_\text{max})$.
\end{lemma}

\begin{proof}
Fix $\varepsilon > 0$.
Since $r_{\omega_0 + \varepsilon,\gamma}^{0,-1} < r_{\omega_0,\gamma}^{0,-1} = \frac{1}{d}\|(-\Delta_\gamma)^\frac{1}{2}Q_{\omega_0,\gamma}\|_{L^2}^2 < r_{\omega_0 - \varepsilon,\gamma}^{0,-1}$ holds for any $Q_{\omega_0,\gamma} \in \mathcal{G}_{\omega_0,\gamma,\text{rad}}$, there exists $\delta_1 > 0$ such that if $\delta_1 > 0$ satisfies
\begin{align*}
	\inf_{Q_{\omega_0,\gamma} \in \mathcal{G}_{\omega_0,\gamma,\text{rad}}}\|(u_0,u_1) - (Q_{\omega_0,\gamma},i\omega_0Q_{\omega_0,\gamma})\|_{H^1 \times L^2}
		< \delta_1,
\end{align*}
then
\begin{align*}
	r_{\omega_0 + \varepsilon,\gamma}^{0,-1}
		< \frac{1}{d}\|(-\Delta_\gamma)^\frac{1}{2}u_0\|_{L^2}^2
		< r_{\omega_0 - \varepsilon,\gamma}^{0,-1}.
\end{align*}
\eqref{112} and Lemma \ref{Strict convexity} deduce that
\begin{align*}
	L_{\omega_0 \pm \varepsilon,\gamma}(Q_{\omega_0,\gamma},i\omega_0Q_{\omega_0,\gamma})
		& = L_{\omega_0,\gamma}(Q_{\omega_0,\gamma},i\omega_0Q_{\omega_0,\gamma}) \pm \varepsilon C(Q_{\omega_0,\gamma},i\omega_0Q_{\omega_0,\gamma}) \\
		& = r_{\omega_0,\gamma}^{0,-1} \pm \varepsilon\left.\frac{d}{d\omega}\right|_{\omega = \omega_0}r_{\omega,\gamma}^{0,-1}
		< r_{\omega_0 \pm \varepsilon,\gamma}^{0,-1}
\end{align*}
for any $Q_{\omega_0,\gamma} \in \mathcal{G}_{\omega_0,\gamma,\text{rad}}$.
Thus, there exists $\delta_2 > 0$ such that if
\begin{align*}
	\inf_{Q_{\omega_0,\gamma} \in \mathcal{G}_{\omega_0,\gamma,\text{rad}}}\|(u_0,u_1) - (Q_{\omega_0,\gamma},i\omega_0Q_{\omega_0,\gamma})\|_{H^1 \times L^2}
		< \delta_2,
\end{align*}
then
\begin{align*}
	L_{\omega_0 \pm \varepsilon,\gamma}(u_0,u_1)
		< r_{\omega_0 \pm \varepsilon,\gamma}^{0,-1}.
\end{align*}
Therefore, if we set $\delta = \min\{\delta_1,\delta_2\}$ and
\begin{align*}
	\inf_{Q_{\omega_0,\gamma} \in \mathcal{G}_{\omega_0,\gamma,\text{rad}}}\|(u_0,u_1) - (Q_{\omega_0,\gamma},i\omega_0Q_{\omega_0,\gamma})\|_{H^1 \times L^2}
		< \delta,
\end{align*}
then $(u_0,u_1) \in \mathcal{R}_{\omega_0 - \varepsilon,+}^{0,-1} \cap \mathcal{R}_{\omega_0 + \varepsilon,-}^{0,-1}$ from Lemma \ref{Invariant set}.
Lemma \ref{Estimate of virial} implies the desired result.
\end{proof}

\begin{proof}[Proof of stability in Theorem \ref{Instability versus stability}]
We prove the theorem by contradiction.
That is, we assume that there exist a positive constant $\varepsilon_0 > 0$, an initial data sequence $\{(u_{0,n},u_{1,n})\} \subset H_\text{rad}^1(\mathbb{R}^d) \times L_\text{rad}^2(\mathbb{R}^d)$, parameters $\{\varepsilon_n\} \subset (0,1)$, and a time sequence $\{t_n\} \subset [0,\infty)$ such that $\varepsilon_n \longrightarrow 0$ as $n \rightarrow \infty$, $\omega_c < |\omega_0 \pm \varepsilon_n| < 1$ for each $n \in \mathbb{N}$,
\begin{align*}
	\inf_{Q_{\omega_0,\gamma} \in \mathcal{G}_{\omega_0,\gamma,\text{rad}}}\|(u_{0,n},u_{1,n}) - (Q_{\omega_0,\gamma},i\omega Q_{\omega_0,\gamma})\|_{H^1 \times L^2}
		< \delta(\varepsilon_n),
\end{align*}
and
\begin{align}
	\inf_{Q_{\omega_0,\gamma} \in \mathcal{G}_{\omega_0,\gamma,\text{rad}}}\|(u_n(t_n),\partial_tu_n(t_n)) - (Q_{\omega_0,\gamma},i\omega Q_{\omega_0,\gamma})\|_{H^1 \times L^2}
		\geq \varepsilon_0, \label{140}
\end{align}
where $\delta$ is given in Lemma \ref{Near the norm} and $(u_n,\partial_tu_n)$ is a solution to \eqref{NLKG} with initial data $(u_{0,n},u_{1,n})$.
From Lemma \ref{Near the norm}, we have $r_{\omega_0 + \varepsilon_n,\gamma}^{0,-1} \leq \frac{1}{d}\|(-\Delta_\gamma)^\frac{1}{2}u_n(t_n)\|_{L^2}^2 \leq r_{\omega_0 - \varepsilon_n,\gamma}^{0,-1}$.
By the continuity of $r_{\omega,\gamma}^{0,-1}$,
\begin{align*}
	\lim_{n \rightarrow \infty}\frac{1}{d}\|(-\Delta_\gamma)^\frac{1}{2}u_n(t_n)\|_{L^2}^2
		= r_{\omega_0,\gamma}^{0,-1}
\end{align*}
holds.
Applying the same argument with the proof of Proposition \ref{Minimizer to r^{0,-1}}, we have boundedness of $\|u_n(t_n)\|_{L^2}^2$, where we note that $(u_n(t_n),\partial_tu_n(t_n)) \in \mathcal{R}_{\omega+\varepsilon_n,-}^{0,-1}$ from Lemma \ref{Invariant set}.
\begin{align*}
	S_{\omega_0 \pm \varepsilon_n,\gamma}(u_n(t_n))
		\leq L_{\omega_0 \pm \varepsilon_n,\gamma}(u_n(t_n),\partial_tu_n(t_n))
		< r_{\omega_0 \pm \varepsilon_n,\gamma}^{0,-1}
\end{align*}
implies that
\begin{align*}
	S_{\omega_0,\gamma}(u_n(t_n))
		= S_{\omega_0 \pm \varepsilon_n,\gamma}(u_n(t_n)) + \frac{\pm 2\varepsilon_n\omega_0 + \varepsilon_n^2}{2}\|u_n(t_n)\|_{L^2}^2
		\longrightarrow r_0
		\leq r_{\omega_0,\gamma}^{0,-1}
\end{align*}
as $n \rightarrow \infty$.
It follows from Lemma \ref{convergence to minimizer} that $r_0 = r_{\omega,\gamma}^{0,-1}$ and $u_n(t_n) \longrightarrow \phi_{\omega_0}$ in $H^1(\mathbb{R}^d)$ as $n \rightarrow \infty$ for some $\phi_{\omega_0} \in \mathcal{G}_{\omega_0,\gamma,\text{rad}}$.
\begin{align*}
	\frac{1}{2}\|\partial_tu_n(t_n) - i(\omega_0 \pm \varepsilon_n)u_n(t_n)\|_{L_x^2}^2
		& = L_{\omega_0 \pm \varepsilon_n,\gamma}(u_n(t_n),\partial_tu_n(t_n)) - S_{\omega_0 \pm \varepsilon_n,\gamma}(u_n(t_n)) \\
		& < r_{\omega_0 \pm \varepsilon_n,\gamma}^{0,-1} - S_{\omega_0 \pm \varepsilon_n,\gamma}(u_n(t_n))
\end{align*}
deduces that
\begin{align*}
	\lim_{n \rightarrow \infty}\|\partial_tu_n(t_n) - i(\omega_0 \pm \varepsilon_n)u_n(t_n)\|_{L_x^2}^2
		\leq r_{\omega_0,\gamma}^{0,-1} - r_{\omega_0,\gamma}^{0,-1}
		= 0.
\end{align*}
On the other hand, we have
\begin{align*}
	\lim_{n \rightarrow \infty}\|\partial_tu_n(t_n) - i(\omega_0 \pm \varepsilon_n)u_n(t_n)\|_{L_x^2}
		= \lim_{n \rightarrow \infty}\|\partial_tu_n(t_n) - i\omega_0 \phi_{\omega_0}\|_{L_x^2}.
\end{align*}
Therefore, we obtain
\begin{align*}
	\lim_{n \rightarrow \infty}\|(u_n(t_n),\partial_tu_n(t_n)) - (\phi_{\omega_0},i\omega_0 \phi_{\omega_0})\|_{H^1 \times L^2}
		= 0.
\end{align*}
This contradicts \eqref{140}.
\end{proof}

\section{Proof of very strong instability}\label{Proof of very strong instability}

Throughout this section, we assume local well-posedness of \eqref{NLKG} with \eqref{IC} for $1 + \frac{2}{d-2} < p < 1 + \frac{4}{d+1}$.
In this section, we prove the very strong instability results in Theorem \ref{Instability versus stability}.
The proof is based on the argument in \cite{OhtTod07}.
If $(u,\partial_tu)$ is a solution to \eqref{NLKG}, then the following identities hold formally:
\begin{gather*}
	- \frac{d}{dt}\text{Re}\int_{\mathbb{R}^d}\{2x\cdot\nabla u(t,x) + d\cdot u(t,x)\}\overline{\partial_tu(t,x)}dx
		= K_{\omega,\gamma}^{d,2}(u(t)), \\
	- \frac{d}{dt}\text{Re}\int_{\mathbb{R}^d}\{2x\cdot\nabla u(t,x) + (d+q) \cdot u(t,x)\}\overline{\partial_tu(t,x)}dx
		= K(u(t)),
\end{gather*}
where $q = \frac{4}{p-1} - d$ and
\begin{gather*}
	K(u(t))
		:= K_{\omega,\gamma}^{d,2}(u(t)) + q \cdot K_1(u(t),\partial_tu(t)), \\
	K_1(u,v)
		:= \|u\|_{L^2}^2 + \|(-\Delta_\gamma)^\frac{1}{2}u\|_{L^2}^2 - \|u\|_{L^{p+1}}^{p+1} - \|v\|_{L^2}^2.
\end{gather*}
In the proof, we use the localized functionals of the functionals.
We take weight functions as \cite{OhtTod07}.
Let $\Phi \in C^2([0,\infty))$ be a nonnegative function satisfying
\begin{equation*}
\Phi(r) =
\begin{cases}
&\hspace{-0.4cm}\displaystyle{d \qquad \text{for} \quad 0 \leq r \leq 1,}\\
&\hspace{-0.4cm}\displaystyle{0 \qquad \text{for} \quad r \geq 2,}
\end{cases}
\qquad \Phi' \leq 0\ \text{ for }\ 1 \leq r \leq 2.
\end{equation*}
For $R > 0$, we define functions $\Phi_R$ and $\Psi_R$ as
\begin{align*}
	\Phi_R(r)
		:= \Phi\left(\frac{r}{R}\right),\ \ \ \ 
	\Psi_R(r)
		:= \frac{1}{r^{d-1}}\int_0^r s^{d-1}\Phi_R(s)ds.
\end{align*}

\begin{lemma}[Ohta--Todorova, \cite{OhtTod07}]\label{Properties of weighted functions}
For $R > 0$, the followings hold:
\begin{gather*}
	\Phi_R(r)
		= d,\ \ \ 
	\Psi_R(r)
		= r,\ \ \ 0 \leq r \leq R, \\
	\Psi_R'(r) + \frac{d-1}{r}\Psi_R(r)
		= \Phi_R(r),\ \ \ r \geq 0, \\
	|\Phi_R^{(k)}(r)|
		\lesssim \frac{1}{R^k},\ \ \ r \geq 0,\ k = 0, 1, 2, \\
	\Psi_R'(r)
		\leq 1,\ \ \ r \geq 0.
\end{gather*}
\end{lemma}

\begin{lemma}\label{Estimate for virial functional}
Let $(u,\partial_t u) \in C_t(I;H_x^1(\mathbb{R}^d) \times L_x^2(\mathbb{R}^d))$ be a radially symmetric solution to \eqref{NLKG}.
Then, there exists $C_0 > 0$ such that
\begin{align*}
	- \frac{d}{dt}I_R^1(t)
		& \leq K_{\omega,\gamma}^{d,2}(u(t)) + \frac{d(p-1)}{p+1}\|u(t)\|_{L^{p+1}(|x| \geq R)}^{p+1} + \frac{C_0}{R^2}\|u(t)\|_{L_x^2}^2, \\
	- \frac{d}{dt}I_R^2(t)
		& \leq K(u(t),\partial_tu(t)) + \frac{d(p-1)}{p+1}\|u(t)\|_{L^{p+1}(|x| \geq R)}^{p+1} + \frac{C_0}{R^2}\|u(t)\|_{L_x^2}^2,
\end{align*}
where $I_R^1$ and $I_R^2$ are defined as
\begin{align*}
	I_R^1(t)
		& := 2\text{Re}\int_{\mathbb{R}^d}\Psi_R(x)\partial_ru(t,x)\overline{\partial_tu(t,x)}dx + \text{Re}\int_{\mathbb{R}^d}\Phi_R(x)u(t,x)\overline{\partial_tu(t,x)}dx, \\
	I_R^2(t)
		& := I_R^1(t) + \left(\frac{4}{p-1} - d\right)\text{Re}\int_{\mathbb{R}^d}u(t,x)\overline{\partial_tu(t,x)}dx
\end{align*}
for any $t \in [0,T_\text{max})$.
\end{lemma}

\begin{proof}
Multiplying \eqref{NLKG} by $\Psi_R\overline{\partial_ru}$ and $\Phi_R\overline{u}$ respectively, we have
\begin{gather}
	- \Psi_R\overline{\partial_ru}\partial_t^2u + \Psi_R\overline{\partial_ru}\Delta_\gamma u - u\Psi_R\overline{\partial_ru}
		= - |u|^{p-1}u\Psi_R\overline{\partial_ru}, \label{106} \\
	- \Phi_R\overline{u}\partial_t^2u + \Phi_R\overline{u}\Delta_\gamma u - \Phi_R|u|^2
		= - \Phi_R|u|^{p+1}. \label{107}
\end{gather}
Integrating \eqref{106} and \eqref{107} respectively, we have
\begin{align*}
	- 2\frac{d}{dt}\text{Re}\int_{\mathbb{R}^d}\Psi_R\partial_ru\overline{\partial_tu}dx
		& = - 2\text{Re}\int_{\mathbb{R}^d}\Psi_R\partial_r\partial_tu\overline{\partial_tu}dx - 2\text{Re}\int_{\mathbb{R}^d}\Psi_R\partial_ru\overline{\partial_t^2u}dx \\
		& \hspace{-1.5cm} = 2\text{Re}\int_{\mathbb{R}^d}\left(\Psi_R'|\partial_tu|^2 + \Psi_R\partial_tu\overline{\partial_r\partial_tu} + \frac{d-1}{r}\Psi_R|\partial_tu|^2\right)dx \\
		& - 2\text{Re}\int_{\mathbb{R}^d}(\Psi_R\overline{\partial_ru}\Delta_\gamma u - u\Psi_R\overline{\partial_ru} + \Psi_R|u|^{p-1}u\overline{\partial_ru})dx \\
		& \hspace{-1.5cm} = \int_{\mathbb{R}^d}\left(\Psi_R' + \frac{d-1}{r}\Psi_R\right)|\partial_tu|^2dx + \int_{\mathbb{R}^d}\left(\Psi_R' - \frac{d-1}{r}\Psi_R\right)|\nabla u|^2dx \\
		& - \int_{\mathbb{R}^d}\left(\Psi_R'\frac{\gamma}{|x|^2} + \Psi_R\frac{(d-3)\gamma}{|x|^3}\right)|u|^2dx - \int_{\mathbb{R}^d}\left(\Psi_R' + \frac{d-1}{r}\Psi_R\right)|u|^2dx \\
		& + \frac{2}{p+1}\int_{\mathbb{R}^d}\left(\Psi_R' + \frac{d-1}{r}\Psi_R\right)|u|^{p+1}dx
\end{align*}
and
\begin{align*}
	- \frac{d}{dt}\text{Re}\int_{\mathbb{R}^d}\Phi_Ru\overline{\partial_tu}dx
		& = - \int_{\mathbb{R}^d}\Phi_R|\partial_tu|^2dx - \text{Re}\int_{\mathbb{R}^d}\Phi_Ru\overline{\partial_t^2u}dx \\
		& = - \int_{\mathbb{R}^d}\Phi_R|\partial_tu|^2dx - \text{Re}\int_{\mathbb{R}^d}(\Phi_R\overline{u}\Delta_\gamma u - \Phi_R|u|^2 + \Phi_R|u|^{p+1})dx \\
		& = - \int_{\mathbb{R}^d}\Phi_R|\partial_tu|^2dx - \frac{1}{2}\int_{\mathbb{R}^d}\Delta\Phi_R|u|^2dx + \int_{\mathbb{R}^d}\Phi_R|\nabla u|^2dx \\
		& \hspace{1.5cm} + \int_{\mathbb{R}^d}\frac{\gamma}{|x|^2}\Phi_R|u|^2dx + \int_{\mathbb{R}^d}\Phi_R|u|^2dx - \int_{\mathbb{R}^d}\Phi_R|u|^{p+1}dx.
\end{align*}
From Lemma \ref{Properties of weighted functions}, we have
\begin{align*}
	- \frac{d}{dt}I_R^1(t)
		& = - 2\frac{d}{dt}\text{Re}\int_{\mathbb{R}^d}\Psi_R\partial_ru\overline{\partial_tu}dx - \frac{d}{dt}\text{Re}\int_{\mathbb{R}^d}\Phi_Ru\overline{\partial_tu}dx \\
		& = \int_{\mathbb{R}^d}\left(\Psi_R' + \frac{d-1}{r}\Psi_R\right)|\partial_tu|^2dx + \int_{\mathbb{R}^d}\left(\Psi_R' - \frac{d-1}{r}\Psi_R\right)|\nabla u|^2dx \\
		& \hspace{1.0cm} - \int_{\mathbb{R}^d}\left(\Psi_R'\frac{\gamma}{|x|^2} + \Psi_R\frac{(d-3)\gamma}{|x|^3}\right)|u|^2dx - \int_{\mathbb{R}^d}\left(\Psi_R' + \frac{d-1}{r}\Psi_R\right)|u|^2dx \\
		& \hspace{1.0cm} + \frac{2}{p+1}\int_{\mathbb{R}^d}\left(\Psi_R' + \frac{d-1}{r}\Psi_R\right)|u|^{p+1}dx - \int_{\mathbb{R}^d}\Phi_R|\partial_tu|^2dx - \frac{1}{2}\int_{\mathbb{R}^d}\Delta\Phi_R|u|^2dx \\
		& \hspace{1.0cm} + \int_{\mathbb{R}^d}\Phi_R|\nabla u|^2dx + \int_{\mathbb{R}^d}\frac{\gamma}{|x|^2}\Phi_R|u|^2dx + \int_{\mathbb{R}^d}\Phi_R|u|^2dx - \int_{\mathbb{R}^d}\Phi_R|u|^{p+1}dx \\
		& = 2\int_{\mathbb{R}^d}\Psi_R'|\nabla u|^2dx - \frac{1}{2}\int_{\mathbb{R}^d}\Delta\Phi_R|u|^2dx + \int_{\mathbb{R}^d}\frac{2\gamma}{|x|^3}\Psi_R|u|^2dx - \frac{p-1}{p+1}\int_{\mathbb{R}^d}\Phi_R|u|^{p+1}dx \\
		& = K_{\omega,\gamma}^{d,2}(u) + 2\int_{|x| \geq R}(\Psi_R'-1)|\nabla u|^2dx - \frac{1}{2}\int_{\mathbb{R}^d}\Delta\Phi_R|u|^2dx \\
		& \hspace{1.0cm} + \int_{\mathbb{R}^d}\left(\frac{2}{|x|}\Psi_R - 2\right)\frac{\gamma}{|x|^2}|u|^2dx - \frac{p-1}{p+1}\int_{|x| \geq R}(\Phi_R-d)|u|^{p+1}dx \\
		& \leq K_{\omega,\gamma}^{d,2}(u) + \frac{d(p-1)}{p+1}\|u\|_{L^{p+1}(|x| \geq R)}^{p+1} + \frac{C_0}{R^2}\|u\|_{L^2}^2.
\end{align*}
\begin{align*}
	- \frac{d}{dt}I_R^2(t)
		& = - \frac{d}{dt}I_R^1(t) - \left(\frac{4}{p-1} - d\right)\frac{d}{dt}\text{Re}\int_{\mathbb{R}^d}u(t,x)\overline{\partial_tu(t,x)}dx \\
		& = - \frac{d}{dt}I_R^1(t) - \left(\frac{4}{p-1} - d\right)\left\{\|\partial_tu\|_{L^2}^2 + \text{Re}\int_{\mathbb{R}^d}u\overline{\partial_t^2u}dx\right\} \\
		& = - \frac{d}{dt}I_R^1(t) - \left(\frac{4}{p-1} - d\right)\left\{\|\partial_tu\|_{L^2}^2 + \text{Re}\int_{\mathbb{R}^d}\overline{u}(\Delta_\gamma u - u + |u|^{p-1}u)dx\right\} \\
		& = - \frac{d}{dt}I_R^1(t) - \left(\frac{4}{p-1} - d\right)(\|\partial_tu\|_{L^2}^2 - \|(-\Delta_\gamma)^\frac{1}{2}u\|_{L^2}^2 - \|u\|_{L^2}^2 + \|u\|_{L^{p+1}}^{p+1}) \\
		& = - \frac{d}{dt}I_R^1(t) + \left(\frac{4}{p-1} - d\right)K_1(u,\partial_tu) \\
		& \leq K(u,\partial_tu) + \frac{d(p-1)}{p+1}\|u\|_{L^{p+1}(|x| \geq R)}^{p+1} + \frac{C_0}{R^2}\|u\|_{L^2}^2.
\end{align*}
\end{proof}

We consider the initial data $(\lambda Q_{\omega,\gamma},\lambda i\omega Q_{\omega,\gamma})$ for $\lambda > 1$.
Let $(u_\lambda,\partial_tu_\lambda)$ be the solution to \eqref{NLKG} with the initial data $(\lambda Q_{\omega,\gamma},\lambda i\omega Q_{\omega,\gamma})$.
Let $[0,T_\lambda)$ be the positive maximal life-span of $(u_\lambda,\partial_tu_\lambda)$.\\

\noindent
8.1. \textbf{Case $1 < p < 1 + \frac{4}{d}$.}

\begin{proof}[Proof of very strong instability for $1 < p < 1 + \frac{4}{d}$ in Theorem \ref{Instability versus stability}]
For any $\varepsilon > 0$, we take $\lambda_0 > 1$ satisfying
\begin{align*}
	\|(\lambda_0Q_{\omega,\gamma},\lambda_0i\omega Q_{\omega,\gamma}) - (Q_{\omega,\gamma},i\omega Q_{\omega,\gamma})\|_{H^1 \times L^2}
		= \varepsilon.
\end{align*}
We consider initial data $(\lambda Q_{\omega,\gamma}, \lambda i\omega Q_{\omega,\gamma})$ for $\lambda \in (1,\lambda_0)$.

First, we prove very strong instability of $\mathcal{G}_{\omega,\gamma,\text{rad}}$.
Set
\begin{align}
	\delta_1
		& := (q+2)\left\{r_{\omega,\gamma}^{2,p-1} - L_{\omega,\gamma}(\lambda Q_{\omega,\gamma},\lambda i\omega Q_{\omega,\gamma})\right\}, \label{136} \\
	\delta_2
		& := q\left\{- \omega C(\lambda Q_{\omega,\gamma},\lambda i\omega Q_{\omega,\gamma}) - \frac{(q+2)\omega^2}{1 - \omega^2}r_{\omega,\gamma}^{2,p-1}\right\}, \label{137}
\end{align}
and $\delta = \delta_1 + \delta_2$.
We note that $\delta_1 > 0$ by Lemma \ref{Estimates for the ground state}.
From Proposition \ref{Characterization of the ground state in mass-subcritical 1}, we have
\begin{align*}
	\frac{(q+2)\omega^2}{1 - \omega^2}r_{\omega,\gamma}^{2,p-1}
		= \omega^2\|Q_{\omega,\gamma}\|_{L^2}^2
		< \lambda^2\omega^2\|Q_{\omega,\gamma}\|_{L^2}^2
		= - \omega C(\lambda Q_{\omega,\gamma},\lambda i\omega Q_{\omega,\gamma}),
\end{align*}
that is, $\delta_2 > 0$ and $\delta > 0$ hold.
We assume for contradiction that there exists $\lambda_1 \in (1,\lambda_0)$ such that a solution $(u_{\lambda_1},\partial_tu_{\lambda_1})$ to \eqref{NLKG} with data $(\lambda_1 Q_{\omega,\gamma},\lambda_1 i\omega Q_{\omega,\gamma})$ satisfies
\begin{align}
	M_1
		:= \sup_{t \geq 0}\|(u_{\lambda_1}(t),\partial_tu_{\lambda_1}(t))\|_{H_x^1 \times L_x^2}
		< \infty. \label{134}
\end{align}
From Lemma \ref{Radial Sobolev inequality} and \eqref{134},
\begin{align*}
	\|u_{\lambda_1}(t)\|_{L^{p+1}(|x| \geq R)}^{p+1}
		\leq CR^{-\frac{(d-1)(p-1)}{2}}\|u_{\lambda_1}(t)\|_{L^2(|x| \geq R)}^\frac{p+3}{2}\|\nabla u_{\lambda_1}(t)\|_{L^2(|x| \geq R)}^\frac{p-1}{2}
		\leq CM_1^{p+1}R^{-\frac{(d-1)(p-1)}{2}}
\end{align*}
for any $t \geq 0$ and $R > 0$.
We take a positive constant $R_0$ such that
\begin{align}
	\sup_{t \geq 0}\left\{\frac{d(p-1)}{p+1}\|u_{\lambda_1}(t)\|_{L^{p+1}(|x| \geq R_0)}^{p+1} + \frac{C_0}{R_0^2}\|u_{\lambda_1}(t)\|_{L^2}^2\right\}
		< \delta. \label{135}
\end{align}
Since $(u_{\lambda_1},\partial_tu_{\lambda_1})$ is radially symmetric with respect to $x$ for any $t \geq 0$, we have
\begin{align*}
	\frac{d}{dt}I_{R_0}^2(t)
		& \geq - K(u_{\lambda_1}(t),\partial_tu_{\lambda_1}(t)) - \frac{d(p-1)}{p+1}\|u_{\lambda_1}(t)\|_{L^{p+1}(|x| \geq R_0)}^{p+1} - \frac{C_0}{R_0^2}\|u_{\lambda_1}(t)\|_{L^2}^2 \\
		& \geq - K(u_{\lambda_1}(t),\partial_tu_{\lambda_1}(t)) - \delta \\
		& \geq - 2(q+2)L_{\omega,\gamma}(u_{\lambda_1}(t),\partial_tu_{\lambda_1}(t)) - 2q\omega C(u_{\lambda_1}(t),\partial_tu_{\lambda_1}(t)) \\
		& \hspace{2.5cm} + 2\{1 - (q+1)\omega^2\}\|u_{\lambda_1}(t)\|_{L^2}^2 - \delta \\
		& \geq - 2(q+2)L_{\omega,\gamma}(\lambda Q_{\omega,\gamma},\lambda i\omega Q_{\omega,\gamma}) - 2q\omega C(\lambda Q_{\omega,\gamma},\lambda i\omega Q_{\omega,\gamma}) \\
		& \hspace{2.5cm} + \frac{2(q+2)\{1 - (q+1)\omega^2\}}{1 - \omega^2}r_{\omega,\gamma}^{2,p-1} - \delta \\
		& = 2\delta - \delta
		= \delta
		> 0
\end{align*}
for any $t \geq 0$ by using \eqref{135}, \eqref{136}, and \eqref{137}, which implies $\lim_{t \rightarrow \infty}I_R^2(t) = \infty$.
On the other hand,
\begin{align*}
	I_{R_0}^2(t)
		\lesssim I_{R_0}^1(t) + \|u_{\lambda_1}(t)\|_{L^2}\|\partial_tu_{\lambda_1}(t)\|_{L^2}
		\lesssim M_1^2
		< \infty
\end{align*}
for any $t \geq 0$.
This is contradiction.
Therefore, for any $\lambda \in (1,\lambda_0)$, the solution $(u,\partial_tu)$ to \eqref{NLKG} with the data $(\lambda Q_{\omega,\gamma},\lambda i\omega Q_{\omega,\gamma})$ either blows up in finite time or ``exists globally in time and $\limsup_{t \rightarrow \infty}\|(u(t),\partial_tu(t))\|_{H_x^1 \times L_x^2} = \infty$''.
Theorem \ref{Global implies boundedness} deduces very strong instability of $\mathcal{G}_{\omega,\gamma,\text{rad}}$.

Next, we prove very strong instability of $\mathcal{A}_{\omega,\gamma,\text{rad}}$.
Let $|\omega| = \omega_c$.
We set
\begin{align*}
	\delta
		= - q\omega_cC(\lambda Q_{{\omega_c},\gamma},\lambda i\omega_cQ_{{\omega_c},\gamma}) - (q+2)(E + \omega_cC)(\lambda Q_{{\omega_c},\gamma},\lambda i\omega_cQ_{{\omega_c},\gamma}).
\end{align*}
From $S_{\omega_c,\gamma}'(Q_{{\omega_c},\gamma}) = 0$, we have $(E + \omega_cC)(\lambda Q_{{\omega_c},\gamma},\lambda i\omega_cQ_{{\omega_c},\gamma}) = S_{\omega_c}(\lambda Q_{\omega_c,\gamma}) < S_{\omega_c}(Q_{\omega_c,\gamma})$ for $\lambda > 1$.
$- C(\lambda Q_{{\omega_c},\gamma},\lambda i\omega_cQ_{{\omega_c},\gamma}) = \omega_c\lambda^2\|Q_{\omega_c,\gamma}\|_{L^2}^2 > \omega_c\|Q_{\omega_c,\gamma}\|_{L^2}^2$ holds for $\lambda > 1$.
Therefore, it follows that
\begin{align*}
	\delta
		> q\omega_c^2\|Q_{\omega_c,\gamma}\|_{L^2}^2 - (q+2)S_{\omega_c}(Q_{\omega_c,\gamma})
		= - \frac{1}{p-1}K_{\omega_c,\gamma}^{2,p-1}(Q_{\omega_c,\gamma})
		= 0.
\end{align*}
Here, we assume for contradiction that there exists $\lambda \in (1,\lambda_0)$ such that the solution $(u_\lambda,\partial_tu_\lambda)$ to \eqref{NLKG} with initial data $(\lambda Q_{{\omega_c},\gamma},\lambda i\omega_cQ_{{\omega_c},\gamma})$ satisfies \eqref{134}.
By the same argument with $\mathcal{G}_{\omega,\gamma,\text{rad}}$ in this proof, we obtain $\frac{d}{dt}I_R^2(t) \geq \delta > 0$ for sufficiently large $R > 0$.
This is contradiction.
\end{proof}

\noindent
8.2. \textbf{Case $1 + \frac{4}{d} \leq p < 1 + \frac{4}{d-2}$.}

\begin{proof}[Proof of very strong instability for $1 + \frac{4}{d} \leq p < 1 + \frac{4}{d-2}$ in Theorem \ref{Instability versus stability}]
For any $\varepsilon > 0$, we take $\lambda_0 > 1$ satisfying
\begin{align*}
	\|(\lambda_0 Q_{\omega,\gamma},\lambda_0 i\omega Q_{\omega,\gamma}) - (Q_{\omega,\gamma},i\omega  Q_{\omega,\gamma})\|_{H^1 \times L^2}
		= \varepsilon.
\end{align*}
We consider initial data $(\lambda Q_{\omega,\gamma}, \lambda i\omega Q_{\omega,\gamma})$ for $\lambda \in (1,\lambda_0)$.
Set
\begin{align}
	\delta
		:= \frac{d(p-1)}{2}\left\{r_{\omega,\gamma}^{d,2} - L_{\omega,\gamma}(\lambda Q_{\omega,\gamma},\lambda i\omega Q_{\omega,\gamma})\right\}. \label{105}
\end{align}
We note that $\delta > 0$ by Lemma \ref{Estimates for the ground state} (2).
We assume for contradiction that there exists $\lambda_1 \in (1,\lambda_0)$ such that a solution $(u_{\lambda_1},\partial_tu_{\lambda_1})$ to \eqref{NLKG} with data $(\lambda_1 Q_{\omega,\gamma},\lambda_1 i\omega Q_{\omega,\gamma})$ satisfies
\begin{align}
	M_1
		:= \sup_{t \geq 0}\|(u_{\lambda_1}(t),\partial_tu_{\lambda_1}(t))\|_{H_x^1 \times L_x^2}
		< \infty. \label{102}
\end{align}
From Lemma \ref{Radial Sobolev inequality} and \eqref{102},
\begin{align*}
	\|u_{\lambda_1}(t)\|_{L^{p+1}(|x| \geq R)}^{p+1}
		\leq CR^{-\frac{(d-1)(p-1)}{2}}\|u_{\lambda_1}(t)\|_{L^2(|x| \geq R)}^\frac{p+3}{2}\|\nabla u_{\lambda_1}(t)\|_{L^2(|x| \geq R)}^\frac{p-1}{2}
		\leq CM_1^{p+1}R^{-\frac{(d-1)(p-1)}{2}}
\end{align*}
for any $t \geq 0$ and $R > 0$.
We take a positive constant $R_0$ such that
\begin{align}
	\sup_{t \geq 0}\left\{\frac{d(p-1)}{p+1}\|u_{\lambda_1}(t)\|_{L^{p+1}(|x| \geq R_0)}^{p+1} + \frac{C_0}{R_0^2}\|u_{\lambda_1}(t)\|_{L^2}^2\right\}
		< \delta. \label{104}
\end{align}
Since $(u_{\lambda_1},\partial_tu_{\lambda_1})$ is radially symmetric with respect to $x$ for any $t \geq 0$, we have
\begin{align*}
	\frac{d}{dt}I_R^1(t)
		& \geq - K_{\omega,\gamma}^{d,2}(u_{\lambda_1}(t)) - \frac{d(p-1)}{p+1}\|u_{\lambda_1}(t)\|_{L^{p+1}(|x| \geq R)}^{p+1} - \frac{C_0}{R^2}\|u_{\lambda_1}(t)\|_{L^2}^2 \\
		& \geq 2\delta - \delta
		= \delta
		> 0
\end{align*}
for any $t \geq 0$ by using Lemmas \ref{Estimate for virial functional}, \ref{Invariant set}, \ref{Estimate of virial}, \eqref{105}, and \eqref{104}.
This inequality implies $\lim_{t \rightarrow \infty}I_R^1(t) = \infty$.
On the other hand,
\begin{align*}
	I_{R_0}^1(t)
		\lesssim \|\Psi_{R_0}\|_{L^\infty}\|\nabla u_{\lambda_1}(t)\|_{L_x^2}\|\partial_tu_{\lambda_1}(t)\|_{L_x^2} + \|\Phi_{R_0}\|_{L^\infty}\|u_{\lambda_1}(t)\|_{L_x^2}\|\partial_tu_{\lambda_1}(t)\|_{L_x^2}
		\lesssim M_1^2
		< \infty
\end{align*}
for any $t \geq 0$.
This is contradiction.
Therefore, for any $\lambda \in (1,\lambda_0)$, the solution $(u,\partial_tu)$ to \eqref{NLKG} with the data $(\lambda Q_{\omega,\gamma},\lambda i\omega Q_{\omega,\gamma})$ either blows up in finite time or ``exists globally in time and $\limsup_{t \rightarrow \infty}\|(u,\partial_tu)\|_{H^1 \times L^2} = \infty$''.
Proposition \ref{Uniform estimate of global solution} implies that there exists a positive constant $J$ such that
\begin{align*}
	\int_T^{T+1}\int_{\mathbb{R}^d}(|u(t)|^2+|\nabla u(t)|^2+|\partial_tu(t)|^2)dxdt
		\leq J
\end{align*}
for each $T$.
By the mean value theorem, for any $j \in \mathbb{N} \cup \{0\}$, there exists $T_j \in [j,j+1]$ such that
\begin{align*}
	\int_{\mathbb{R}^d}(|u(T_j)|^2 + |\nabla u(T_j)|^2 + |\partial_tu(T_j)|^2)dx
		\leq J.
\end{align*}
We define a function $J_R(t)$ as
\begin{align*}
	\frac{d}{dt}I_R^1(t)
		\geq 2\delta - \frac{d(p-1)}{p+1}\|u(t)\|_{L^{p+1}(|x| \geq R)}^{p+1} - \frac{C}{R^2}\|u(t)\|_{L^2}^2
		=: 2\delta - J_R(t).
\end{align*}
Integrating this inequality on $[T_j,T_{j+2}]$, we have
\begin{align}
	I_R^1(T_{j+2}) - I_R^1(T_j)
		\geq 2\delta(T_{j+2} - T_j) - \int_{T_j}^{T_{j+2}}J_R(t)dt
		\geq 2\delta - \int_{T_j}^{T_{j+2}}J_R(t)dt \label{139}
\end{align}
for any $j \in \mathbb{N} \cup \{0\}$.
We take a function $\chi(t,r) \in C^\infty(\mathbb{R} \times \mathbb{R})$ satisfying $\chi(t,r) = 1$ on $\{(t,r) : |t| \leq 2\text{ and }|r| \geq 1\}$ and $\chi(t,r) = 0$ on $\{(t,r) : |t| \geq 4\text{ or }|r| \leq \frac{1}{2}\}$.
For $R > 1$ and $T > 4$, we set $v_k(t,r) = \chi(t-T,\frac{r}{2^kR})u(t,|r|)$.
Applying the following inequality
\begin{align*}
	\int_{\mathbb{R}^2}(|\partial_tv_k|^2 + |\partial_rv_k|^2 + |v_k|^2)drdt
		& = \int_{T-4}^{T+4}\int_{|r| \geq 2^{k-1}R}(|\partial_tv_k|^2 + |\partial_rv_k|^2 + |v_k|^2)drdt \\
		& \leq C\int_{T-4}^{T+4}\int_{2^{k-1}R}^\infty(|\partial_tu|^2 + |\partial_ru|^2 + |u|^2)drdt \\
		& \leq C(2^{k-1}R)^{1-d}\int_{T-4}^{T+4}\int_{|x| \geq 2^{k-1}R}(|\partial_tu|^2 + |\nabla u|^2 + |u|^2)dxdt \\
		& \leq 8C(2^{k-1}R)^{1-d}J,
\end{align*}
we have
\begin{align*}
	\int_{T-2}^{T+2}\int_{|x| \geq R}|u(t,x)|^{p+1}dxdt
		& = C\sum_{k=0}^\infty \int_{T-2}^{T+2}\int_{2^kR \leq r \leq 2^{k+1}R}|v_k(t,r)|^{p+1}r^{d-1}drdt \\
		& \leq C\sum_{k=0}^\infty (2^{k+1}R)^{d-1} \int_\mathbb{R}\int_\mathbb{R}|v_k(t,r)|^{p+1}drdt \\
		& \leq C\sum_{k=0}^\infty (2^{k+1}R)^{d-1} \left(\int_{\mathbb{R}^2}|\partial_tv_k|^2 + |\partial_rv_k|^2 + |v_k|^2drdt\right)^\frac{p+1}{2} \\
		& \leq C\sum_{k=0}^\infty (2^{k+1}R)^{d-1}\cdot\{8C(2^{k-1}R)^{1-d}J\}^\frac{p+1}{2} \\
		& = CJ^\frac{p+1}{2}R^{-\frac{(d-1)(p-1)}{2}}\sum_{k=0}^\infty 2^{-\frac{k(d-1)(p-1)}{2}}.
\end{align*}
There exists $R_0 > 0$ such that
\begin{align*}
	\int_{T_j}^{T_j + 2}J_{R_0}(t)dt
		< \delta
\end{align*}
for any $j \geq 4$.
Combining this inequality and \eqref{139}, we have $I_{R_0}^1(T_{j+2}) - I_{R_0}^1(T_j) \geq \delta$ for any $j \geq 4$.
This inequality deduces that $\lim_{j \rightarrow \infty}I_{R_0}^1(T_j) = \infty$.
On the other hand, we have $I_{R_0}^1(T_j) \leq CR_0J$ for any $j \geq 1$.
This is contradiction.
\end{proof}

\subsection*{Acknowledgements}
The authors would like to express deep appreciation to Professor Noriyoshi Fukaya for teaching about the detail of proof for Proposition \ref{Regularity of the ground state}.
M.H. is supported by JSPS KAKENHI Grant Number JP19J13300.
M.I. is supported by JSPS KAKENHI Grant Number JP18H01132, JP19K14581, and JST CREST Grant Number JPMJCR1913.

\end{document}